\documentclass[oneside,english]{amsart}
\usepackage[T1]{fontenc}
\usepackage{color}
\usepackage{babel}
\usepackage{amstext}
\usepackage{amsthm}
\usepackage{amsmath}
\usepackage{amssymb}
\usepackage{bbold}
\usepackage{graphicx}
\usepackage[show]{ed}

\usepackage[a4paper]{geometry}

\usepackage[unicode=true,pdfusetitle,
 bookmarks=true,bookmarksnumbered=false,bookmarksopen=false,
 breaklinks=false,pdfborder={0 0 1},backref=false,colorlinks=true]
 {hyperref}
 \usepackage{mathrsfs}
 
 \usepackage{bm}

\definecolor{myblue}{rgb}{0,0,0.6}

\hypersetup{
 linkcolor=red,urlcolor=red,citecolor=red}

\makeatletter

\AtBeginDocument{\providecommand\lemref[1]{\ref{lem:#1}}}
\AtBeginDocument{\providecommand\thmref[1]{\ref{thm:#1}}}

\numberwithin{equation}{section}
\numberwithin{figure}{section}

\theoremstyle{plain}
\newtheorem{thm}{\protect\theoremname}[section]

\theoremstyle{definition}
\newtheorem{defn}[thm]{\protect\definitionname}

\theoremstyle{plain}
\newtheorem{lem}[thm]{\protect\lemmaname}

\theoremstyle{plain}

\theoremstyle{plain}
\newtheorem{prop}[thm]{\protect\propositionname}

\theoremstyle{plain}
\newtheorem*{claim*}{\protect\claimname}

\theoremstyle{remark}

\newtheorem{remark}[thm]{\protect\remarkname}

\theoremstyle{plain}
\newtheorem{assum}[thm]{Assumption}

\makeatother

\providecommand{\definitionname}{Definition}
\providecommand{\lemmaname}{Lemma}
\providecommand{\propositionname}{Proposition}
\providecommand{\theoremname}{Theorem}
\providecommand{\remarkname}{Remark}
\providecommand{\claimname}{Claim}
\providecommand{\corname}{Corollary}

\newcommand\R{\mathbb{R}}
\newcommand\N{\mathbb{N}}

\newcommand\bna{\begin{align*}}
\newcommand\ena{\end{align*}}

\newcommand\bnan{\begin{align}}
\newcommand\enan{\end{align}}

\newcommand\bneq{\begin{eqnarray*}\left\lbrace \begin{array}{rcl}}
\newcommand\eneq{\end{array} \right.\end{eqnarray*}}
\newcommand\bneqn{\begin{eqnarray}\left\lbrace \begin{array}{rcl}}
\newcommand\eneqn{\end{array} \right.\end{eqnarray}}

\newcommand\nor[2]{\left\|#1\right\|_{#2}}

\newcommand{\e}{\varepsilon}

\newcommand\tend[2]{\underset{#1 \to #2}{\longrightarrow}}

\newcommand\limsu[2]{\underset{#1\to #2}{\varlimsup}}
\newcommand\limvar[2]{\underset{#1\to #2}{\lim}}
\newcommand{\SO}{ S_{\Omega}}
\newcommand{\NLSO}{{\mathscr S}_{\Omega}}
\newcommand{\NLSF}{\mathscr S_{\mathbb R ^3}}

\newcommand{\EO}{\mathscr E_{\Omega}}
\newcommand{\SF}{S_{\mathbb R ^3}}

\newcommand{\Hu}{\mathcal{H}}
\newcommand\petito[1]{o(#1)}

\def\OO{{\mathcal O}}

\title[Nonlinear waves outside weakly trapping obstacles]{
On scattering and profile decomposition for critical nonlinear waves outside weakly trapping obstacles}

\author[D. Lafontaine]{David Lafontaine}
\address{CNRS and Institut de Mathématiques de Toulouse ; UMR5219, Université de Toulouse ; CNRS, UPS IMT, F-31062 Toulouse Cedex 9 (France)}
\email{david.lafontaine@math.univ-toulouse.fr}

\author[C. Laurent]{Camille Laurent}
\address{CNRS UMR 9008, Universit\'e Reims-Champagne-Ardennes, Laboratoire de Math\'ematiques de Reims (LMR), Moulin de la Housse-BP 1039, 51687 REIMS cedex 2, France}
\email{camille.laurent@univ-reims.fr}

\begin{document}

\maketitle
\begin{abstract}
    We prove scattering for the defocusing energy-critical non-linear wave equation with Dirichlet boundary conditions outside two strictly convex obstacles in dimension three.
    This is the first large data scattering result for such an equation in the presence of trapped trajectories.
    
    Our result is in fact more general and can be used as a black box in other geometries. 
    More precisely, under the assumptions that the corresponding linear wave equation satisfies global Strichartz estimates, that the domain is weakly non-trapping and that trajectories do not reconcentrate, we show linear and nonlinear profile decompositions in infinite time. This implies scattering under the rigidity assumption that the only compact-flow solution is the trivial one.
    
\end{abstract}

\setcounter{tocdepth}{1}
\tableofcontents

\section{Introduction}

We are interested in the defocusing energy-critical non-linear wave equation outside some obstacles in $\mathbb R^3$ with Dirichlet boundary condition
\begin{equation} \label{eq:NLWomega} \tag{NLW$_\Omega$}
\begin{cases}
\partial^2_t u - \Delta u + u^5 =0 \text{ in }\Omega,\\
u = 0 \text{ on }\partial \Omega, \\
\vec u_{\restriction t=0}= (\varphi^0, \varphi^1) \in \dot H_0 ^1 \times L^2(\Omega),
\end{cases}
\end{equation}
where $\Omega := \mathbb R^3 \backslash \Theta$ and $\Theta \subset \mathbb R^3$ is a  compact subset of $\mathbb R^3$ with smooth boundary, and $\vec u$ denotes $(u, \partial_t u)$. Solutions enjoy the conserved energy
$$
\EO(\vec u(t)) := \frac 12 \int_\Omega |\nabla u(t,x)|^2\,dx + \frac 12 \int_\Omega |\partial_t u(t,x)|^2\,dx + \frac 16 \int_\Omega |u(t,x)|^6\,dx,
$$
which yields a uniform bound of the norm of solutions in $\dot H^1 \times L^2$.

In the case of the free space $\Omega = \mathbb R^3$,
global existence of solutions was proved for smooth radial data by Struwe \cite{struwe88},
then extended to the non-radial case by Grillakis \cite{Grillakis90}, and 
global existence for data in the energy space $\dot H ^1 \times L^2$ was then obtained by Shatah and Struwe \cite{ShatahStruwe94}. Global well-posedness of solutions in the particular case of the exterior of a strictly convex obstacle was then proved by Smith and Sogge \cite{SS95}.
Global well-posedness of solutions with the equation posed \emph{inside} a domain was finally obtained by Burq, Lebeau, and Planchon \cite{BurqLebeauPlanchon08}, and their result extends immediately to the case we are interested in of the exterior of any obstacle by finite speed of propagation. 
 
 As solutions exist for all time, it is natural to wonder if they scatter: that is, if solutions behave linearly in large time. More precisely, we say that a solution to (\ref{eq:NLWomega}) \emph{scatters} in $\dot H ^1 \times L^2$ if there exists a solution $u_L$ to the \emph{linear} wave equation in $\Omega$ 
\begin{equation} \label{eq:LWomega} \tag{LW$_\Omega$}
\partial^2_t u_L - \Delta u_L =0 \text{ in }\Omega,\hspace{0.3cm} u_L = 0 \text{ on } \partial \Omega,
\end{equation}
 so that
 $$
 \Vert (u(t), \partial_t u(t)) - (u_L(t), \partial_t u_L(t))\Vert_{\dot H ^1 \times L^2} \rightarrow 0 \; \text{ as } t \rightarrow \pm \infty.
 $$ 
In a seminal work, Bahouri and Shatah \cite{BahouriShatah98} proved the decay of the nonlinear part of the energy $\int |u|^6 \rightarrow 0$ for solutions in the free space. Bahouri and G\'erard \cite{BahouriGerard99} then combined this decay estimate with suitable Strichartz estimates, such as described below in the paper, to show the scattering in $\mathbb R^3$.
 The decay of  the nonlinear part of the energy is obtained by integrating a divergence identity over a light-cone extending to infinity. Blair, Smith and Sogge \cite{MR2566711} remarked that, reproducing the same integration by parts outside an obstacle, a boundary term arises, which has the right sign when the obstacle is star-shaped, leading  to scattering in this case. In general, one can still obtain scattering provided the boundary term decays sufficiently in large time: this yields to the remark that, if there is no energy in mean in large time near the obstacle, then, there is scattering (see Lemma \ref{lem:contrenerg} below) -- such a property can be seen as a non-linear counterpart to the decay of the local energy for the linear equation, the study of which was pioneered by Morawetz, Raltson and Strauss \cite{Morawetz61, Morawetz75, Strauss75, MorawetzRalston77}. A natural conjecture is that such a decay of the local energy of solutions to (\ref{eq:NLWomega}), and hence the scattering, should hold at least for any non-trapping obstacle, for which all the rays of geometrical optics are going to infinity. However, such a result seems out of reach with the current techniques: as there is little hope to be able to use truly micro-local techniques for such a non-linear problem, a strategy in the spirit of \cite{GV85a} is to use vector-field multipliers computations, 
 inspired by the  works of Morawetz \cite{Morawetz75}, and relying on the choice of a multiplier adapted to the geometry of the obstacle.
 However, this method turns out to be very rigid in the nonlinear case, because the terms that were of lower order in the linear case now contribute a priori equally, and for these terms to have the right sign, a multiplier with a negative bilaplacian would be needed, which seems rather unatural (such a remark already goes back to \cite{Strauss75} in the linear case, before it was observed that theses terms are indeed of lower order in this case thanks to Lax theory).  For this reason, scattering of solutions to (\ref{eq:NLWomega}) had for now only been shown for star-shaped obstacles \cite{MR2566711} and generalisations  of this notion \cite{Farah1, Farah2}, \cite{DStrW}. We refer also to some recent results for small perturbation of the flat metric, \cite{LT:21, DL:26}.
 
 The situation is even worse in the \emph{trapping} case, i.e.~when some rays of geometrical optics stay in a compact for all time. Indeed, if $Z$ is a Morawetz vector-field (for example, the gradient of a multiplier such as described above), $(x, \xi) \rightarrow Z(x) \cdot \xi$ is an espace function, that is, a function of the phase-space that is strictly increasing along the flow of geometrical optics. Therefore, periodic orbits rule out the existence of such vector fields. As a consequence, there is no hope to be able to show the decay of the nonlinear local energy, and hence the scattering, by vector fields methods only, outside a trapping obstacle. One of the purposes of this paper is to show the first scattering result  in such a trapping situation. More precisely, we are interested in the canonical case of \emph{unstable} trapping for the problem with boundaries: the exterior of two strictly convex obstacles, for which there is only one, unstable, trajectory trapped both in the future and in the past. Recall that 
such an unstable trapping framework has attracted a lot of attention in the linear case since the works of Ikawa \cite{Ikawa1, Ikawa2} -- see in particular \cite{BurqDuke}, \cite{NZ1, NZ2}. We are able to deal with pairs of strictly convex obstacles $\Theta_1, \Theta_2$ satisfying the following symmetry assumption, which is for example verified in the case of two balls.
\begin{assum}\label{ass:normal}
The line which carries the trapped trajectory intersects $\partial\Theta_1$ and $\partial\Theta_2$ only normally.
\end{assum} 
Our first main result is the following.
 
 \begin{thm}
 \label{th:main}
Let $\Theta_1$ and $\Theta_2$ be two smooth strictly convex subsets of $\mathbb R ^3$ verifying Assumption \ref{ass:normal}, and $\Omega := \mathbb R^3 \backslash (\Theta_1 \cup \Theta_2)$. Then, any solution of \eqref{eq:NLWomega} scatters.
Moreover, for any $R_0>0$, there exists $C(R_0)>0$ so that for any $(\varphi^0, \varphi^1) \in \dot H_0 ^1 \times L^2 (\Omega)$ with $\EO (\varphi^0, \varphi^1)\leq R_0$, the solution $u$ of \eqref{eq:NLWomega} satisfies
\begin{align}
\label{unifStrich}
    \nor{u}{L^5(\R,L^{10}(\Omega)) }\leq C(R_0).
\end{align}
\end{thm}

Note that by \emph{strictly convex}, we mean having principal curvatures bounded below by a strictly positive constant.

Because of the periodic trajectory, a Morawetz multiplier cannot exist in the exterior of 
$\Theta_1 \cup \Theta_2$. However, following the progress \cite{D2pot} concerning a model case without boundary, we are able to exhibit an \emph{almost} Morawetz multiplier, that is, a multiplier which has the right behaviour \emph{everywhere but in an arbitrarily small neighborhood of the trapped trajectory}. 
This remark alone is not enough to obtain scattering, but it permits to rule out the existence of (hypothetical) non-trivial non-scattering solutions with a compact flow in $\dot H^1 \times L^2$ (see Theorem \ref{thm:rig}): such solutions cannot concentrate their energy in a set of arbitrarily small measure, hence in particular in the neighborhood of the trapped ray, and therefore must scatter by our almost-multiplier argument.  
This argument will be nothing but the \emph{rigidity} step in a concentration-compactness/rigidity scheme, first introduced by Kenig and Merle \cite{KeMe06, KeMe08} and relying on profile decompositions that originate from \cite{BahouriGerard99} in their modern form, to show the scattering. Arguing by contradiction, one constructs in a \emph{concentration-compactness} step a non-trivial, non-scattering compact-flow solution, the existence of which is ruled out in the rigidity step. 

Whereas this later rigidity step of our proof relies heavily on the precise geometry we are interested in, we are able to show profile decompositions and carry the concentration-compactness step in a quite general framework of \emph{weak trapping}. This constitutes the second main result of this paper, stated as Theorem \ref{th:const_intro} below, and we now state the precise assumptions under which we are able to do so. In the following, $t \mapsto \gamma_t(x, \xi)$ denotes the generalized geodesic starting at $x\in \overline \Omega$ in direction $\xi \in \mathbb S^2$ and parametrized by the time $t\in \mathbb R$. More precisely, it is defined in the following way: we denote $\varphi_t$ the generalized bicharacteristic flow for $-\Delta$ in $\overline{\Omega}$, defined on the compressed co-tangent bundle $^b T^* 
\Omega$, $j : T^*\overline{\Omega} \to ^b T^* \Omega$ the canonical projection (we refer to \cite[Section 3.1]{B:97b} for the precise definitions; see Section \ref{sec:conc_prof} below for more details), and $\pi_x$ the projection on $\overline{\Omega}$. Then 
$\gamma_t(x,\xi):=\pi_x  \varphi_t j(x,\xi)$.

\begin{assum}[Non reconcentration] \label{ass:nonreco}
For all $x,x_{0}\in\overline{\Omega}$ and all $t\geq0$ we have
\[
\Big|\big\{ \xi\in \mathbb{S}^2\ \text{s.t.}\ \gamma_t(x,\xi)=x_{0} \big\}\Big| =0,
\]
where $|\cdot|$ denote the Lebesgues mesure  on $\mathbb{S}^2$.
\end{assum}
\begin{assum}[Weak trapping] \label{ass:weaktrap}
For any $x_{0}\in\overline{\Omega}$, there exists a non-increasing family of closed subsets $(V_n)_{n\in\N}$, $V_n \subset \mathbb{S}^2$, so that the following holds: $|V_n| \to 0$ as $n \to \infty$, and, for any $R>0$ and $n\in \N$, there exists a time $T_{R,n}<+\infty$ so that for all direction $\xi\notin V_{n}$ we have $\gamma_t(x_0,\xi)\in B(0,R)^c$ for all $|t|\geq T_{R,n}$.
\end{assum}

\begin{assum}[Strichartz estimates]  \label{ass:strichartz}
The linear wave equation in $\Omega$ verifies the same global-in-time Strichartz estimates as in $\mathbb R^3$, at least for the Strichartz-admissible pair $( 5, 10)$ and another Strichartz admissible-pair $(r, s)$ with $s>10$. In other words, there exists $(r, s)$ with $s>10$  veryfying 
\begin{align}
\label{admissibleStric}
   \frac 12 = \frac 1r + \frac 3s, \hspace{0.5cm} \frac 1r \leq \frac 12 - \frac 1s, \hspace{0.5cm} r \neq 2,
\end{align}
so that for any solution $u$ to the linear wave equation in $\Omega$ 
$$
\Vert u \Vert_{L^p(\mathbb R, L^{q}(\Omega))} \lesssim \Vert (u(0), \partial_t u(0)) \Vert_{\dot H^1 \times L^2} \; \text{ for } (p, q) \in \{ (5, 10), (r, s) \}.
$$
\end{assum}

Let us comment on these assumptions. Assumption \ref{ass:nonreco} states that in any point of the domain, the set of directions that allow the trajectories of geometrical optics to re-concentrate later in time is of finite measure: the typical examples of geometries ruled out by this assumption are obstacles the boundary of which contain a (concave) portion of a sphere. This assumption is verified in the exterior of two strictly convex obstacles in Lemma \ref{lem:rays_noreco}. While, to the difference of Assumptions \ref{ass:weaktrap} and \ref{ass:strichartz} 
(that are in particular always true outside any non-trapping obstacle as we will see), this assumption might be only technical, and we believe that it is verified generically.
Assumption \ref{ass:weaktrap} expresses that in any point, all directions but a set of arbitrarily small measure allow to escape at infinity:
this reflects the fact that the obstacle is weakly trapping. It is proved in the exterior of two strictly convex obstacles in Lemma \ref{lem:finite_trap}. Remark that it is an assumption on the \emph{size} (at any point of the physical space) of the trapped set (and its tails, i.e.~trajectories trapped only in the future or only in the past), not on its stability. The \emph{stability} of the trapped set in turns plays a role in Assumption \ref{ass:strichartz}, that concerns the linear flow in the domain. This assumption states that the trapping, and finite-time concentration effects, are sufficiently  weak to preserve the free Strichartz estimates (obtained in $\mathbb R^3$ by \cite{Str77}, \cite{GV85b}, \cite{LS}, and \cite{KT}) for at least two well chosen couple of Strichartz exponents. Such estimates were obtained (with the full range of indices) outside one convex obstacle by \cite{SS95} (see also \cite{Oana10} for the Schr\"odinger equation). In the exterior of two strictly convex obstacles, this is the main result of \cite{DStrW} (see also \cite{DStrS} for the Schr\"odinger equation, \cite{DStrM} for finitely many strict convex verifying the Ikawa condition, and \cite{BGH} for the analogous boundaryless case). Remark that \cite{BurqLtG} proved that under the non-trapping assumption, the local-energy decay permits to glue finite-time Strichartz estimates to obtain global ones. In the case of \cite{DStrW}, and following \cite{BGH}, the loss due to the trapped trajectory in the local-energy decay (when asking for a rate with respect to the energy of the data)  is compensated by proving Strichartz estimates in time $\sim | \log h |$ for a data of frequency $\sim h^{-1}$. Finally, let us mention that when a portion of the boundary is concave, there is a loss in the range of admissible Strichartz pairs with respect to the free case \cite{Oana12}.
However, \cite{BurqLebeauPlanchon08} and \cite{MR2566711} showed that (finite time) Strichartz estimates with Strichartz pairs $(5, 10)$ and some $(r, s)$ with $r<5$ (for example $(\frac 72, 14)$) always hold: combined with the result of \cite{BurqLtG}, this shows that Assumption \ref{ass:strichartz} in particular always holds outside a non-trapping obstacle.

We are able to carry the concentration-compactness step of our proof of scattering under the previous general assumptions, which constitutes the second main result of this paper:
\begin{thm}[Concentration-compactness under weak trapping]\label{th:const_intro}
Assume that $\Omega$ has a smooth, compact boundary and verifies Assumptions \ref{ass:nonreco}, \ref{ass:weaktrap}, \ref{ass:strichartz}. If (\ref{eq:NLWomega})
has no non-trivial solution with a relatively compact flow $\big\{ (u(t), \partial_t u(t)), \; t \in \mathbb R\big\}$ in $\dot H_0^1 \times L^2 (\Omega)$, then
any solution of  (\ref{eq:NLWomega}) scatters and satisfies the same uniform Strichartz estimates as \eqref{unifStrich}.
\end{thm}
As a consequence, the conjecture of scattering in any non-trapping geometry stated in the beginning of our introduction now reduces (at least for non reconcentrating obstacles, i.e.\ verifying Assumption \ref{ass:nonreco}) to showing the rigidity property that the equation has no non-trivial compact flow solution. In addition, 
note that Theorem \ref{th:const_intro}  applies in particular to the exterior of a star-shaped obstacle satisfying Assumption \ref{ass:nonreco}  where the scattering was already proved, see \cite{MR2566711} using arguments of \cite{BahouriShatah98}, but it does not seem that the uniform Strichartz estimate \eqref{unifStrich} was known.

At the heart of the proof of Theorem \ref{th:const_intro} lies a \emph{linear profile decomposition} for the linear wave flow in $\Omega$, in the spirit of the one introduced by \cite{BahouriGerard99} in the free space, which we are able to prove under the same Assumptions \ref{ass:nonreco}, \ref{ass:weaktrap}, \ref{ass:strichartz} and is stated as Theorem \ref{th:lindec} below. Observe that \cite{GG} introduced such a profile decomposition in finite time for the wave equation outside a strictly convex obstacle. Our decomposition given in Theorem {\ref{th:lindec}} is more general than their result, in three directions. First, it is valid under the general weak trapping assumption discussed above, wereas \cite{GG} is only valid outside one convex obstacle. Second, it is proved in infinite time. Finally, we don't make the ``compactness at infininity'' assumption used by \cite{GG}, and hence our analysis involves more regimes of parameters. 

The proof of the linear profile decomposition relies on the asymptotic description of the profiles in all such possible regimes. We show that, when the center of a profile is going to infinity, {or} its scale parameter goes to infinity, it behaves asymptotically as a free profile, that is, a profile for the flow in $\mathbb R^n$ (see Lemma \ref{lem:assfree}) and the analog is true for corresponding non-linear profiles (Lemma \ref{lem:assfree_nl}). The most delicate case is when a profile concentrates (that is, when its scale parameter goes to zero) while remaining in a bounded region of space. In this case, the typical frequency of the solution tends to $+\infty$ and it is expected that, at least in finite time, the approximation by geometric optic is valid. As discussed before, the point is to detect where there can be some points of reconcentration which, as noticed in G\'erard \cite{G:96} following the concentration-compactness principle of Lions \cite{L:85}, can induce that the Strichartz norm becomes large. This explains our Assumption \ref{ass:nonreco} which ensures non-reconcentration for the ``classical flow''. For the linear solutions of the wave equation, this program was fulfilled by \cite{BahouriGerard99} in the free space where no reconcentration can occur. This turns out to be more complicated in non flat geometry. Such description was made in Gallagher-G\'erard \cite{GG} in the case of the exterior of a convex obstacle where no reconcentration can occur. In a compact manifold, without further assumption, the linear profile decomposition was achieved by the second author in \cite{L:11} using results of Ibrahim \cite{I:04}. Note that the phenomenum of reconcentration is described in the specific case of the sphere $\mathbb{S}^{3}$ where a data concentrating at the south pole can reconcentrate at the north pole.

Once the linear profile decomposition is at hand, we can follow the Kenig-Merle scheme to prove Theorem \ref{th:const_intro}, by showing an analogous non-linear profile decomposition, that involves now the asymptotic description of the corresponding non-linear profiles. We use in particular the known result of scattering in the free space \cite{BahouriGerard99} to deal with the non-compact non-linear profiles. 

Finally, in the exterior of two convex obstacles, the combination of Theorem \ref{th:const_intro} and of the rigidity argument discussed above and carried out as Theorem \ref{thm:rig} gives Theorem \ref{th:main}.

We end this introduction by giving a few more references related to the problem. When a Neumann boundary condition is imposed in (\ref{eq:NLWomega}), much less is known -- this is due to the fact that the boundary term which arises in Morawetz-type computations is not signed even in the simplest case of the exterior of a ball. Scattering for such a problem has been shown in the radial case in \cite{DL}, following a concentration-compactness/rigidity approach, where the rigidity part is obtained thanks to the channels of energy method (introduced in \cite{DuKeMe11a}, \cite{DuKeMe13}).
Finally, note that similar questions also appeared for Schr\"odinger like equations where the global geometry is important even in small times due to the infinite time of propagation. Some profile decompositions were studied in specific geometries, for instance in \cite{IP:12} for $\mathbb{T}^{3}$, by \cite{IPS:12} in the hyperbolic space, in \cite{PTW:14} on $\mathbb{S}^{3}$, in \cite{KVZ:16} in the exterior of a convex obstacle, and in \cite{J:19} for a non-trapping compact metric perturbation of $\R^{3}$.

\subsection*{Structure of the paper}

As preliminaries, Section \ref{sec:preli} introduces the notations used throughout the paper, states perturbative results, and discuss the vocabulary used to describe profiles.
Section  \ref{sec:ass_prof} describes the profiles with center or scale parameter going to infinity.
Section  \ref{sec:conc_prof} deals with the concentrating profiles with a localised center.  
Section \ref{sec:lin_prof} states and proves the linear profile decomposition (Theorem \ref{th:lindec}).  Section \ref{sec:cc} proves Theorem \ref{th:const_intro}. Section \ref{sec:rigi} proves the rigidity part of the argument in the exterior of two strictly convex obstacles (Theorem \ref{thm:rig}). Section \ref{sec:verifA} verifies that the assumptions of Theorem \ref{th:const_intro} are satisfied in this case.

\section{Preliminaries} \label{sec:preli}

\subsection{Notations}

\subsubsection*{Spaces of functions}
\begin{itemize}

\item $L^{p}L^{q}:=L^{p}(\mathbb{R},L^{q}(\Omega))$.
\item For an arbitrary domain $X$, $\mathcal{H}( X ) := \dot{H}^1_0(X)\times L^{2}( X )$. 
{Here $\dot{H}^1_0(X)$ is the completion of $C^{\infty}_0(X)$ for the norm $\nor{f}{X}=\nor{\nabla f}{L^2(X)}$.}
\item $\mathcal H (\Omega)$ will often be denoted $\mathcal H$.
\end{itemize}

\subsubsection*{Time derivatives}
\begin{itemize}

\item If $u$ is a function of time and space, $\vec u$ is understood to be $(u,\partial_t u)$.
\item {Conversely, if $\vec u \in \mathcal H(\Omega)$, $u$ is understood to be the first component of $\vec u$.}

\end{itemize}

\subsubsection*{Linear and nonlinear flows}
\begin{itemize}
\item $\SF$ and $\SO$ are flows of the linear wave
equation respectively in $\mathbb{R}^{3}$ and in $\Omega$
with Dirichlet boundary condition. If $(u_0,u_1)$ is the initial data,
we will denote { by} $\SF(t)(u_0,u_1)$ or $( \SF(u_0,u_1)\big)(t)$,   
and $\SO(t)(u_0,u_1)$ or $( \SO(u_0,u_1)\big)(t)$ the corresponding solutions evaluated at time $t$. 

\item $\mathscr{S}_{\mathbb{R}^{3}}$ and $\mathscr{S}_{N}$ are the corresponding
nonlinear flows for the defocusing energy critical wave equation in $\mathbb R^3$
and $\Omega$ with Dirichlet boundary condition respectively.

\item We use similar notations for $\SO$, $\SF$, and the nonlinear flows $\NLSO$ and $\NLSF$. 
The arrowed versions $\vec{ \SF }$
and $\vec{ \SO }$ denote the flows together with their
first time derivative.

\end{itemize}

\subsubsection*{Projection and extension by zero}

\begin{itemize}

\item We denote by 
$\mathcal P_\Omega$ the orthogonal projection from $\dot H^1(\mathbb R^3)$ onto
$\dot H_0^1(\Omega)$. 
\item Abusing slightly notations, if $f\in \dot H^1(\mathbb R^3)$, we denote also  by $\mathcal P_\Omega f$ the extension by zero of $\mathcal P_\Omega f$ to
$\dot H^1(\mathbb R^3)$. 
\item Finally, we denote $P_\Omega := (\mathcal P_\Omega, \mathbb{1}_\Omega)$.
\end{itemize}

\subsubsection*{Cores $\mathcal O$, profiles $\vec \varphi _{\Omega, \mathcal{O},k}$ , scaled domain $\Omega_n$ and limit domain $X_{\mathcal O}$} The notations corresponding to these notions can be found in the dedicated \S \ref{subsec:profnot}.

\subsection{Perturbative theory}
We first notice that Assumption \ref{ass:strichartz} actually implies more estimates.
\begin{lem} \label{lem:L1L2_}
 If Assumption \ref{ass:strichartz} holds, then there is $C>0$ so that for any $T>0$ we have the estimates 
 $$
 \Vert u \Vert_{L^p([0,T], L^q(\Omega))} \leq C\big(\Vert (u(0), \partial_t u(0))\Vert_{\mathcal H(\Omega)}+ \nor{f}{L^1([0,T],L^2(\Omega))} \big)$$
 for any $(p,q)$ with $r\leq p\leq +\infty$ satisfying \eqref{admissibleStric} and $u$ solution of $\Box u=f$ together with the Dirichlet boundary condition.
\end{lem}
\begin{proof}
As is classical, this is a simple consequence of Minkowski inequality: indeed, by Assumption \ref{ass:strichartz} and Duhamel formula
    $$
        \Vert u \Vert_{L^p([0,T], L^q(\Omega))}
        \leq \Vert (u(0), \partial_t u(0))\Vert_{\mathcal H(\Omega)}+ \Big\Vert \int_0^t \SO(t-s)(0, f(s)) \, ds \Big\Vert_{L_t^p([0,T], L^q(\Omega))},
    $$
    where, by Minkowski inequality and Assumption \ref{ass:strichartz}
    \begin{align*}
    &\Big\Vert \int_0^t \SO(t-s)(0, f(s)) \, ds \Big\Vert_{L_t^p([0,T], L^q(\Omega))}= 
    \Big\Vert \int_0^T \mathbf 1_{s\in[0,t]}\SO(t-s)(0, f(s)) \, ds \Big\Vert_{L_t^p([0,T], L^q(\Omega))}  \\
    &\qquad\leq \int_0^T \Vert \mathbf 1_{s\in[0,t]}\SO(t-s)(0, f(s)) \Vert_{L_t^p([0,T], L^q(\Omega))} \, ds \\
    &\qquad\leq \int_0^T \Vert \SO(t-s)(0, f(s)) \Vert_{L_t^p (\mathbb R, L^q(\Omega))} \, ds \lesssim \int_0^T \Vert f(s) \Vert_{L^2} \, ds. 
    \end{align*}
\end{proof}

\begin{defn}
 \label{def:scattering}
We say that a solution $u$ of the nonlinear wave equation (\ref{eq:NLWomega})  \emph{scatters in the future} when there exists a solution $u_L$ of the corresponding linear wave equation  such that 
$$\lim_{t\to+\infty}\left\|\vec{u}(t)-\vec{u}_L(t)\right\|_{\mathcal H(\Omega)}=0.$$
We define similarly \emph{scattering in the past}. We say that the solution \emph{scatters} when it scatters both in the future and in the past.
\end{defn}

The next two Propositions are consequences of Strichartz estimates in a classical way.

\begin{prop}
\label{prop:perturb_scat}Let $(u_0,u_1)\in\mathcal{H}(\Omega)$ and $u(t)=\NLSO(t)(u_0,u_1)$.
\begin{equation}
u\in L^{5}\left([0,+\infty),L^{10}\right)\implies u \text{ scatters in the future}.\label{eq:pert_scat_carac}
\end{equation}
A similar property holds in the past.
Moreover, there exists $\epsilon_{0}>0$ such that, for any $(u_0,u_1)\in\mathcal{H}(\Omega)$,
\begin{equation}
\Vert(u_0,u_1)\Vert_{\mathcal{H}(\Omega)}\leq\epsilon_{0}\implies \NLSO(\cdot)(u_0,u_1)\in L^{5}(\R,L^{10}(\Omega),\label{eq:pert_scat_small}
\end{equation}
together with the estimate
\begin{equation}
\label{e:strigchglob}
    \nor{\NLSO(\cdot)(u_0,u_1)}{L^{5}(\R,L^{10}(\Omega))}\leq C \Vert(u_0,u_1)\Vert_{\mathcal{H}(\Omega)}
\end{equation}
and $\NLSO(\cdot)(u_0,u_1)$ scatters.
 Finally, for any $(u_0,u_1)\in\mathcal{H}(\Omega)$,
 there exists a solution $U^{\pm} \in L^5(\mathbb R_{\pm}, L^{10})$ of (\ref{eq:NLWomega}) such that
 \begin{equation}
 \Vert \vec{U}^{\pm}(t)-\vec \SO(t)(u_0,u_1)\Vert_{\mathcal{H}(\Omega)}\longrightarrow0,\text{ as }t\longrightarrow\pm\infty.\label{eq:pert_scat_Uj}
 \end{equation}
\end{prop}

\begin{proof}[{Sketch of proof}]
The properties (\ref{eq:pert_scat_carac}) and (\ref{eq:pert_scat_small})
are classical consequences of the global in time Strichartz estimates,
and (\ref{eq:pert_scat_Uj}) can be proved by a fixed point argument using the Strichartz estimates as well.
\end{proof}

\begin{prop}[Perturbation] 
\label{lem:perturb}For any $M>0$, there exists $\epsilon(M)>0$
such that, for any $0<\epsilon\leq\epsilon(M)$, and all $(u_0, u_1), \,(\tilde u_0, \tilde u_1)\in\mathcal{H}(\Omega)$,
$e\in L^{1}L^{2}$ and $u\in L^{5}L^{10}$ verifying 
\[
\Vert u\Vert_{L^{5}L^{10}}\leq M,\hspace{1em}\Vert \SO(\cdot)\big((u_0, u_1)-(\tilde u_0, \tilde u_1)\big)\Vert_{L^5L^{10}}\leq\epsilon,\hspace{1em}\Vert e\Vert_{L^{1}L^{2}}\leq\epsilon,\hspace{1em}\Vert \tilde e\Vert_{L^{1}L^{2}}\leq\epsilon,
\]
if $u,\tilde u$ are solutions of 
\begin{equation*}
\begin{tabular}{cc}
$
\left\{
\begin{aligned}
\partial_{t}^{2}u-\Delta u&=-u^{5}+e\text{ in }\Omega, \\
\vec u_{\restriction t=0}&=(u_0,u_1),\\
u&=0\text{ on }\partial \Omega,
\end{aligned} \right.
$
&
$
\left\{
\begin{aligned}
\partial_{t}^{2}\tilde u-\Delta \tilde u&=-\tilde u^{5} + \tilde e\text{ in }\Omega, \\
\vec {\tilde u}_{\restriction t=0}&=(\tilde u_0,\tilde u_1),\\
\tilde u&=0\text{ on }\partial \Omega,
\end{aligned}\right.
$
\end{tabular}
\end{equation*}
then $\tilde u\in L^{5}L^{10}$ and we have 
\[
\Vert u-\tilde u\Vert_{L^{5}L^{10}}\lesssim\epsilon.
\]
In addition, the same statement holds for the corresponding equations in $\mathbb{R}^{3}$.
\end{prop}

\begin{proof}[{References for proof}]
The proof is classical and similar to Proposition 4.7 of \cite{MR2838120}. It can be found as a sole consequence of the Strichartz estimates in \cite[Proposition 2.13]{DL}  with the slight addition that we added a small source term to both equations.
\end{proof}

Finally, we will need the following \emph{linear} scattering result
\begin{lem}[Linear scattering {\cite[Chapter 9, Proposition 5.5 p.230]{Taylor2}}] \label{lem:lin_scat}
Let $\vec \varphi \in \mathcal H(\Omega)$. Then, there exist $\vec \varphi_\pm^\prime \in \mathcal H(\mathbb R^3)$ so that
$$
\Vert \SO(t) \vec \varphi - \SF(t) \vec \varphi_\pm^\prime \Vert_{\mathcal H(\Omega)} \rightarrow 0 \text{ as }t\rightarrow \pm \infty.
$$
\end{lem}

\subsection{Description of profiles} \label{subsec:profnot}
The description of the (loss of) compactness for sequences of solutions requires the introduction of linear and nonlinear {\it profiles} associated to different {\it scale-cores}. 

\textbf{Scale cores $\mathcal O$.} A {\it scale-core} written in general $\mathcal{O}$ will be a sequence $\mathcal{O}_n=(\lambda_n,t_n,x_n)$ with $\lambda_{n}>0$, $t_{k}\in \R$ and $x_{n}\in \R^{3}$.
\begin{def}
\label{defprofils}
Let $\mathcal{O}=\{(\lambda_k,t_k,x_k)\}$ be a scale core and let $\vec \varphi = (\varphi^0, \varphi^1 )\in \mathcal H  (\mathbb R^3)$. 

\textbf{Profiles $\vec \varphi _{\Omega, \mathcal{O},k}$.} We define the associated profile in $\Omega$, $\{\vec \varphi^\Omega _{\mathcal{O},k}\}$ as the sequence given by
\begin{equation}\label{EProf}
\vec \varphi_{\Omega, \mathcal{\mathcal{O}},k}:=\SO(-t_k) T^\Omega_{\mathcal{O},k}\vec \varphi ,\quad 
\end{equation}
where 
\begin{align*}
T_{\mathcal{O},k}\vec \varphi (x)&:=\left(\frac{1}{\lambda_k^{1/2}}\varphi^{0} \left(\frac{x-x_k}{\lambda_k}\right),\frac{1}{\lambda_k^{3/2}}\varphi^{1} \left(\frac{x-x_k}{\lambda_k}\right)\right),\\
T_{\mathcal{O},k}^{\Omega}\vec \varphi&:=P_{\Omega}T_{\mathcal{O},k}\vec \varphi.
\end{align*}
and we recall that $P_{\Omega}=(\mathcal{P}_{\Omega},1_{\Omega})$ is the orthogonal projection from $\dot{H}^{1}(\R^{3})\times L^{2}(\R^{3})$ into $\dot{H}^{1}_{0}(\Omega)\times L^{2}({\Omega})$.

Note that the operators $T_{\mathcal{O},k}$ and $T_{\mathcal{O},k}^{\Omega}$ do not depend on the component $t_k$ of $\mathcal{O}$. 

\textbf{Orthogonality of cores.} We say that two frames $\mathcal{O}=\{(\lambda_k , t_k , x_k )\}$ and $\widetilde{\mathcal{O}}=\{(\widetilde{\lambda}_k , \widetilde{t}_k , \widetilde{x}_k  )\}$ are \emph{orthogonal} if
$$
\lim_{k\to +\infty }\ln \left(\frac{\lambda_k}{\widetilde{\lambda}_k} + \lambda_k^{-1}|t_k-\widetilde{t}_k | + \lambda_k^{-1}|x_k-\widetilde{x}_k | 
\right)= +\infty.
$$
Two frames that are not orthogonal are called equivalent.
\end{def}

\textbf{Conjugate core $\mathcal O^{-1}$.}
We also define $\mathcal{O}^{-1}=\{(\lambda_k^{-1},-t_k\lambda_k^{-1},-x_k\lambda_k^{-1})\}$, so that 
$$
T_{\mathcal{O}^{-1},k}\circ T_{\mathcal{O},k}=\operatorname{I},\hspace{0.3cm}T_{\mathcal{O},k}^*=T_{\mathcal{O}^{-1},k}$$
for the duality of $\dot{H}^1(\R^3)\times L^2(\R^3)$.

\textbf{Scaled domain $\Omega_n$ and limit domain $X_{\mathcal O}$.}
Given a frame $\mathcal{O}=\{(\lambda_k , t_k , x_k )\}$, denote $\Omega_k:=\frac{\Omega-x_k}{\lambda_k}$. Note that $T_{\OO,k} $ is an isometry from $\mathcal{H}(\Omega_{k})$ to $\mathcal{H}(\Omega)$. We will use several time in the paper that
\begin{equation}
\label{eq:conj_proj}
T_{\mathcal{O}^{-1},n} P_{\Omega}T_{\mathcal{O},n}= P_{\Omega_n}.
\end{equation}

\begin{def} \label{defconvergespace} Following \cite{GG}, we say that $\Omega_k$ converges to $X_\mathcal{O}$ if 
\begin{itemize}
\item $\forall K$ compact subset of $X_{\OO}$, $\exists N\in \N$ so that $\forall k\geq N$, $K\subset \Omega_{k}$,
\item $\forall K'$ compact subset of $\overline{X_\mathcal{O}}^c$, $\exists N'\in \N$ so that $\forall k\geq N'$, $K'\subset \overline{\Omega_{k}}^c$.
\end{itemize}
\end{def}
The limit space $X_{\OO}$ of a scale core $\mathcal{O}$, when it exists (which can always be assumed up to a subsequence) will be denoted $X_\mathcal{O}$ since it is unique up to subsequence. It is quite clear that the asymptotics of $\mathcal{O}$ will describe very different behaviors, and we will have to distinguish cases :
\begin{itemize}
\item The case $\lambda_k\to 0$ is considered in \cite{GG}. If $\lambda_{k}\rightarrow 0$ and $x_{k}\rightarrow x_{\infty}\in \Omega$, then we have $\Omega_{k }\rightarrow X_{\OO}=\R^{3}$. If $\lambda_k \rightarrow 0$ and $x_{k}\rightarrow x_{\infty}\in \partial \Omega$, we assume that $\Omega$ is defined locally by $\Omega=\left\{x\in\R^{3};\phi(x)>0\right\} $ with $\nabla\phi(x_{\infty})=\vec n(x_{\infty})\neq 0$. We define $\alpha=\lim \frac{\phi(x_{k)}}{h_{k}}\in\overline{\R}$ and obtain $\Omega_{k }\rightarrow X_{\OO}=\left\{x\in\R^{3};x\cdot \vec n(x_{\infty}) > -\alpha\right\}$.
\item If $\lambda_n\to +\infty $ or $|x_n|\to +\infty$, we define $X_{\OO}:=\R^{3}$. \item If $\lambda_n\to \lambda_0 \in \R^*_+$ and $x_n\to x_{0}$, we define $X_{\OO}:=\frac{\Omega-x_{0}}{\lambda_0}$.
\end{itemize}
\section{Dilating profiles and profiles going to infinity} \label{sec:ass_prof}

The goal of this section is to show that dilating profiles and profiles going to infinity are asymptotically free, in the following sense.

\begin{lem} \label{lem:assfree}
Let $\vec \varphi = (\varphi^0, \varphi^1 )\in \dot H^1 \times L^2 (\mathbb R^3)$, $f\in L^1(\mathbb R, L^2(\mathbb R^3))$, $(t_n)_{n\geq1}$ be an arbirary sequence of times, and $(\lambda_n)_{n\geq 1}$, $(x_n)_{n\geq 1}$ 
be so that 
$$
\lambda_n \rightarrow \infty \hspace{0.7cm}\text{or}\hspace{0.7cm} |x_n|\rightarrow \infty.
$$ 
We define 
$$
f_n := \lambda_n^{-5/2}f\left(\frac{\cdot  - t_n}{\lambda_n}, \frac{\cdot-x_{n}}{\lambda_n}\right),
$$
and $v_n$ to be the solution of
$$
\begin{cases}
\partial_t^2 v_n - \Delta v_n = f_n \text{ in }\mathbb R^3, \\
(v_n, \partial_t v_n) (t=t_n) = \Big(\lambda_n^{-1/2} \varphi^0(\frac{\cdot - x_n}{\lambda_n}),  \lambda_n^{-3/2} \varphi^1(\frac{\cdot-x_n}{\lambda_n})\Big).
\end{cases}
$$
In addition, let
 $u_n$ be the solution of
$$
\begin{cases}
\partial_t^2 u_n - \Delta u_n =\mathbb{1}_{\Omega} f_n \text{ in } \Omega, \\
u_n = 0 \text{ on } \partial \Omega \\
(u_n, \partial_t u_n) (t= t_n) = (\mathcal P_\Omega v_n, \mathbb{1}_\Omega \partial_t v_n) (t=t_n). 
\end{cases}
$$
Then, as $n\rightarrow\infty$,
\begin{equation}
\label{znNRJ}
\sup_{t\in\mathbb R} \int_\Omega |\nabla(u_n - v_n)|^2 + |\partial_t(u_n - v_n)|^2 \to 0.
\end{equation}
In addition, 
\begin{equation}
\label{strichartzzn}
\Vert u_n - v_n \Vert_{L^{\infty}(\mathbb R, L^{6}(\Omega))}+\Vert u_n - v_n \Vert_{L^5(\mathbb R, L^{10}(\Omega))} \to 0.
\end{equation}
\end{lem}

Lemma \ref{lem:assfree} is proved in subsections \S \ref{subsec:to_infinity} and \S  \ref{subsec:dilating} below, specializing respectively to the cases of a profile going to infinity ($|x_n| \rightarrow \infty$ with $\lambda_n$ bounded), and a dilating profile ($\lambda_n \rightarrow \infty$). Subsection \S \ref{subsec:assfree_reduc} is dedicated to the reduction of both cases to the proof of (\ref{znNRJ}) for $t_n = 0$ and eliminating the data projection.
 
As a consequence of Lemma \ref{lem:assfree}, we get an analog result for the associated nonlinear profiles:

\begin{lem} \label{lem:assfree_nl}
Let $\vec \varphi = (\varphi^0, \varphi^1 )\in \dot H^1 \times L^2 (\mathbb R^3)$, $(t_n)_{n\geq1}$ be an arbirary sequence of times, and $(\lambda_n)_{n\geq 1}$,  $(x_n)_{n\geq 1}$  be so that
$$
\lambda_n \rightarrow \infty \hspace{0.7cm}\text{or}\hspace{0.7cm} |x_n|\rightarrow \infty.
$$ 
Let $v \in L^5 L^{10}$ be solution of
$$
\partial^2_t v - \Delta v =- v^5 \text{ in }\mathbb R^3, 
$$
and
$$
v_n := \frac{1}{\lambda_n ^{1/2}}v\left(\frac{t-t_n}{\lambda_n}, \frac{x-x_n}{\lambda_n}\right).
$$
In addition, let $u_n$ be solution to
$$
\begin{cases}
\partial^2_t u_n - \Delta u_n =- u_n^5 \text{ in } \Omega, \\
u_n = 0 \text{ on }\partial\Omega, \\
(u_n, \partial_t u_n)(t_n) =  (\mathcal P_{\Omega}v_n, \mathbb{1}_\Omega \partial_t v_n)(t_n).
\end{cases}
$$
Then,
$$
\sup_n \Vert u_n \Vert_{L^5 L^{10}} < \infty
$$
and
$$
\sup_{t\in\mathbb R} \int_\Omega |\nabla(u_n - v_n)|^2 + |\partial_t(u_n - v_n)|^2 
+
\Vert u_n - v_n \Vert_{L^5(\mathbb R, L^{10}(\Omega))} \to 0.
$$

\end{lem}
\begin{proof}
Similarly to \cite[Lemma 3.3]{DL}, the Lemma follows from the linear case, here Lemma \ref{lem:assfree}, by a Gronwall argument that we reproduce  for the convenience of the reader.

Let $z_n$ be the solution of 
\begin{equation}
\label{e:znomega}
\left\{
\begin{aligned}
\partial_t^2z_n-\Delta z_n + \mathbb{1}_\Omega v_n^5&=0\quad\text{in }\Omega, \\
z_n&=0\quad \text{on }\partial\Omega,\\
\vec{z}_{n\restriction t= t_n}&=\vec{u}_{n\restriction t= t_n}.\end{aligned}\right.
\end{equation}
By Lemma \ref{lem:assfree} applied to $z_n$ and $v_n$ with $f:=-v^5$, we get
\begin{equation} \label{eq:dilnl_ZU}
\sup_{t\in \mathbb R} \Vert \vec z_n - \vec v_n\Vert_{\mathcal H(\Omega)} + \Vert z_n - v_n\Vert_{L^5 L^{10}(\Omega)} \to 0.
\end{equation}

We now show that
\begin{equation} \label{eq:nlporf_goal}
 \sup_n \Vert u_n \Vert_{L^5 L^{10}} < \infty \; \text{ and } \;\Vert u_n - v_n \Vert_{L^5(\mathbb R, L^{10})}  \to 0.
\end{equation}
As $L^5(\mathbb R, L^{10})$ is invariant by translation in time, in order to show the above, we can assume that $t_n = 0$ by replacing $u_n, v_n, z_n$ by $u_n(\cdot + t_n)$, $v_n(\cdot + t_n)$, $z_n(\cdot + t_n)$. 
Let now $T>0$; observe that
\begin{equation*}
\left\{
\begin{aligned}
\partial_t^2(z_n-u_n) + \Delta (z_n-u_n) &= -\left(u_n ^5 - 1_\Omega v_n ^5\right) \quad\text{in }\Omega, \\
(z_n-u_n)&=0\quad \text{on }\partial\Omega,\\
\vec{z}_{n} - \vec{u}_{n\restriction t= t_n}&=\vec{0},
\end{aligned}\right.
\end{equation*}
and therefore, we have, by Strichartz estimates together with Hölder and Minkowski inequalities, with an
implicit constant which is independent of $T>0$
\begin{align}
\Vert z_n - u_n \Vert_{L^5(-T,T) L^{10}} &\lesssim \Vert u_n^5 - v_n^5 \Vert_{L^1(-T,T) L^{2}} \nonumber \\
&\lesssim \int_{-T} ^T \Big[ \Vert v_n(t) \Vert^4_{L^{10}} \Vert u_n(t) - v_n(t) \Vert_{L^{10}} + \Vert u_n(t) - v_n(t) \Vert^5_{L^{10}} \Big] \; dt \nonumber \\
&\lesssim \int_{-T}  ^T   \Big[ \Vert v_n(t) \Vert_{L^{10}}^4 \Vert z_n(t) - u_n(t) \Vert_{L^{10}} + \Vert z_n(t) - u_n(t) \Vert^5_{L^{10}} \Big] \,dt\;+\epsilon_n(T), \label{eq:dilnl_dec1} 
\end{align}
where we decomposed $u_n(t) - v_n(t) = u_n(t) - z_n(t) + z_n(t) - v_n(t)$ in the last line, and 
\begin{equation*}
\epsilon_n(T) :=  \int_{-T}^T \Vert v_n(t) - z_n(t)\Vert_{L^{10}}^5 + \Vert v_n(t)\Vert_{L^{10}}^4\Vert v_n(t) - z_n(t)\Vert_{L^{10}} \; dt.
\end{equation*}
By Hölder inequality and (\ref{eq:dilnl_ZU})
\begin{equation}\label{eq:dilnl_epsp}
\epsilon'_n := \sup_{T>0} \epsilon_n(T) \leq \Vert {v}_n - {z}_n\Vert_{L^5(\mathbb R, L^{10})}^5 + \Vert v \Vert_{L^5(\mathbb R, L^{10}(\mathbb R ^3))}^4\Vert {v}_n - {z}_n\Vert_{L^5(\mathbb R, L^{10})} \to 0.
\end{equation}
By (\ref{eq:dilnl_dec1}), we have, with an implicit constant independent of $T$
\begin{equation} \label{eq:dilnl_decprefin}
\Vert z_n - u_n \Vert_{L^5(-T,T) L^{10}} \lesssim \int_{-T} ^T  \|v_n(t)\|_{L^{10}}^4 \Vert z_n(t) - u_n(t) \Vert_{L^{10}} \; dt +   \epsilon'_n  + \Vert z_n - u_n \Vert_{L^5(-T,T) L^{10}} ^5.
\end{equation}
Now, 
$\|v_n\|_{L^{10}}^{4}\in L^{\frac{5}{4}}(\mathbb R^3)$ and $\left\|\|v_n\|_{L^{10}}^4\right\|_{L^{\frac{5}{4}}(\mathbb R^3)}=\|v\|_{L^{5}L^{10}}^4.  $
Thus we get,  by (\ref{eq:dilnl_decprefin}), using the Gronwall-type lemma of \cite[Lemma 8.1]{MR2838120}, for all $T>0$, with $C>0$ independent of $T>0$:
\begin{equation} \label{eq:dilnl_decfin}
\Vert z_n - u_n \Vert_{L^5(-T,T) L^{10}} \leq C \big(  \epsilon'_n  + \Vert z_n - u_n \Vert_{L^5(-T,T) L^{10}} ^5 \big).
\end{equation} 
Let $\epsilon>0$ be small enough so that 
$
2 C\epsilon^5 \leq \frac{1}{2} \epsilon,
$
and $n_0$ large enough so that 
$
\epsilon'_n \leq \epsilon ^5.
$
for all $n\geq n_0$.
From (\ref{eq:dilnl_decfin}), it follows that if $T$ is such that $\Vert z_n - u_n \Vert_{L^5(-T,T) L^{10}} \leq \epsilon$, we have
$$
\Vert z_n - u_n \Vert_{L^5(-T,T) L^{10}} \leq \frac{1}{2} \epsilon.
$$
By a continuity argument, it gives $
\Vert z_n - u_n \Vert_{L^5(-T,T) L^{10}} \leq \epsilon/2
$ for any $T>0$ and $n\geq n_0$.
We can therefore send $T$ to infinity, and we obtain (\ref{eq:nlporf_goal}) thanks to (\ref{eq:dilnl_ZU}). 

To conclude, it remains to show that
$$
\sup_{t\in \mathbb R} \Vert \vec u_n(t) - \vec v_n(t) \Vert_{\mathcal H(\Omega)} \to 0.
$$
In order to do so, thanks to (\ref{eq:dilnl_ZU}) again, it suffices to show that 
\begin{equation*}\label{eq:dilnl_finish}
\sup_{t\in \mathbb R} \Vert \vec z_n(t) - \vec u_n(t) \Vert_{\mathcal H(\Omega)} \to 0.
\end{equation*}
This follows from (\ref{eq:nlporf_goal}): indeed, by equation \eqref{e:znomega}, energy estimates, then H\"older inequality
\begin{align*}
&\sup_{t\in\mathbb R} \Vert \vec z_n(t) - \vec u_n(t) \Vert_{\mathcal H(\Omega)} \leq \Vert u_n^5 - v_n^5 \Vert_{L^1(\mathbb R, L^2)} \\
&\hspace{1cm}\lesssim \int_{-\infty}  ^\infty   \Big[ \Vert v_n(t) \Vert_{L^{10}}^4 \Vert v_n(t) - u_n(t) \Vert_{L^{10}} + \Vert v_n(t) - u_n(t) \Vert^5_{L^{10}} \Big] \,dt 
\\ &\hspace{1cm}\lesssim \Vert v \Vert^{4}_{L^5(\mathbb R, L^{10})} \Vert v_n - u_n \Vert_{L^5(\mathbb R, L^{10})} + \Vert v_n - u_n \Vert^5_{L^5(\mathbb R, L^{10})},
\end{align*}
 which goes to zero thanks to (\ref{eq:nlporf_goal}), and the Lemma follows.\end{proof}

\subsection{Reduction to $t_n = 0$ and energy norms.} \label{subsec:assfree_reduc}

The goal of this paragraph is to reduce ourselves to $t_n=0$ and energy norms, and show that the projection of the initial data is harmless.

 \begin{lem} \label{lem:asfree_reduc}
 In order to show Lemma \ref{lem:assfree}, it suffices to obtain (\ref{znNRJ}) in the case $t_n =0$, $\vec\varphi\in C^{\infty}_c(\R^3)^2$ and $f\in C^{\infty}_c(\R^{1+3})$. It is also sufficient to assume \emph{either} $\lambda_n\to+\infty$; \emph{or} $|x_n|\to+\infty$ and $\lambda_n\to\lambda_0\in \R_+$.
\end{lem}

\begin{proof}
The reduction to $\vec\varphi\in C^{\infty}_c(\R^3)^2$ and $f\in C^{\infty}_c(\R^{1+3})$ is immediate by density using energy and Strichatz estimates fo all the linear equations. The second reduction is only performed by taking a subsequence in the case $\lambda_n$ bounded.

\textbf{Reduction to the energy-norm decay.} We first show that (\ref{znNRJ}) implies (\ref{strichartzzn}). By Sobolev embedding, (\ref{znNRJ}) implies the $L^{\infty}L^{6}$ decay
$$
\Vert u_n - v_n \Vert_{L^\infty L^6} \rightarrow 0.
$$
To obtain the decay in Strichartz norm (\ref{strichartzzn}) from the above, it suffices to interpolate between $L^{\infty}L^{6}$ and another Strichartz norm admissible for $\R^{3}$ and $\Omega$. Thanks to Assumption \ref{ass:strichartz}, we can select one $6<q<10$ and $5<p<+\infty$ with $ \frac 12 = \frac 1p + \frac 3q$ so that the pair $(p,q)$ is admissible. With an appropriate $0<\theta<1$, we get
\begin{align*}
\nor{u_{n}-v_{n}}{L^{5}L^{10}}&\leq \nor{u_{n}-v_{n}}{L^{\infty}L^{6}}^{\theta}\nor{u_{n}-v_{n}}{L^{p}L^{q}}^{1-\theta} \\
&\leq \nor{u_{n}-v_{n}}{L^{\infty}L^{6}}^{\theta}\left(\nor{u_{n}}{L^{p}L^{q}}+\nor{v_{n}}{L^{p}L^{q}}\right)^{1-\theta}.
\end{align*}
\textbf{Reduction to $t_n = 0$.} Now, we assume that Lemma \ref{lem:assfree} holds for $t_n = 0$ and show Lemma \ref{lem:assfree}
in the general case.
Let 
$$
\widetilde{u_n} := u_n(\cdot + t_n), \hspace{0.3cm} \widetilde{v_n} := v_n(\cdot + t_n),
 \hspace{0.3cm} \widetilde{f_n} := f_n(\cdot + t_n).
$$
Observe that $\widetilde{u_n}$, $\widetilde{v_n}$, $\widetilde{f_n}$ verify the assumptions of Lemma \ref{lem:assfree} with $\widetilde{t_n} := 0$. It follows that
$$
\sup_{t\in \mathbb R} \int_{\Omega} |\nabla(\widetilde{u_n} - \widetilde{v_n})(t)|^2 + |\partial_t(\widetilde{u_n} - \widetilde{v_n})(t)| ^2\to 0.
$$
But, for any $t$,
\begin{align*}
\int_{\Omega} |\nabla({u_n} - {v_n})(t)|^2 + |\partial_t(u_n - {v_n})(t)| 
&= \int_{\Omega} |\nabla(\widetilde{u_n} - \widetilde{v_n})(t-t_n)|^2 + |\partial_t(\widetilde{u_n} - \widetilde{v_n})(t-t_n)| \\ &\leq \sup_{s\in \mathbb R} \int_{\Omega} |\nabla(\widetilde{u_n} - \widetilde{v_n})(s)|^2 + |\partial_t(\widetilde{u_n} - \widetilde{v_n})(s)| ^2,
\end{align*}
and the result follows.
\end{proof}

\subsection{Asymptotic behaviour of projections}
The first show that the effect of the projections on $\Omega_n$ in harmless in the regimes we are interested in.
\begin{lem}\label{lm:projectdilat}
Assume that $\lambda_n \to + \infty$ or $|x_n| \to +\infty$. Then,
for any $\vec \varphi\in \mathcal{H}(\mathbb R^3)$, we have as $n \to \infty$
\begin{equation*}
 \Big\Vert P_{\Omega_n}\vec \varphi - \vec \varphi \Big\Vert_{\mathcal{H}(\mathbb R^3)} 
\underset{n\to +\infty}{\longrightarrow}  0.
\end{equation*}
\end{lem}
\begin{proof}
By density and the fact that all operators are projections, it is enough to prove the result for $\vec \varphi\in C_c^{\infty}(\R^3)^2$. We have $\Omega_n=\frac{\Omega- x_n}{\lambda_n}=\R^3\setminus \Theta_n$ with $\Theta_n=\frac{\Theta- x_n}{\lambda_n}$. 

We first assume $\lambda_n \to +\infty$. Up to a subsequence, we can distinguish the case $\frac{|x_n|}{\lambda_n}\to +\infty$ and $\frac{x_n}{\lambda_n}\to y_{\infty}\in \R^3$. When $\frac{|x_n|}{\lambda_n}\to +\infty$, we denote $R$ so that $\operatorname{supp}(\vec \varphi)\subset B(0,R)$. Since $\Theta$ is compact, there exists $n$ large enough so that $\Theta_n=\frac{\Theta- x_n}{\lambda_n}$ is included in $B(\frac{x_n}{\lambda_n},1)$ which is included in $B(0,R)^c$ for $n$ large enough. In particular, $ \operatorname{supp}(\vec \varphi)\subset \Theta_n^c=\Omega$. This implies $\vec \varphi\in \mathcal{H}(\Omega_n)$ and $P_{\Omega_n}\vec \varphi=\vec \varphi$. 

If $\lambda_n \to +\infty$ and $\frac{x_n}{\lambda_n}\to y_{\infty}\in \R^3$, using Lemma~\ref{lem:H1moins_point} and a similar result for $L^2(\R^3)$, it is enough to prove the result for $\vec \varphi\in C_c^{\infty}(\R^3\setminus \{y_{\infty}\})^2$. Let $\e>0$ so that $\operatorname{supp}(\vec \varphi)\subset \R^3 \setminus B(y_{\infty},\e)$. Since $\Theta$ is compact, $\lambda_n \to +\infty$ and $\frac{x_n}{\lambda_n}\to y_{\infty}$, there exists $n$ large enough so that $\Theta_n=\frac{\Theta- x_n}{\lambda_n}$ is included in $B(y_{\infty},\e)$. In particular, $\operatorname{supp}(\vec \varphi)\subset \R^3 \setminus \Theta_n=\Omega_n$. This gives again $\vec \varphi\in \mathcal{H}(\Omega_n)$ and $P_{\Omega_n}\vec \varphi=\vec \varphi$.

The remaining case, up to a subsequence is $|x_n| \to +\infty$ and $\lambda_n$ bounded by a constant $M>0$. In particular, $\frac{1}{\lambda_n}\geq \frac{1}{M}$. For $n$ large enough, $\Theta- x_n$ is included in $B(0,RM)^c$. In particular, $\Theta_n$ is included in $B(0,R)^c$ and we conclude as before.
\end{proof}

\begin{lem} \label{lem:proj_harmless}
Assume that $\lambda_n \to + \infty$ or $|x_n| \to +\infty$. Then, as $n\rightarrow \infty$,
\begin{equation}\label{eq:comp_dil000}
\Big(\lambda_n^{-1/2} \mathcal P_\Omega\big(\varphi^0(\frac{\cdot - x_n}{\lambda_n})\big),  \lambda_n^{-3/2}\mathbb{1}_\Omega \varphi^1(\frac{\cdot- x_n}{\lambda_n})\Big) 
-
\Big(\lambda_n^{-1/2} \varphi^0(\frac{\cdot- x_n}{\lambda_n}),  \lambda_n^{-3/2} \varphi^1(\frac{\cdot- x_n}{\lambda_n})\Big)
\to  0 
\quad\text{in}\quad\mathcal{H}(\mathbb R^3).
\end{equation}
\end{lem} 

\begin{proof}
By change of variables
\begin{equation*}
\big\Vert \text{l.h.s of }(\ref{eq:comp_dil000}) \big\Vert_{\mathcal{H}(\mathbb R^3)} = \Big\Vert\big(\mathcal P_{\frac{\Omega- x_n}{\lambda_n}}\varphi^0,  \mathbf 1_{\frac{\Omega- x_n}{\lambda_n}} \varphi^1\big) - \vec \varphi \Big\Vert_{\mathcal{H}(\mathbb R^3)}
\longrightarrow \Big\Vert\big(\mathcal P_{X_{\mathcal O}}\varphi^0, \mathbb{1}_{X_{\mathcal O}} \varphi^1\big) - \vec \varphi\Big\Vert_{\mathcal{H}(\mathbb R^3)} = 0,
\end{equation*}
which is proved in Lemma~\ref{lm:projectdilat}.
 Here $X_{\mathcal O} = \mathbb R^3$ or $\mathbb R^3 \backslash \{ 0 \}$, and we used the fact that $\dot H_0^1(\mathbb R^3 \backslash \{ 0 \}) = \dot H^1(\mathbb R^3)$ (see Lemma \ref{lem:H1moins_point}) and hence $\mathcal P_{X_{\mathcal O}} = \operatorname{Id}$.
\end{proof}

\subsection{Profiles going to infinity} \label{subsec:to_infinity}

We now show Lemma \ref{lem:assfree} for a profile going to infinity, that is with
$$
\lambda_n \rightarrow \lambda_{0}\in {\R_+} \hspace{0.5cm} \text{and} \hspace{0.5cm} |x_n| \rightarrow \infty.
$$

\begin{proof}[Proof of Lemma \ref{lem:assfree} for $\lambda_n \rightarrow \lambda_{0}\in \R_+$ and $|x_n| \rightarrow \infty$.]
By Lemma \ref{lem:asfree_reduc}, we can assume that $t_{n}=0$, and it suffices  to obtain (\ref{znNRJ}). 
 Hence
 $$
 v_n = \lambda_n^{-1/2}v(\frac{\cdot}{\lambda_n},\frac{\cdot - x_n}{\lambda_n}),
 $$
 where $v$ is solution to
 $$
 \begin{cases}
 \partial_t^2 v - \Delta v = f \text{ in }\mathbb R^3, \\
 (v, \partial_t v) (t=0) = \vec \varphi.
 \end{cases}
 $$
Let $\beta \in C^\infty_c$ be so that $\beta = 1$ near $\Omega^c$ (i.e. near the obstacles) and
$$
z_n := u_n - (1-\beta)v_n.
$$
Observe that $z_n$ satisfies (recalling the notation $P_\Omega := (\mathcal P_\Omega, \mathbf 1_\Omega)$)
$$
\begin{cases}
\partial_t^2 z_n - \Delta z_n = -[\Delta, \beta]v_n+\beta f_n \text{ in } \Omega, \\
\vec{z_n}(0) = {-}\beta P_\Omega \vec {v_n} (0)
+ \vec{\epsilon_n}, \\
z_n = 0 \text{ on }\partial \Omega,
\end{cases}
$$
where
$$
\vec{\epsilon_n} := (1 - \beta)(P_\Omega - \operatorname{I})\vec{v_n}(0){+\beta (P_\Omega - \operatorname{I})\vec{v_n}(0)}
$$
verifies, by Lemma \ref{lem:proj_harmless}, 
$$
\vec{\epsilon_n} \to \vec 0 \; \text{ in }\dot H^1 \times L^2 (\mathbb R^3).
$$
Therefore, using Duhamel formula, it suffices in order to obtain (\ref{znNRJ}) to show that
\begin{equation} \label{eq:pre_xn_infty_reduc}
\Vert[\Delta, \beta]v_n \Vert_{L^1 L^2} +\Vert\beta v_n \Vert_{L^{\infty} H^{1}}+\Vert\beta \partial_{t}v_n \Vert_{L^{\infty} L^{2}}+\Vert\beta f_n \Vert_{L^{1} L^{2}}\rightarrow 0,
\end{equation} 
where all the norms are taken in $\mathbb R \times \mathbb R^3$ and we used again the fact that, by Lemma \ref{lem:proj_harmless}, $P_\Omega \vec{v_n}(0) - \vec{v_n}(0) \to 0$ in $\dot H^1 \times L^2 (\mathbb R^3)$ in order to eliminate  the projection.
In order to show (\ref{eq:pre_xn_infty_reduc}), it is sufficient to prove that for any $\gamma \in C^\infty_c(\mathbb R^3)$ equal to one on the support of $\beta$, we have
\begin{equation} \label{eq:xn_infty_reduc}
\Vert \gamma v_n \Vert_{L^1 L^2} + \Vert \gamma \nabla v_n \Vert_{L^1 L^2} + \Vert \gamma v_n \Vert_{L^{\infty} L^2} + \Vert \gamma \nabla v_n \Vert_{L^{\infty} L^2} 
+\Vert \gamma \partial_{t}v_n \Vert_{L^{\infty} L^{2}}+ \Vert \gamma f_n \Vert_{L^{1} L^{2}}\ \to 0.
\end{equation}
Showing (\ref{eq:xn_infty_reduc}) occupies the remainder of the proof.
As noticed in Lemma \ref{lem:asfree_reduc}, we can assume that both $\vec \varphi = (\varphi^0, \varphi^1 )$ and $f$ are smooth and compactly supported.

\textbf{The $L^1 L^2$ decay.} We will first show
\begin{equation} \label{eq:xn_infty_L1L2}
\Vert \gamma v_n \Vert_{L^1 L^2} + \Vert \gamma \nabla v_n \Vert_{L^1 L^2} \rightarrow 0.
\end{equation} 
As
$
v_{n}=\lambda_n^{-1/2}v(\frac{\cdot}{\lambda_n}, \frac{\cdot-x_{n}}{\lambda_n})
$,
we have
\begin{equation}\label{eq:xn_infty_chgtvar}
\Vert\gamma \nabla v_n \Vert_{L^1 L^2}=\lambda_{n}\Vert\gamma(\cdot\lambda_{n}+x_{n}) \nabla v \Vert_{L^1 L^2}.
\end{equation}
As $\vec \varphi = (\varphi^0, \varphi^1 )$ and $f$ are assumed to be compactly supported, by finite speed of propagation and Huygens principle, $v$ is supported in 
\begin{equation} \label{eq:xn_inf:suppv}
\operatorname{supp} v \subset \{ -R \leq |t| - |x| \leq R \}
\end{equation}
for some $R>0$. Moreover, if $\gamma$ is supported in $B(0,C)$,  $\gamma(\cdot\lambda_{n}+x_{n})$ is supported in 
\begin{equation} \label{eq:xn_inf_suppgamman}
\operatorname{supp} \gamma(\cdot\lambda_{n}+x_{n}) \subset B(-\frac{x_{n}}{\lambda_{n}},\frac{C}{\lambda_{n}}) \subset \Big\{|x| \in \big[\frac{|x_n| - C}{\lambda_n}, \frac{|x_n| + C}{\lambda_n}\big] \Big\}.
\end{equation}
In particular by (\ref{eq:xn_inf:suppv}) and (\ref{eq:xn_inf_suppgamman}), the product $\gamma(\cdot\lambda_{n}+x_{n}) \nabla v$ is supported in a time interval 
\begin{equation} \label{eq:xn_inf_suppgproduit}
\operatorname{supp} \gamma(\cdot\lambda_{n}+x_{n}) \nabla v \subset \big\{ |t|\in I_{n} \big \}, \; I_{n} :=\Big[\frac{|x_n|-C}{\lambda_n}-R, \frac{|x_n|+C}{\lambda_n}+R\Big]
\end{equation}
of size 
\begin{equation}
 \label{eq:xn_inf:sizeIn}
|I_{n}|\leq \frac{2C}{\lambda_{n}} + 2R \lesssim \frac{1}{\lambda_n}
\end{equation}
We notice also that $I_n$ is going to infinity, in the sense that, for any $T >0$, $I_n \cap [-T, T] = \emptyset$ for $n$ big enough in both cases $\lambda_0=0$ or $\lambda_0>0$ and $|x_n|\to +\infty$. We only treat the part of the norm where $t\in I_n$, $t>0$ (that is $t$ large positive). The other part for large negative $t$ is treated similarly. 

Recall that $f$ is assumed now compactly supported in time thanks to the approximation. In particular, for $t$ large enough, $v$ is solution of $\Box v=0$. Denoting $\widetilde{F}\in L^{2}(\R\times \mathbb{S}^{2})$ the radiation field of Friedlander \cite{F:80}, see for instance \cite[Proposition 1.1]{CL} we have
\begin{equation*} 
\Big\Vert\nabla v(t)-\frac{1}{|\cdot|}\widetilde{F}(|\cdot|-t, \frac{\cdot}{|\cdot|})\frac{.}{|.|}\Big\Vert_{L^2}\underset{t\to+\infty}{\longrightarrow}0.\end{equation*} 
We fix $\e>0$. Approximating $\widetilde{F}$ in $L^{2}(\R\times \mathbb{S}^{2})$, we can pick $F\in C^{0}_{c}(\R\times \mathbb{S}^{2})^{3}$ and $T_0>0$ so that 
\begin{equation} \label{eq:profil_LaurentCote0}
\Big\Vert\nabla v(t)-\frac{1}{|\cdot|}F(|\cdot|-t, \frac{\cdot}{|\cdot|})\Big\Vert_{L^2}\leq \epsilon, \hspace{0.5cm}\forall t \geq T_0.
\end{equation} 
Let $n$ big enough so that $I_n \subset [T_0, + \infty)$, then, by the above (\ref{eq:profil_LaurentCote0}) together with (\ref{eq:xn_inf:sizeIn})
\begin{equation} \label{eq:xn_inf_approx}
\lambda_{n}\Big\Vert\gamma(\cdot\lambda_{n}+x_{n}) \left(\nabla v - \frac{1}{|\cdot|}F(|\cdot|-t,  \frac{\cdot}{|\cdot|})\frac{.}{|.|}\right) \Big\Vert_{L^1(I_n) L^2} \lesssim \lambda_n |I_n| \epsilon \lesssim \epsilon.
\end{equation}
On the other hand, using  (\ref{eq:xn_inf_suppgamman}) to bound $\frac{1}{|x|}$, then (\ref{eq:xn_inf:sizeIn})
\begin{align} 
\lambda_{n}\Big\Vert\gamma(\cdot\lambda_{n}+x_{n}) \frac{1}{|\cdot|}F(|\cdot|-t,  \frac{\cdot}{|\cdot|}))\Big\Vert_{L^1(I_n) L^2} &\leq \lambda_n |I_n| \frac{\lambda_n}{|x_n|}\sup_{t\in I_n} \Big\Vert \gamma(\cdot\lambda_{n}+x_{n})F(|\cdot|-t,  \frac{\cdot}{|\cdot|})) \Big\Vert_{L^2} \nonumber \\
&\leq \frac{\lambda_n}{|x_n|}\sup_{t\in I_n} \Big\Vert \gamma(\cdot\lambda_{n}+x_{n})F(|\cdot|-t,  \frac{\cdot}{|\cdot|})) \Big\Vert_{L^2}. \label{eq:xn_inf_L1L2}
\end{align}
But, for fixed time $t \in I_n$, the support of $\gamma(\cdot\lambda_{n}+x_{n})F(|\cdot|-t,  \frac{\cdot}{|\cdot|}))$ is included in 
\begin{equation}\label{eq:Cn0}
\operatorname{supp}\gamma(\cdot\lambda_{n}+x_{n})F(|\cdot|-t,  \frac{\cdot}{|\cdot|})) \subset C_n(t) := B(-\frac{x_n}{\lambda_n}, \frac{C}{\lambda_n}) \cap \big\{x\in \R^3\textnormal{ s.t. } |x| \in [t-R, t+R] \big\}.
\end{equation}
Applying Lemma \ref{lm:volC} with $x_0=\frac{x_n}{\lambda_n}$ and $r=\frac{C}{\lambda_n}$ with $\epsilon=\max(C |x_n|^{-1},R\lambda_n|x_n|^{-1})\lesssim |x_n|^{-1}$ and taking $n$ large enough so that the Lemma applies, we obtain $|C_n(t)| \lesssim R t^2 |x_n|^{-1}$ for $t\geq \epsilon_0^{-1}R$. Therefore, for $t \in I_n$ (recall the definition (\ref{eq:xn_inf_suppgproduit})) we have the estimate
\begin{equation} \label{eq:taille_Cn}
\sup_{t\in I_n}|C_n(t)| \lesssim \frac{|x_n|}{\lambda_n^2}.
\end{equation}

Now, from (\ref{eq:xn_inf_L1L2}) together with (\ref{eq:Cn0}) and (\ref{eq:taille_Cn})
\begin{equation*}
\lambda_{n}\Big\Vert\gamma(\cdot\lambda_{n}+x_{n}) \frac{1}{|\cdot|}F(|\cdot|-t,  \frac{\cdot}{|\cdot|}))\Big\Vert_{L^1(I_n) L^2} \lesssim  \frac{\lambda_n}{|x_n|}\sup|F| \sup|\gamma|\sup_{t\in I_n} |C_n(t)|^\frac{1}{2} \lesssim \frac{1}{|x_n|^{\frac 12}} \to 0.
\end{equation*} 
Combining the above with (\ref{eq:xn_inf_approx}), (\ref{eq:xn_inf_suppgproduit}) and (\ref{eq:xn_infty_chgtvar}), we obtain
$$
\Vert \gamma \nabla v_n \Vert_{L^1L^2} \rightarrow 0.
$$
For the $L^{2}$ term $\Vert \gamma v_n \Vert_{L^2}$, in the case $\lambda_{0}=0$, we only need to estimate
$$
\Vert\gamma  v_n \Vert_{L^1 L^2}=\lambda_{n}^{2}\Vert\gamma(\cdot\lambda_{n}+x_{n}) v \Vert_{L^1 L^2}\leq \lambda_{n}^{2}\Vert v \Vert_{L^1(I_{n}) L^2}\leq \lambda_{n}\Vert v \Vert_{L^{\infty}L^2}\leq C\lambda_{n} \rightarrow 0,
$$
where we have used \eqref{eq:xn_inf:sizeIn}.
In the case $\lambda_{0}\neq 0$, we make the same reasoning as for the gradient term approximating $v$ in $L^{2}$. The $L^1L^2$ decay (\ref{eq:xn_infty_L1L2}) follows.

\textbf{The uniform energy decay.} We now show
\begin{equation} \label{eq:xn_infty_LinftyL2}
\Vert \gamma v_n \Vert_{L^\infty L^2} + \Vert \gamma \nabla v_n \Vert_{L^\infty L^2} + \Vert \gamma \partial_t v_n \Vert_{L^\infty L^2} \rightarrow 0.
\end{equation} 
For fixed $t$, we estimate
\begin{equation} \label{eq:xn_infty_chgtvar2}
\Vert\gamma \nabla v_n(t) \Vert_{L^2}=\Vert \gamma(\cdot\lambda_{n}+x_{n})\nabla v(\frac{t}{\lambda_{n}}) \Vert_{L^2}.
\end{equation}
By (\ref{eq:xn_inf_suppgproduit}), this term is supported in 
\begin{equation} \label{eq:xn_infty_suppprod2}
\operatorname{supp} \gamma(\cdot\lambda_{n}+x_{n})\nabla v(\frac{t}{\lambda_{n}}) \subset \Big\{|t| \in \lambda_n I_n \Big\} \subset \Big\{ |t| \in \big[|x_n| - C, |x_n| + C\big]\Big\},
\end{equation}
for another constant $C$. We take $n$ large enough so that $\lambda_n I_n \subset [T_0, + \infty)$. Then, using (\ref{eq:profil_LaurentCote0})
\begin{equation} \label{eq:xn_infty_approx2}
\sup_{|t|\in \lambda_n I_n} \Big\Vert \gamma(\cdot\lambda_{n}+x_{n})(\nabla v(\frac{t}{\lambda_{n}}) - \frac{1}{|\cdot|}F(|\cdot|-\frac{t}{\lambda_n}, \frac{\cdot}{|\cdot|}) \Big\Vert_{L^2} \leq \epsilon. 
\end{equation}
On the other hand, by (\ref{eq:xn_inf_suppgamman}), (\ref{eq:Cn0}), and (\ref{eq:taille_Cn})
\begin{equation*}
\sup_{|t|\in \lambda_n I_n} \Big\Vert \gamma(\cdot\lambda_{n}+x_{n}) \frac{1}{|\cdot|}F(|\cdot|-\frac{t}{\lambda_n}, \frac{\cdot}{|\cdot|}) \Big\Vert_{L^2} 
\leq \frac{\lambda_n}{|x_n|} \sup|\gamma| \sup|F| \sup_{t\in \lambda_n I_n} |C_n(\frac{t}{\lambda_n})|^\frac{1}{2} \lesssim \frac{1}{|x_n|^{\frac 12}},
\end{equation*}
and combining the above with (\ref{eq:xn_infty_chgtvar2}), (\ref{eq:xn_infty_suppprod2}) and (\ref{eq:xn_infty_approx2}) gives the decay of the gradient term
$$
\Vert \gamma \nabla v_n \Vert_{L^\infty L^2} \to 0.
$$
We proceed similarly for the term $\Vert\gamma \partial_{t} v_n(t) \Vert_{L^2}$.
For the $L^{2}$ term $\Vert\gamma v_n(t) \Vert_{L^2}$, in the case $\lambda_0 = 0$, we only need to estimate
$$
\Vert\gamma  v_n(t) \Vert_{L^2}=\lambda_{n}\Vert \gamma(\cdot\lambda_{n}+x_{n}) v(\frac{t}{\lambda_{n}}) \Vert_{L^2}\leq C \lambda_{n}\Vert v\Vert_{L^{\infty}L^2},
$$
and in the case $\lambda_n \neq 0$, we proceed in the same way as for the gradient term approximating $v$ in $L^{2}$. The uniform decay (\ref{eq:xn_infty_LinftyL2}) follows.

\textbf{Conclusion.}
To finish, we need to estimate $\Vert \gamma  f_{n} \Vert_{L^1 L^2}$. But, as we assumed that $f$ is compactly supported in $\R\times \R^{3}$, for $n$ large enough $\gamma f_{n}=0$. Combined with the two previous steps (\ref{eq:xn_infty_L1L2}) and (\ref{eq:xn_infty_LinftyL2}), we thus have (\ref{eq:xn_infty_reduc}) and hence the uniform decay of the energy  (\ref{znNRJ}).

\end{proof}
\begin{remark}
We chose to present the above proof dealing with all the cases $\lambda_n$ bounded, $|x_n|\to\infty$ at once. But observe that the 
case where $\lambda_n$ is bounded away from zero (included in the proof above), is actually simpler: for example, via Huygens principle, $\Vert \gamma \nabla v_n \Vert_{L^1 L^2} = \lambda_n\Vert \gamma(\cdot + x_n) \nabla v \Vert_{L^1 L^2} \lesssim \Vert \nabla v \Vert_{L^\infty(t\geq x_n / 2) L^\infty}\rightarrow 0$.
\end{remark}

\subsection{Dilating profiles} \label{subsec:dilating}

We now show Lemma \ref{lem:assfree} for a dilating profile, that is with
$$
\lambda_n \rightarrow + \infty.
$$
The proof will rely on the following, related to the hidden regularity.

\begin{lem}[A priori control of the boundary term] \label{lem:apriori_bdd}
Assume that $\partial\Omega$ is smooth and bounded. Then, there exists $C>0$ so that, 
for any $\vec \varphi \in \dot H^1 \times L^2 (\mathbb R^3)$ and $f\in L^1(\mathbb R, L^2(\mathbb R^3))$, if $u$
is solution to
$$
\begin{cases}
\partial_t^2 u - \Delta u = f \text{ in } \Omega, \\
u = 0 \text{ on } \partial \Omega, \\
(u, \partial_t u) (t=0) = \vec \varphi,
\end{cases}
$$
then, for any $t_2>t_1$,
$$
\int_{t_1}^{t_2} \int_{\partial \Omega} |\partial_\nu u |^2 \; d\sigma dt \leq C \big(1+t_2 - t_1) \big(\Vert \vec \varphi \Vert^2_{\dot H^1 \times L^2} + \Vert f \Vert^2_{L^1L^2}\big).
$$
\end{lem}
\begin{proof}We assume $\vec \varphi\in C^{\infty}_c(\R^{1+3})^2$, $f\in C^{\infty}_c(\R^{1+3})$ and conclude by density.
We first derive a bound on the energy. Multiplying the equation by $\partial_t u$ and integrating by parts, we get
$$
\partial_t \Vert (u, \partial_tu) (t) \Vert^2_{\dot H^1 \times L^2} = 2 \int_\Omega f \partial_t u.
$$
Hence, by Cauchy-Schwarz inequality
$$
\Big| \partial_t \mathfrak \Vert (u, \partial_tu) (t) \Vert^2_{\dot H^1 \times L^2} \Big| \leq 2 \Vert f \Vert_{L^2} \Vert \partial_t u \Vert_{L^2}
\leq  2 \Vert f \Vert_{L^2} \Vert (u, \partial_tu) (t) \Vert_{\dot H^1 \times L^2}.
$$
It follows that, for any $t$
\begin{equation} \label{eq:apriori_bdd_control_energy}
\Vert (u, \partial_tu) (t) \Vert_{\dot H^1 \times L^2} \leq \Vert \vec \varphi \Vert_{\dot H^1 \times L^2} + \Vert f \Vert_{L^1L^2}.
\end{equation}

Now, observe that, for any $\chi$ sufficiently regular and decaying, using the equation and integrating by parts yields the following multiplier identity
\begin{multline} \label{eq:Mo_basic}
\partial_t\Big( \int_\Omega - \partial_t u \nabla u \cdot \nabla \chi\Big) = \int_\Omega D^2 \chi \nabla u \cdot \nabla u - \frac 1 2 \int_\Omega | \nabla u |^2 \Delta \chi  + \frac 1 2 \int_\Omega | \partial_t u |^2 \Delta \chi \\ - \int f \nabla u \cdot \nabla \chi - \frac 1 2 \int_{\partial \Omega} |\partial_\nu u |^2 \partial_\nu \chi. 
\end{multline}
Let $A>0$ be so that $\Omega^c \subset B(0,A)$ and define $\chi \in C^\infty_c$ so that
$$
\operatorname{supp}\chi\subset B(0,A+1), \hspace{0.3cm} \nabla \chi = - \nu \text{ on }\partial \Omega.
$$
Integrating (\ref{eq:Mo_basic}) between $t_1$ and $t_2$, the bound on the energy (\ref{eq:apriori_bdd_control_energy}) and Cauchy-Schwarz inequality give the result.
\end{proof}

\begin{proof}[Proof of Lemma \ref{lem:assfree} for $\lambda_n \rightarrow +\infty$.] 
By Lemma \ref{lem:asfree_reduc}, we can assume that $t_{n}=0$ and $\vec\varphi\in C^{\infty}_c(\R^3)^2$, $f\in C^{\infty}_c(\R^{1+3})$, 
 and it then suffices to show (\ref{znNRJ}). Hence $v_n$ writes
$$
v_n = \lambda_n^{-1/2}v(\frac{\cdot}{\lambda_n},\frac{\cdot - x_n}{\lambda_n}),
$$
where $v$ is solution to
$$
\begin{cases}
\partial_t^2 v - \Delta v = f \text{ in }\mathbb R^3, \\
(v, \partial_t v) (t=0) = \vec \varphi.
\end{cases}
$$
Using the equations and integrating by parts, we get
\begin{equation}\label{eq:comp_dil1}
\frac 12 \frac{d}{dt}\int_\Omega |\nabla(u_n - v_n)|^2 + |\partial_t(u_n - v_n)|^2 =- \int_{\partial \Omega} \partial_t v_n \partial_\nu(u_n - v_n) \; d\sigma.
\end{equation}
Now, observe that, as $\partial_t v$ and $\nabla v$ are bounded,
\begin{equation}\label{eq:comp_dil2}
\sup_{\partial \Omega}|\partial_t v_n | + |\partial_\nu v_n | \lesssim \lambda_n^{-3/2}.
\end{equation}
In addition, as  $\vec \varphi$ and $f$ are compactly supported, by the strong Huygens principle, $v$ is supported in $\{ -R \leq |t|-|x| \leq R \}$, and thus
\begin{equation}\label{eq:comp_dil3}
\operatorname{supp} v_n \cap \big(\mathbb R \times \partial\Omega\big)\subset\Big\{ - A - R\lambda_n + |x_n| \leq |t| \leq A + R\lambda_n + |x_n|\Big\}.
\end{equation}
Therefore, combining (\ref{eq:comp_dil1}) with (\ref{eq:comp_dil2}) and (\ref{eq:comp_dil3}) and using Cauchy-Schwarz inequality together with the fact that $|\partial \Omega | < \infty$ we get, {for $t>0$}
\begin{multline}\label{eq:comp_dil3a}
\frac{d}{dt}\int_\Omega |\nabla(u_n - v_n)|^2 + |\partial_t(u_n - v_n)|^2 \lesssim \lambda_n^{-3}\mathbb{1}_{t \in [{|x_n|}-C\lambda_n, {|x_n|}+C\lambda_n]} \\ + \lambda_n^{-3/2}\mathbb{1}_{t \in [{|x_n|}-C\lambda_n, {|x_n|}+C\lambda_n]} \Big( \int_{\partial \Omega} |\partial_\nu u_n|^2 d\sigma\Big)^{1/2}.
\end{multline}
For $t>0$, integrating the above between $0$ and $t$ and using Lemma \ref{lem:proj_harmless} to control the integral at $t=0$, then using Cauchy-Schwarz inequality, we obtain 
\begin{multline} \label{eq:comp_dil4}
\int_\Omega |\nabla(u_n - v_n)|^2(t) + |\partial_t(u_n - v_n)|^2(t) \lesssim o(1) + \lambda_n^{-3/2}  \int_{{|x_n| - C\lambda_n}}^{{|x_n|+}C\lambda_n} \Big( \int_{\partial \Omega} |\partial_\nu u_n|^2 d\sigma \Big)^{1/2}dt \\
\lesssim  o(1) +  \lambda_n^{-1}  \Big( \int_{{|x_n| - C\lambda_n}}^{{|x_n|+}C\lambda_n} \int_{\partial \Omega} |\partial_\nu u_n|^2 \; d\sigma dt\Big)^{1/2}.
\end{multline}
In addition, by the a-priori estimate of the boundary term Lemma \ref{lem:apriori_bdd} above, then using the fact that $P_\Omega : \mathcal H(\mathbb R^3) \to \mathcal H(\Omega)$ is bounded and the scale invariance of $\dot H^1 \times L^2(\mathbb R^3)$ and $L^1 L^2 (\mathbb R^3)$, we  have
\begin{equation}\label{eq:comp_dil5}
\int_{{|x_n| - C\lambda_n}}^{{|x_n|+}C\lambda_n} \int_{\partial \Omega} |\partial_\nu u_n|^2 \; d\sigma dt \lesssim \lambda_n \big(\Vert P_\Omega \vec{v_n}(0)\Vert_{\mathcal H} + \Vert f_n \Vert_{L^1 L^2}) \lesssim \lambda_n.
\end{equation}
Combining (\ref{eq:comp_dil4}) with (\ref{eq:comp_dil5}) gives
$$
\sup_{t>0} \int_\Omega |\nabla(u_n - v_n)|^2 + |\partial_t(u_n - v_n)|^2 \to 0.
$$
The proof of the decay of the supremum over $t<0$ is similar, integrating
$- \frac{d}{dt}\int_\Omega |\nabla(u_n - v_n)|^2 + |\partial_t(u_n - v_n)|^2$. 
This finishes the proof.
\end{proof}

\section{Concentrating profiles with localized center} \label{sec:conc_prof}

In this Section, we deal with profiles $\lambda_{n}\to 0$ and $x_{n}\to x_{\infty}\in \R^{3}$. In order to be consistent with the semiclassical literature, we will denote here $h_n := \lambda_n$.
One of the main goal of the Section is to prove the following non concentration property.

\begin{prop}[Non concentration in the case of concentrating localised profiles]
\label{propnoncon_loc}
$\;$\newline Let $(C_n)_{n\geq 1}$ be an arbitrary sequence converging to $+\infty$, and $v_n$ a linear concentrating wave at scale $h_n$, that is, $v_n = \varphi_{\Omega,\mathcal{O},n}$ for a scale-core $\mathcal{O} = \{ (h_n)_{n\geq 1}, 0, (x_n)_{n\geq 1}  \}$ with $h_n \to 0$ and $x_n \to x_0 \in \overline \Omega$.
Then, we have for $I_n=\R \setminus [-C_n h_n,C_n h_n]$,
$$
\nor{v_n}{L^{\infty}(I_n,L^6(\Omega))}\rightarrow 0.
$$
\end{prop}

\subsection{Reduction to non-trapped data}
We first show that, thanks to the weak trapping Assumption \ref{ass:weaktrap}, we can reduce ourselves to show Proposition \ref{propnoncon_loc} for a non-trapped data, in the sense of the following definition. In the following, $\widetilde \varphi_t$ denotes the bicharacteristic flow of $-\Delta$  in $\Omega$, defined on the compressed co-tangent bundle $^b T^* \Omega$, renormalized at speed one, and $j : T^* \Omega \to ^b T^* \Omega$ is the canonical projection (see the next subsection for more details).

\begin{defn}
[Non-trapped data] \label{ass:nontrapdata}Let $x_{0}\in\overline{\Omega}$ and $\vec \varphi\in \mathcal{S}(\R^{3})^{2}$. We say that $(x_{0},\vec \varphi)$ is a non-trapped data if for any $R>0$, there exists $T_{\rm esc}(R)>0$ so that for every non zero direction $\xi\in \operatorname{supp}(\widehat{\varphi^{0}})\cup \operatorname{supp}(\widehat{\varphi^{1}})$ we have $\widetilde \varphi_t j(x_0, \xi)\in T^*B(0,R)^c$ for all $\vert t\vert \geq T_{\rm esc}(R)$.
\end{defn}

\begin{lem}
\label{lm:decphonontrap}
Let $\vec \varphi\in \mathcal{S}(\R^{3})^{2}$ and $\e>0$. Then, under Assumption \ref{ass:nonreco}, we can decompose $\vec \varphi=\vec \varphi_{\rm trap}+\vec \varphi_{\rm nontrap}$ with $\nor{\vec \varphi_{\rm trap}}{\dot{H}^{1}\times L^{2}(\R^{3})}\leq \e$ and $\vec \varphi_{\rm nontrap}\in \mathcal{S}(\R^{3})^{2}$  is a non-trapped data as in Definition \ref{ass:nontrapdata}.
\end{lem}
\begin{proof}
We first cut the mass at infinity and near zero. Let $R\gg1$ be so that 
$$
\Vert  \widehat{\vec \varphi} \Vert_{ L^2(|\xi|^2d\xi, |\xi|\geq R) \times L^2(|\xi|\geq R)} \leq  \frac \epsilon3,
$$
then, $R>0$ being fixed, let $\chi_{\leq R}(x) := \chi(\frac{|x|^2}{R^2})$ where $\chi \in C_c^\infty(\mathbb R)$ is so that $\chi = 1$ near $[0,R]$, $\operatorname{supp} \chi \subset [0,2R]$, and $0\leq \chi \leq 1$, in such a way that 
$$
\Vert (1-\chi_{\leq R})  \widehat{\vec \varphi} \Vert_{L^2(|x|^2dx) \times L^2} \leq \frac \epsilon3.
$$
Similarly, for $\eta > 0$ small enough, we construct $\chi_{\leq \eta}$ so that $\chi = 1$ near $B(0, \eta)$, $\operatorname{supp} \chi_{\leq \eta} \subset B(0, 2\eta)$ and
$$
\Vert \chi_{\leq \eta} \widehat{\vec \varphi} \Vert_{L^2(|x|^2dx) \times L^2} \leq \frac \epsilon3.
$$

We now cut the trapped frequencies in $B(0,2R) \backslash B(0, \eta)$.
 Recall the definition of the set $V_n \subset \mathbb{S}^2$ from the weak trapping Assumption \ref{ass:weaktrap}.
Since $|V_n| \to 0$ and from the regularity of the Lebesgues measure on $\mathbb{S}^2$, there exists a family of open sets $U_n$ so that $V_n \subset U_n$ and $|U_n| \to 0$. From this family, using $V_{n+1}\subset V_n$, we can construct a sequence of open sets $\mathcal O_n \subset \mathbb S^2$ so that
$$
V_{n} \subset \mathcal O_n, \hspace{0.3cm} \mathcal{O}_{n+1} \subset \mathcal{O}_n, \hspace{0.3cm} |\mathcal O_n| \to 0.
$$
Indeed, to construct $\mathcal O_n$, it suffices to work inductively, and to set $\mathcal O_{n+1} := U_{n+1} \cap \mathcal O_{n}$. We can check by induction that it satisfies the above property. The second property is immediate and the third one is a consequence of $|U_n| \to 0$ and $\mathcal O_{n} \subset U_{n}$. For the first one, we use the property $V_{n+1}\subset U_{n+1}$ and the decreasing property $V_{n+1}\subset V_n$, combined with the iteration assumption $V_{n} \subset \mathcal O_n$ to prove $V_{n+1}\subset U_{n+1}\cap \mathcal O_{n}= \mathcal O_{n+1}$ which is the result at step $n+1$.

 Now, denoting $IX := \{ \lambda x, \; \lambda \in I, \, x\in X\}$, observe that $\mathbb{1}_{(\frac \eta 2,3R) {\mathcal O_n}} \to 0$ almost surely because $|(\frac \eta 2,3R) {\mathcal O_n}| \to 0$ and $(\frac \eta 2,3R) {\mathcal O_n}$ are non-increasing. Hence, by dominated convergence,
    $$
    \Vert 1_{(\frac \eta 2,3R) \mathcal O_n)}\widehat{\vec \varphi}  \Vert_{(L^2(|\xi|^2 d\xi) \times L^{2})} \to 0.
    $$
    We now fix $n \gg 1$ big enough so that the above quantity is $\leq \frac \epsilon 3$.
    Thanks to (the differential version of) Urysohn's Lemma, there exists a function $\chi_{\rm trap} \in C^\infty(\mathbb R^3, [0,1])$ so that $\chi_{\rm trap} = 1$ near $[\eta,2R]V_{n}$ and $\operatorname{supp} \chi_{\rm trap} \subset (\frac \eta 2, 3R)\mathcal O_n$. We then have
    $$
    \Vert \chi_{\rm trap}\widehat{\vec \varphi}  \Vert_{L^2(|\xi|^2 d\xi) \times L^{2}} \leq \frac \epsilon 3.
    $$
    To conclude, we take
    \begin{align*}
     &\widehat{\vec \varphi_{\rm nontrap}} := \chi_{\leq R}(1-\chi_{\leq \eta})(1-\chi_{\rm trap}) \widehat{\vec \varphi}, \\ &\widehat{\vec \varphi_{\rm trap}} := (1-\chi_{\leq R}) \widehat{\vec \varphi} + \chi_{\leq \eta}\widehat{\vec \varphi}+ \chi_{\leq R}(1-\chi_{\leq \eta})\chi_{\rm trap} \widehat{\vec \varphi},
    \end{align*}
    which satisfies the claim by construction.
\end{proof} 

\begin{lem}
\label{lm:enoughnonconc}
In order to show Proposition \ref{propnoncon_loc}, 
we can assume that the data $\vec \varphi_n$ is non-trapped as in Definition \ref{ass:nontrapdata}.
\end{lem}
 \begin{proof} 
 Fix $\e>0$. For $\vec \varphi\in \dot{H}^1\times L^2(\R^3)$, we approximate $\vec \varphi$ by a function $\vec \varphi_{\rm reg}\in \mathcal{S}(\R^{3})^{2}$ so that $\vec \varphi= \vec \varphi_{\rm reg}+\vec \varphi_{\rm rem1}$ and $\nor{\vec \varphi_{rem}}{\dot{H}^1\times L^2(\R^3)}\leq \e$. Now, using Lemma \ref{lm:decphonontrap} for $\vec \varphi_{\rm reg}$, we can decompose $\vec \varphi_{\rm reg}=\vec \varphi_{\rm trap}+\vec \varphi_{\rm nontrap}$ with $\nor{\vec \varphi_{\rm trap}}{\dot{H}^{1}\times L^{2}(\R^{3})}\leq \e$ and $\vec \varphi_{\rm nontrap}\in \mathcal{S}(\R^{3})^{2}$  is a non-trapped data. We have therefore $\vec \varphi= \vec \varphi_{\rm nontrap}+\vec \varphi_{\rm rem2}$ with $\nor{\vec \varphi_{\rm rem2}}{\dot{H}^1\times L^2(\R^3)}\leq 2\e$. Using this decomposition for the linear concentrating wave, with  the obvious notations, we decompose $v_n=v_{n, {rem2}}+v_{n, {\rm nontrap}}$. By Sobolev embedding and conservation of the energy, we can write 
\begin{align*}
&\nor{v_{n, {\rm rem2}}}{L^{\infty}(\R,L^6(\Omega))}\lesssim \nor{v_{n,  {\rm rem2}}}{L^{\infty}(\R,\dot{H}^1(\Omega))}\lesssim \nor{\vec v_{n, {\rm rem2}}(0)}{\mathcal{H}(\Omega)}\\
&\quad\lesssim \nor{ T_{\OO, k}^{\Omega}\vec \varphi_{ {\rm rem2}}}{\mathcal{H}(\Omega)}\lesssim \nor{ T_{\OO, k}\vec \varphi_{ {\rm rem2}}}{\mathcal{H}(\R^3)}\lesssim\nor{ \vec \varphi_{ {\rm rem2}}}{\mathcal{H}(\R^3)}\leq \e,
\end{align*} and thus
\begin{align*}
\nor{v_n}{L^{\infty}(I_n,L^6(\Omega))}&\leq \nor{v_{n,  {\rm rem2}}}{L^{\infty}(I_n,L^6(\Omega))}+\nor{v_{n, {\rm nontrap}}}{L^{\infty}(I_n,L^6(\Omega))}\\
&\leq  {C} \e+\nor{v_{n, {\rm nontrap}}}{L^{\infty}(I_n,L^6(\Omega))}.
\end{align*}
 \end{proof}  
 
\subsection{Boundary measure}
We will use semiclassical microlocal defect measure or Wigner measure. It was defined independently by  G\'erard \cite{G:90} and Lions-Paul \cite{LP:93} (see also G\'erard \cite{G:91} Tartar \cite{T:90} for classical microlocal defect measures). We refer to the the review article \cite{B:97} for more precisions and historical remarks, and to \cite{B:97b} for an introduction to semiclassical measures.
We here follow the approach of \cite{B:97}. 

\subsubsection{Notations} We begin by gathering  the notations used throughout this subsection. We first recall the notations from \cite[Section 3.2]{B:97}.
\begin{itemize}
    \item $M:=\R_t \times \Omega$, with points denoted $z=(t,x)\in M$ (here we differ from the notation of \cite{B:97} for coherence). 
    \item For $\vec u_{n}\in C(\R,\mathcal{H})$, we denote $\underline{{u_{n}}}$ the extension by zero of ${u_{n}}$ to $\R^{1+3}$.
    \item Denote $^bTM$ the bundle of rank $3+1$ whose sections are the vector fields tangent to $\partial M$, by $^bT^*M$ the dual bundle (Melrose's compressed cotangent bundle), and by $j:T^*M\rightarrow ^bT^*M$ the canonical map (restriction dual of the embedding $^bTM\hookrightarrow TM$). Abusing slightly notations, we also denote $j$ the canonical projection $T^*\Omega \to ^b T^* \Omega $.
    \item Let $^b\pi_{\tau,x,\xi}$ be the projection from $^bT^*M$ suppressing the component in $t$: $^b\pi_{\tau,x,\xi}(t,\tau,x,\xi)=(\tau,x,\xi)$. Similarly, we define $^b\pi_{t,x}$, and $\pi_{t,x}$ the corresponding projection from $T^* M$.
    \item We denote by $\textnormal{Char}(\widetilde{P})$ the characteristic set of the wave operator $\widetilde{P}=\partial_t^{2}-\Delta$ and denote by $Z$ its projection
\begin{align*}
\textnormal{Char}(\widetilde{P})&=\left\{(z,\eta)=(t,x,\tau,\xi)\in T^*\R^{1+3}_{\left|\overline{M}\right.};p(x,\xi)=\tau^{2}\right\},\\
Z&=j(\textnormal{Char}(\widetilde{P})),\\
Y&=^b\pi_{\tau,x,\xi}(Z), 
\end{align*}
where $p(x,\xi)$ is the principal symbol of $-\Delta$ in the chosen local coordinates.

\item For $a\in C^\infty_c(T^{*}\R^{1+3})$, we define the associated semiclassical pseudo-differential operator acting at scale $h_n$ by (with $z=(t,x)$)
\begin{align*}
\operatorname{Op}(a)(z,h_{n}D_z)f&:=\frac{1}{(2\pi)^{1+3} }\iint e^{i(z-s)\cdot \zeta}a(z,h_{n}\zeta)f(s)~d\zeta ds.
\end{align*}
\end{itemize}

\noindent Three distinct bicharacteristic flows appear in the statements and proofs below. We refer for example to \cite[Section 3.2]{B:97b} for precise definitions, and recap them here:
\begin{itemize}
    \item $\Phi$ is the generalized flow of $\partial_t^2 - \Delta$ on $^bT^*M$.
    \item $\varphi$ is the generalized flow of $ - \Delta$ on $^bT^*\Omega$. In particular,
    $$
        \Phi_s(t, \tau, x, \xi) = (t-2\tau s, \tau, \varphi_s(x, \xi)).
    $$
    \item $\widetilde{\varphi}$ is the generalized flow of $ - \Delta$ on $^bT^*\Omega$ reparametrized at speed one:
    $$
        \varphi_s(x, \xi) = \widetilde{\varphi}_{2s|\xi|}(x, \xi).
    $$
\end{itemize}
\noindent Finally, we recall the definition of the push-forward measure:
\begin{itemize}
\item  For two measurable spaces $(X_1, \Sigma_1)$, $(X_2, \Sigma_2)$, $\mu \in \mathcal M(\Sigma_1)$ and $f:X_1\to X_2$, the push-forward measure $f_* \mu \in \mathcal M(\Sigma_2)$ is defined by
$$
\langle f_* \mu, a \rangle = \langle \mu, a \circ f \rangle,
$$
for all measurable function $a$ defined on $X_2$.
\end{itemize} 

\subsubsection{Definition and properties of the measure}
The following lemma describes the behavior of the solution for some semiclassical times of the order of $h_n$ (in the case of the exterior of a convex obstacle, this corresponds to Lemma 2.3.3 of \cite{GG}).

\begin{prop}
\label{propmeasure}
Let $\vec v_n$ be a linear concentrating profile associated to the initial data $f$ and $\OO$ as above. 
Then, the following holds.
\begin{enumerate}
    \item[\textbf{(1) Existence of the measure.}] Up to a subsequence, there exists a positive Radon measure $\mu$ on $T^{*}\R^{1+3}$ such that 
\begin{equation*}
\forall a\in C^\infty_c(T^{*}\R^{1+3}), \quad \limvar{n}{\infty} \left(\operatorname{Op}(a)(t,x,h_nD_{t,x})\underline{\partial_t v_n},\underline{\partial_t v_n}\right)_{L^2} + \left(\operatorname{Op}(a)(t,x,h_nD_{t,x})\underline{\nabla v_n},\underline{\nabla v_n}\right)_{L^2}=\left\langle \mu, a\right\rangle.
\end{equation*}
   \item[\textbf{(2) Invariance.}]
   $\mu$ is invariant by the generalized bicharacteristic flow $\Phi$ of $\partial_t^2 - \Delta$, that is, for $m\in C^{0}_c(T^{*}\R^{1+3})$
\begin{align}
\label{propagflow}
\left\langle \mu, m\circ\Phi_s\right\rangle= \left\langle \mu, m\right\rangle=\left\langle \mu(t,\tau,x,\xi), m(t-2\tau s,\tau, \varphi_s(x,\xi)\right\rangle
\end{align}
where $\varphi_s$ is the generalized flow of $-\Delta$ on $^bT^*\Omega$ (note that it is well defined because $\Phi_{s}$ is well defined $\mu$-almost everywhere).

   \item[\textbf{(3) Trace measure.}] In particular, $j_{*}\mu$ (seen as a measure with value in $\mathcal{M}(Z)$) is continuous in time. 
We can define the trace of the measure at time $t$: there exists $\mu_t\in \mathcal{M}(Y)$  such that
$$
j_{*}\mu=dt\otimes \mu_t
$$
and $\mu_t=(\operatorname{Id}_{\tau}, \widetilde \varphi_{(\operatorname{sgn}\tau)t})_* \mu_0$, that is for $a\in C^{0}_c(Y)$
\begin{align} \label{eq:prop_repar}
\left\langle \mu_t, a\circ(\operatorname{Id}_{\tau}, \widetilde \varphi_{(\operatorname{sgn}\tau)t})\right\rangle= \left\langle \mu_0, a\right\rangle=\left\langle \mu_t(\tau,x,\xi),a(\tau, \widetilde \varphi_{(\operatorname{sgn}\tau)t}(x,\xi)\right\rangle,
\end{align}
where $\widetilde \varphi_t$ is the generalized bicharacteristic flow at speed one.
   \item[\textbf{(4) Initial measure.}] In appropriate local coordinates, 
the initial measure $\mu_0$ satisfies
\begin{itemize}
\item if $x_{\infty}\in \Omega$:
\begin{equation}\label{eq:mes_data_int}
\mu_0(\tau,x,\xi)=\frac{1}{2}\sum_{\pm}\delta_{x=x_{\infty}}\otimes \delta_{\tau=\pm |\xi|} \left|\widehat{\varphi^{0}}(\xi)\pm |\xi|\widehat{\varphi^{1}}(\xi)\right|^2d\xi,\\
\end{equation}
\item if $x_{\infty}\in \partial \Omega$: 
\begin{equation} \label{eq:mes_data_bord}
\mu_0 = j_*(\delta_{x_\infty}\otimes \lambda_0), \quad \lambda_0(\xi,\tau)=\sum_{\pm} \delta(\tau \mp |\xi|)H_{\pm}(\xi)d\xi,
\end{equation}
where $H_{\pm}(\xi)\in L^1(\R^3_{\xi})$ are functions depending on $\vec \varphi$ and $\OO$.
\end{itemize}
In addition, in both cases, if $\vec \varphi$ is nontrapping  (Definition \ref{ass:nontrapdata}), then, for any $R>0$ there exists $T>0$ so that
\begin{equation}\label{eq:mes_echap}
s \geq T \implies \forall (x, \tau, \xi) \in \operatorname{supp}\mu_0, \; \widetilde \varphi_s(x, \xi) \in T^* B(0,R)^c.
\end{equation}
\end{enumerate}
\end{prop}
\begin{proof}
The proof is a combination of several ingredients that can be found in the literature. We explain where to find and how to combine these arguments, and give details when modifications or generalizations are needed.

\textbf{(1)--(2): Existence and propagation.} The existence and propagation properties are similar to Burq \cite{B:97}, with the addition of one argument originally written for microlocal defect measures in \cite{B:97b}. More precisely, \cite[Theorem 4]{B:97} (see also \cite[Theorem 15]{B:97b} in the context of microlocal defect measure) states that the measure satisfies 
\begin{align} \label{eq:inv_Hp}
 ^{t}H_{p}(\mu)=\int_{\rho\in \mathcal{H}\cap\mathcal{G}}\frac{\delta(\xi-\xi_{+}(\rho))-\delta(\xi-\xi_{-}(\rho))}{\langle\xi_{+}-\xi_{-},n(x(\rho)\rangle} \nu (d\rho)
\end{align}
where $\nu$ is the semiclassical measure of the normal derivative $(\partial_{\nu} u_{n})_{\left|\partial M\right.}$ at the boundary.  In the context of microlocal defect measures (not semiclassical), it is shown in \cite[Theorem 15 , iv)]{B:97b} that the analog of (\ref{eq:inv_Hp}) implies the invariance of the measure by the generalised bicharacteristic flow, i.e. the analog to \eqref{propagflow}. As remarked e.g in \cite[p.13]{B:02} for the stationary problem, the same arguments still hold verbatim for the semiclassical measure, hence (\ref{eq:inv_Hp}) gives \eqref{propagflow}.

\textbf{(3)--(4) Trace mesure and initial measure.}

Observe that, once the existence of the trace measure is established, if $x_\infty \in \Omega$, approximating the data by a compactly supported function, the semiclassical measure coincides in small times with the one in the free case, for which the result \eqref{eq:mes_data_int} is classical. We will therefore focus on the most difficult case $x_\infty \in \partial \Omega$.

\underline{(a) Existence of the trace measure, and initial measure for  $x_\infty \in \partial \Omega$, proof of \eqref{eq:mes_data_int}--\eqref{eq:mes_echap}.}
\newline The proof  is  
similar to the proof of  \cite[Lemma 2.3.5]{GG}; we sketch it stressing the necessary modifications to generalize it to our  setting. The proof uses a symmetrization argument in appropriate coordinates that does not use the specific expression of the propagator $\Phi_t$, denoted $G_t$ in \cite{GG}.

In appropriate geodesic normal coordinates, the boundary is flattened to $\{x_{1}=0\}$. By a symmetrization argument (formula (3.38) in \cite{GG}), we end up with a solution $v_{n}$ to the wave equation (formula (3.39) in \cite{GG}) in a domain with a Lipschitz metric that is smooth except at $\{x_{1}=0\}$, where there is a jump of  normal derivative. We denote $\nu$ the Wigner measure of $v_{n}$. Up to change of coordinates, $\nu$ coincides with $\mu$ in $\{x_{1}>0\}$, which corresponds to the interior of $\Omega$ in original coordinates. The same computations lead to the fact that $\nu$ is continuous in $t$. Moreover, formula  \cite[(3.55)]{GG} of transport equation for the measure still holds and did not use any convexity assumption. In particular, it is still possible to compute the initial condition for $\nu_{\left| t=0\right.}$ and to obtain formula \cite[(3.64)']{GG} that provides \eqref{eq:mes_data_bord}. Now, the proof differs slightly from \cite{GG} only by the fact that the property 
$$
\frac\xi\tau \cdot N > 0 \implies \forall t \geq 0, \hspace{0.3cm}\pi_x\varphi_t(x_\infty, \tau, \xi) \in \Omega
$$
(where $N = N(x_\infty)$ denotes the normal to the boundary) used in \cite{GG} to end the proof relies on the convexity assumption and is not true anymore in our more general setting. We detail the necessary modification as follows.

Let $\epsilon >0$. There exists $t_0(\epsilon)>0$ so that for any $\xi$ with $\frac\xi\tau \cdot N > \epsilon$ and any $t \in [0, t_0(\epsilon)]$,
$$
\pi_x\varphi_t(x_\infty, \tau, \xi) \in \Omega,
$$
therefore, applying \cite[Lemma 2.3.7]{GG} to 
$$
\nu^\epsilon := \nu 1_{\frac\xi\tau \cdot N > \epsilon}
$$
in times $[0, t_0(\epsilon)]$ in the complement of $\{ \tau = 0\}\cup \{ \frac\xi\tau \cdot N = 0\}$ gives, with the same proof,
\begin{equation} \label{eq:gen_GG0}
\nu(t) \geq \varphi_t(j_* (\delta_{x_\infty}\otimes \lambda_0 1_{\frac\xi\tau \cdot N > \epsilon})),
\end{equation}
where $\lambda_0$ is explicitly given by \cite[(3.64)']{GG} and {has total mass}
\begin{equation} \label{eq:gen_GG1}
\operatorname{TM}(\lambda_0)= \Vert P_{{X_{\mathcal{O}}}} \vec \psi \Vert^2_{\mathcal H} = \lim_{n \rightarrow \infty} \Vert \vec u_n(0) \Vert^2_{\mathcal H}.
\end{equation} 
Using the fact that  that $\mu = \nu$ in $\Omega$, \eqref{eq:gen_GG0} implies
$$
\forall \epsilon>0, \; \forall t \in [0, t_0(\epsilon)], \hspace{0.3cm} \mu \geq \varphi_t(j_*(\delta_{x_\infty}\otimes \lambda_0 1_{\frac\xi\tau \cdot N > \epsilon})) \otimes dt.
$$
As both these measures are invariant by the bicharacteristic flow, we therefore have
$$
\forall \epsilon>0, \; \forall t \in [0, T], \hspace{0.3cm} \mu \geq \varphi_t(j_*(\delta_{x_\infty}\otimes \lambda_0 1_{\frac\xi\tau \cdot N > \epsilon})) \otimes dt,
$$
and hence
\begin{equation} \label{eq:gen_GG2}
\mu \geq \varphi_t(j_*(\delta_{x_\infty}\otimes \lambda_0)) \otimes dt.
\end{equation}
On the other hand, conservation of energy together with \eqref{eq:gen_GG1} implies that both measures in \eqref{eq:gen_GG2} have the same total mass. It follows that \eqref{eq:gen_GG2} is in fact an equality,
\begin{equation}
\label{propagmut}
\mu = \varphi_t(j_*(\delta_{x_\infty}\otimes \lambda_0)) \otimes dt = \varphi_t \mu_0 \otimes dt, \quad \mu_0 := j_*(\delta_{x_\infty}\otimes \lambda_0),
\end{equation}
 so that the measure $\varphi_t(j_{*}(\delta_{x_\infty}\otimes \lambda_0))$ is continuous in time. The computation of $\lambda_0$ comes from the formula \cite[(3.64')]{GG}.
 It gives 
 $$
\lambda_0(\xi,\tau)=\sum_{\pm} \delta(\tau \mp |\xi|)H_{\pm}(\xi)d\xi,
$$ 
with $H$ defined as follows
\begin{equation}
H_{\pm}(\xi)=
\begin{cases}
\frac{1}{2}\left(\left|\hat{\psi}(\xi)\pm i|\xi|\hat{\varphi} (\xi)\right|^2 +\left|\hat{\psi}(R(\xi))\pm i|\xi|\hat{\varphi}(R(\xi))\right|^2 \right) & \textnormal{ if }X_{\OO}=\R^3,\\
\frac{1}{2}\left|\widehat{ \mathbb{1}_{X_{\OO}}\psi}(\xi)-\widehat{\mathbb{1}_{X_{\OO}}\psi}(R(\xi))e^{2i\alpha \xi\cdot \vec n}\right| \\
\left.
\qquad \pm i|\xi| \left(\widehat{ \mathcal{P}_{X_{\OO}}\varphi}(\xi)-\widehat{\mathcal{P}_{X_{\OO}}\varphi}(R(\xi))e^{2i\alpha \xi\cdot \vec n} \right) \right|^2 & \textnormal{ if }X_{\OO}=\left\{\xi\cdot \vec n(x_{\infty})>\alpha\right\},\end{cases}
\end{equation}
 where $\vec n=\vec n(x_{\infty})$ is the normal vector and $R$ is the symmetry with respect to $T_{x_\infty}(\partial \Omega)$, defined by
 $$
R(\xi) := \xi - 2(\xi \cdot \vec n(x_\infty))\vec n(x_\infty).
 $$
The fact that $H_\pm \in L^1$ comes directly from the fact that $\vec \varphi = (\varphi, \psi) \in \dot H^1 \times L^2$.
 Finally,
 $$
    {}^b\pi_{x,\xi} \operatorname{supp} \mu_0  \subset  j_* \operatorname{supp}(\delta_{x_\infty} \otimes \lambda_0) \subset j_*\big(\{x_\infty\}\times(\operatorname{supp}(\widehat \varphi) \cup \operatorname{supp}(\widehat \varphi\circ  R)\cup \operatorname{supp}(\widehat \psi) \cup \operatorname{supp}(\widehat \psi\circ  R)\big), 
 $$
 but, as $j(x_\infty, R(\xi)) = j(x_\infty, \xi)$, this gives
 $$
{}^b\pi_{x,\xi} \operatorname{supp} \mu_0  \subset
j_*\big(\{x_\infty\}\times(\operatorname{supp}(\widehat \varphi))\cup \operatorname{supp}(\widehat \psi)\big),
 $$
 and (\ref{eq:mes_echap}) follows by definition of a non-trapped data (Definition \ref{ass:nontrapdata}).

\underline{(b) Reparametrizing the flow to obtain (\ref{eq:prop_repar}).}  It only remains to show (\ref{eq:prop_repar}), and hence to make the link between $\varphi_{t}$ and $\widetilde{\varphi}_{t}$. To do so, we reparametrize the flow. Notice that 
$$
\Phi_s(t, \tau, x, \xi) = (t-2\tau s, \tau, \varphi_s(x, \xi))
= (t-2\tau s, \tau, \widetilde \varphi_{2s|\xi|}(x, \xi)),
$$
hence, as $\mu$ is supported in $\{ \tau^2 = |\xi|^2\}$
$$
\Phi_s(t, \tau, x, \xi) = (t-2\tau s, \tau, \widetilde \varphi_{2s|\tau|}(x, \xi)) \quad \text{ on }\operatorname{supp}\mu.
$$
For any $c<d$ and $0<a<b$, $\Phi_s$ is well defined for $(s,\tau,t,(x,\xi))\in [c,d]\times [a,b]\times \R\times ^bT^*\Omega$ to $[a,b]\times \R\times ^bT^*\Omega$ and we can check that it is proper. Then, applying Lemma \ref{lmrepar}, as by our previous steps, $\mu(\{\tau =0 \})=0$, $\mu$ is also invariant by
$$
\widetilde{\Phi}_{s}(t, \tau, x, \xi)=\Phi_{\frac{s}{2\tau}}(t, \tau, x, \xi) = (t- s, \tau, \widetilde{\varphi}_{(\operatorname{sgn}\tau) s}(x, \xi)),
$$ 
hence
\begin{equation*}
    \left<\mu,a\circ \Phi_{\frac{s}{2\tau}}\right>=\int_t a\left(\operatorname{Id}, \widetilde{\varphi}_{(\operatorname{sgn}\tau)s}(x,\xi)\right),
\end{equation*}
and (\ref{eq:prop_repar}) follows.
\end{proof}

\subsubsection{From properties of the measure to properties of the concentrating wave}
We define 
$$
e_{n}(t):=(\partial_{t}u_{n}(t))^{2}+(\nabla u_{n}(t))^{2}dx
$$
the local density energy which satisfies the equation on $\Omega$
\begin{align*}
\partial_{t}e_{n}=2\operatorname{div}_{x}(\partial_{t} u_{n}\nabla_{x}u_{n}).
\end{align*}
Notice that since $\partial_{t}u_{n}=0$ on $\R\times \partial \Omega$, the previous equation still holds in $\R\times \R^{3}$  if we extend $e_{n}$ and $u_{n}$ by $0$ outside $\Omega$. As a consequence, $e_{n}$ is weakly-* equicontinuous as a measure-valued function of $t$ and
the limit points $e_{\infty}$ of $e_{n}$ are weakly-* continuous measure-valued functions of $t$.

\begin{lem}
\label{lemmapropmeasure}
We have the following formula for the local limit energy density
\begin{align}
\label{formuleeinfty}
 e_{\infty}(t)dt=\int_{\R_{\tau}\times \R^{3}_{\xi}}\mu(t,x,d\tau,d\xi) =
 \int_{\tau, \xi}j_*\mu(t,x,d\tau,d\xi),
\end{align}
in the sense that
$$
 e_{\infty}(t)dt= (\pi_{t,x})_* \mu =  ({}^b\pi_{t,x})_* j_* \mu,
$$
where $\mu$ is described by Proposition \ref{propmeasure}. In addition, it satisfies the following properties:
\begin{enumerate}
\item \label{enumescape}(Escape property) Assume that the data is non-trapping (Definition \ref{ass:nontrapdata}). Then, for any $\varphi \in C^{\infty}_0(B(0,R))$, we have $\varphi(x) e_{\infty}(t)=0$ for $t\geq T_{\rm esc}(R)$ where $T_{\rm esc}(R)\in \R$ is defined in Definition \ref{ass:nontrapdata}.
\item \label{finitespeed} (Finite speed) For any $t\in \R$, $ e_{\infty}(t)$ is supported in $B(0,\vert t\vert+\vert x_{0}\vert+1)$.  
\item \label{enumnoncon} (Non reconcentration) Assume Assumption \ref{ass:nonreco}. Then, for any $t\neq 0$ and any $y_0\in\R^3$, $ e_{\infty}(t)\left[\left\{y_0\right\}\right]=0$.
\end{enumerate}
\end{lem}

\begin{proof}
We have classically
\begin{align*}
\int_{\R_{\tau}\times \R^{3}_{\xi}}\mu(t,x,d\tau,d\xi)\leq e_{\infty}(t)dt.
\end{align*}
Moreover, since \eqref{eq:gen_GG2} has been proved to be an equality, we have, with $\operatorname{TM}$ denoting the total mass,
\begin{align*}
\operatorname{TM}\left(\int_{\R_{\tau}\times \R^{3}_{\xi}}\mu(t,x,d\tau,d\xi)\right)=\operatorname{TM}(\lambda_0)=\Vert P_{\Omega_\infty} \vec \psi \Vert^2_{\mathcal H} = \lim_{n \rightarrow \infty} \Vert \vec u_n(0) \Vert^2_{\mathcal H}
\end{align*}
where we have used \eqref{eq:gen_GG1}. Concerning $e_{\infty}$, we compute
\begin{align*}
\operatorname{TM}(e_{\infty}(t))\leq \limsup_{n \rightarrow \infty} TM(e_{n}(t))= \lim_{n \rightarrow \infty} \Vert \vec u_n(t) \Vert^2_{\mathcal H}=\Vert P_{\Omega_\infty} \vec \psi \Vert^2_{\mathcal H}.
\end{align*}
In particular, we have equality in all the previous inequalities and we have obtained the expected formula \eqref{formuleeinfty}.  The proof of the second equality is similar.

Now that formula \eqref{formuleeinfty} is obtained, the remaining Properties are direct consequences of the precise description of $\mu$ given by Proposition \ref{propmeasure}. 

More precisely, Property \ref{enumescape} follows from (\ref{eq:mes_echap}). 

Property \ref{finitespeed} follows from the semiclassical finite speed of propagation and the fact that $\mu_{0}$ is supported at $x_{\infty}$.

For Property \ref{enumnoncon}, observe that, using (\ref{eq:prop_repar})
\begin{align*}
e_\infty(t) \otimes dt &= ({}^b\pi_{t,x})_* j_* \mu = ({}^b\pi_{t,x})_* (\mu_t \otimes dt)
= [({}^b\pi_{x})_*\mu_t]\otimes dt \\
&= [({}^b\pi_{x})_*(\operatorname{Id}_{\tau},\widetilde \varphi_{(\operatorname{sgn}\tau)t})_*(\delta_{x_\infty}\otimes\lambda_0)]\otimes dt,
\end{align*}
therefore
$$
e_\infty(t) = ({}^b\pi_{x})_*(\operatorname{Id}_{\tau},\widetilde\varphi_{(\operatorname{sgn}\tau)t})_* (\delta_{x_\infty}\otimes\lambda_0),
$$
and we have, using Proposition \ref{propmeasure}, for example in the case where $x_\infty \in \partial\Omega$:
\begin{align*}
 &e_{\infty}(t)\left[\left\{y_0\right\}\right]  =
({}^b\pi_{x})_*(\operatorname{Id}_{\tau},\varphi_{(\operatorname{sgn}\tau)t})_*(\delta_{x_\infty}\otimes\lambda_0)(\{ y_0 \}) \\
&\hspace{0.5cm}=  (\operatorname{Id}_{\tau},\varphi_{(\operatorname{sgn}\tau)t})^*(\delta_{x_\infty}\otimes \lambda_0)({}^b\pi_{x}^{-1}(\{ y_0 \} )) \\
&\hspace{0.5cm}= (\delta_{x_\infty}\otimes \lambda_0)\big((\operatorname{Id}_{\tau},\widetilde\varphi_{-(\operatorname{sgn}\tau)t})({}^b\pi_{x}^{-1}(\{ y_0 \} )) \big)\\
&\hspace{0.5cm}= \lambda_0\left\{ (\xi', \tau) \in {\mathbb R}^{2+1}, \;\text{s.t.} \;  (x_{\infty},\xi')\in \widetilde\varphi_{-(\operatorname{sgn}\tau)t}({}^b\pi_{x}^{-1}(\{ y_0 \} )) \right\}
 \\
&\hspace{0.5cm}= \lambda_0\big(\{ (\xi', \tau) \in {\mathbb R}^{2+1}, \;\text{s.t.} \;  {}^b\pi_{x}\widetilde\varphi_{(\operatorname{sgn}\tau)t}(x_\infty, \xi')=y_{0}\}\big) = 0,
\end{align*}
thanks to Assumption \ref{ass:nonreco}, where the last two identities are written in local coordinates near $x_\infty$.
\end{proof}

These three properties of the measure can be immediatly translated into three properties of the concentrating wave:
\begin{lem}
\label{lemmapropconcentrwave}
Let $v_n$ be a linear concentrating profile as above. Then, it satisfies the following properties:
\begin{enumerate}
\item \label{enumescapewave} For any $\chi \in C^{\infty}_0(B(0,R))$, we have for any $C\geq T_{\rm esc}(R)$,
$$
\nor{\chi(x)v_n}{L^{\infty}([T_{\rm esc}(R),C],\Hu)}\rightarrow 0.
$$
 where $T_{\rm esc}\in \R$ is defined in Assumption \ref{ass:nontrapdata}.
\item \label{finitespeedwave} For any $T>0$ there exists $\chi\in C^{\infty}_0(\R^3)$ such that we have
$$v_n=\chi(x) v_n+o(1)_{L^{\infty}([0,T],\Hu)}.$$
\item \label{enumnonconwave} For any $0<C_1<C_2$, we have
$$
\nor{v_n}{L^{\infty}([C_1,C_2],L^6)}\rightarrow 0.
$$
\end{enumerate}
\end{lem}
\begin{proof}
Property \ref{enumescapewave} is a direct consequence of Property \ref{enumescape} of Lemma \ref{lemmapropmeasure}, indeed
$$
\Vert \chi v_n(t)  \Vert_{\mathcal H} \lesssim \Vert \nabla \chi v_n(t) \Vert_{L^2} + \Vert  \chi \nabla v_n(t) \Vert_{L^2} + \Vert  \chi \partial_t v_n(t) \Vert_{L^2},
$$
but $v_n(t) \rightarrow 0$ in $L^2$ and 
$$
\Vert  \chi \nabla v_n(t) \Vert^2_{L^2} + \Vert  \chi \partial^2_t v_n(t) \Vert_{L^2} = \int \chi e_n(t) \rightarrow \int \chi e_\infty(t) = 0
$$
for $t \geq T_{\rm esc}$ by  Property \ref{enumescape} of Lemma \ref{lemmapropmeasure}.

Property \ref{finitespeedwave} is a direct consequence of the finite speed of propagation (Property \ref{finitespeed} of Lemma \ref{lemmapropmeasure}). 

For Property \ref{enumnonconwave}, we prove that for any sequence $t_n \rightarrow t\in [C_1,C_2]$, we have $\nor{v_n(t_n)}{L^6}\rightarrow0$. We will use the principle of concentration-compactness of P. L. Lions \cite[Lemma I.1]{L:85}. We can assume that $t_n\rightarrow t$ and therefore, the weak limit $\nu$ of $|\nabla v_n(t_n)|^2$ verifies $\nu \leq e_\infty(t)$. The previous Property implies the second assumption of  \cite[Lemma I.1]{L:85}. So, we can conclude using Property \ref{enumnoncon} of Lemma \ref{lemmapropmeasure}  that $\nu(\left\{x_0\right\})=0$ for any $t\neq 0$ and any $x_0\in \R^3$.
\end{proof}

\subsection{Escape from the obstacle}
The main purpose of this Subsection is to prove the following Proposition. Roughly speaking, it says that after a long enough time, the solution has escaped and its behavior is as in the Euclidian case:
\begin{prop}[Escape with the obstacle]
\label{propescapedom}
Let $\chi\in C^{\infty}_0(\R^3)$ such that $\chi(x)=1$ on a neighborhood of $\Omega^c$ (i.e. near the obstacle). Let $\vec v_n = \vec\varphi^\Omega _{\mathcal{O},n}$ be a linear concentrating profile associated with a non-trapped data (Definition \ref{ass:nontrapdata}). For any $\e>0$, there exists $C>0$, such that 
$$
 \limsup_{n \rightarrow +\infty} \left\|\chi(x) \vec{v_n} \right\|_{L^\infty([C,+\infty[,\mathcal{H})}\leq \e.
$$
Moreover, if $\vec w_n$ is the Euclidian solution of
\bneq
\partial_t^2 w_n - \Delta w_n &=&0, \quad (t,x)\in \R\times \R^3\\
w_n(C)&=&v_n(C),
\eneq 
then, we have
$$
\limsup_{n \rightarrow +\infty} \nor{\vec v_n-\vec w_n}{L^{\infty}([C,+\infty[,\mathcal{H})} \leq \e.
$$

 \end{prop}
 \begin{remark}
One can take $C=T_{\rm esc}(R)$ (where $\operatorname{supp} \chi \subset B(0,R)$) given by Assumption \ref{ass:nontrapdata}. 
\end{remark}
The idea of the proof will be to distinguish two time periods on which we have to prove that the Euclidian solution and the solution on $\Omega$ are close:
\begin{itemize}
\item During \emph{semiclassical times} $[T_{\rm esc}(R), C]$, the semiclassical description given by the measure allows to show that the solution is away from the obstacle and propagates freeely, as in the Euclidian case. This is the object of Proposition \ref{propescapedsc}, which mainly relies on the non-trapping property and its consequences in term of measures developed in the previous subsection.  
\item In \emph{classical times} $[C, +\infty[$, we use Huygens principle, stated as Lemma \ref{lmeuHuygens} below, to prove that the Euclidian solution with same initial data at time $T_{\rm esc}(R)$ is far from the obstacle in a time $[C, +\infty[$ with $C$ large and is therefore close to be a solution with Dirichlet condition during the time $[C, +\infty[$. Since, we have already shown in the first semiclassical step that the Euclidian and Dirichlet solutions are close on $[T_{\rm esc}(R), C]$, the two solutions are actually close on  $[T_{\rm esc}(R),+\infty[$.
\end{itemize}
We first show the following lemma, which shows that if the solution is localized far from the obstacle, then the Dirichlet and Euclidian solutions are close.
\begin{lem}
\label{solproche}
Let $\chi\in C^{\infty}_0(\R^3)$ such that $\chi(x)=1$ on a neighborhood of $\Omega^c$ and $\widetilde{\chi}\in C^{\infty}_0(\R^3)$ such that $\widetilde{\chi}(x)=1$ on the support of $\nabla \chi$.
Let $\vec u_0\in \mathcal{H}$. Denote $u=\SF \vec u_0$ and $v=\SO \vec u_0$. Then, we have 
$$
\Vert \vec u-\vec v\Vert_{L^{\infty}(I,\mathcal{H})}\lesssim \nor{\chi(x)\vec{u_{0}}}{\Hu}+\min \left(\nor{\chi(x)\vec v}{L^1(I,H^1)},\nor{\widetilde{\chi}(x)\vec u}{L^1(I,H^1)}\right).
$$
with some constant independent on the interval $I$.
\end{lem}
\begin{proof}
 Let $w=(1-\chi) v$. It is solution in the distributional sense of 
\begin{eqnarray*}
\left\lbrace
\begin{array}{rcl}
\partial_t^2 w-\Delta w&=&f\quad \textnormal{on}\quad \R\times \R^3\\
(w,\partial_t w)(0)&=&(1-\chi(x))\vec u_{0}
\end{array}
\right.
\end{eqnarray*}
with $f=2\nabla \chi(x)\cdot \nabla v+\Delta \chi(x) v$. We denote $F=(0,f)$.
Therefore, we have $w(t)= \SF(t) (1-\chi(x))\vec u_{0}+ \int_{0}^{t} \SF(t-s)F(s)ds$. Moreover, since $w(t,x)=0$ near $\Omega^{c}$, the same formula is true with the Dirichlet semi group. So, we obtain the formula 
$$
\left[ \SF(t)- \SO(t)\right]\vec u_{0}=\left[ \SF(t)- \SO(t)\right]\chi \vec u_{0}-\int_{0}^{t} \left[ \SF(t-s)- \SO(t-s)\right]F(s)ds.
$$
This gives the result with $\nor{\widetilde{\chi}(x)v}{L^1(I,H^1)}$ after noticing that $\operatorname{supp}f \subset \operatorname{supp}\nabla \chi \subset \{ \widetilde \chi = 1\}$. The same proof with $s=(1-\chi) u$ gives the result, with right hand side $\nor{\widetilde{\chi}(x)u}{L^1(I,H^1)}$ instead of $\nor{\widetilde{\chi}(x)v}{L^1(I,H^1)}$.
\end{proof}
\begin{prop}[Escape in large time]
\label{propescapedsc}
Let $\chi\in  C^{\infty}_0(B(0,R))$ such that $\chi(x)=1$ on a neighborhood of $\Omega^c$ (near the obstacle). Let $\vec v_n = \SO(\cdot) \vec\varphi^\Omega _{\mathcal{O},n}$ be a linear concentrating profile with $\vec \varphi$ a non-trapped data (Definition \ref{ass:nontrapdata}) and $T_{\rm esc}(R)$ as in Definition \ref{ass:nontrapdata}. For any $C>0$, we have 
$$
 \left\|\chi(x)\vec v_n\right\|_{L^\infty([T_{\rm esc},C],\Hu)}\tend{n}{+\infty} 0.
$$
Moreover, if $\vec w_n(t)=\SF(t-T_{\rm esc})\vec v_n(T_{\rm esc})$ is the euclidian solution of
\bneq
\partial_t^{2} w_n - \Delta w_n &=&0, \quad (t,x)\in \R\times \R^3\\
\vec w_n(T_{\rm esc})&=&\vec v_n(T_{\rm esc}),
\eneq 
then, we have
\begin{eqnarray}
\label{closeeucl}
\nor{\vec v_n-\vec w_n}{L^{\infty}([T_{\rm esc},C],\mathcal{H})}\tend{n}{+\infty}  0.
\end{eqnarray}
 \end{prop}

\begin{proof}[Proof of Proposition \ref{propescapedsc}]
Let $\e >0$ and  $\widetilde{\chi}\in C^{\infty}_0(\Omega\cap B(0,R))$ such that $\widetilde{\chi}=1$ on the support of $\nabla\chi$ (in particular, it is supported away from the boundary).
Using Property \ref{enumescapewave} of Lemma \ref{lemmapropconcentrwave}, we get that for any $C\geq T_{\rm esc}(R)$,
$$
\nor{\widetilde{\chi}(x)\vec v_n}{L^{\infty}([T_{\rm esc},C],\Hu)}\rightarrow 0.
$$
H\"older inequality in time gives
$$
\nor{\widetilde{\chi}(x)\vec{v_n}}{L^{1}([T_{\rm esc},C],\Hu)}\rightarrow 0,
$$
and we conclude using Lemma \ref{solproche} and the strong convergence of the $L^2$ norm: indeed, the initial data $ \vec\varphi^\Omega _{\mathcal{O},n}$ goes to zero in $L^2 \times H^{-1}(\Omega)$, and the $L^2 \times H^{-1}(\Omega)$ norm is conserved.
\end{proof}

We will also need the following Euclidian result.

\begin{lem}[Huygens principle]
\label{lmeuHuygens}
Assume $\vec u$ compactly supported in $B(0,R)$. Then, $\SF(t)\vec u$ is supported in $\{|x|\geq t-R\}$ for any $t\geq R$.

In particular, for $\chi, \widetilde{\chi}\in C^{\infty}_{c}(\R^{3})$, then, there exists $C>0$ and $A>0$ so that we have $$\nor{\chi \SF \vec u_{0}}{L^{\infty}(\R\setminus [-A,A],\mathcal{H})}\leq C \nor{(1-\widetilde{\chi})u_0}{\mathcal{H}}$$
and 
$$
\forall |t| \geq A, \hspace{0.3cm} \chi \SF(t)\widetilde{\chi} \vec u_0 = 0
$$
for any $\vec u_0\in \mathcal{H}$.
\end{lem}
\begin{proof}
The first statement is classical in odd dimension. We decompose $\vec u_0=\widetilde{\chi} \vec u_0+(1-\widetilde{\chi})\vec u_0$. The Huygens principle gives $A$ only depending on $\chi$, and $\widetilde{\chi}$ so that $\chi \SF(t)\widetilde{\chi} \vec u_0=0$ for $|t|\geq A$. Then, we write $\nor{\chi \SF(t)(1-\widetilde\chi) \vec u_0}{\mathcal{H}}\leq C\nor{ \SF(t)(1-\widetilde{\chi}) \vec u_0}{\mathcal{H}}\leq C\nor{ (1-\widetilde{\chi}) \vec u_0}{\mathcal{H}}$ uniformly for $t\in \R$.
\end{proof}
\begin{proof}[Proof of Proposition \ref{propescapedom}]
Let $R>0$ so that $\chi\in C^{\infty}(B(0,R))$. Let $T_{\rm esc}(R)$ given by Assumption \ref{ass:nontrapdata}. Let $\widetilde{\chi}$ given by Item \ref{finitespeedwave} of Lemma \ref{lemmapropconcentrwave} so that
 \begin{equation} \label{approx_supp_vn}
\nor{(1-\widetilde{\chi})v_n(T_{\rm esc})}{\mathcal{H}}\to 0.
\end{equation}

Fix $A$ as in Lemma \ref{lmeuHuygens} depending on $\chi$ and $\widetilde{\chi}$. Thanks to Proposition \ref{propescapedsc}, the conclusion holds on the interval $[T_{\rm esc},{\widetilde{C}}]$ with ${\widetilde{C}}=T_{\rm esc}+A$
\begin{equation}
\label{linsmcl}
\nor{\vec v_n-\SF(t-T_{\rm esc})\vec v_n(T_{\rm esc})}{L^{\infty}([T_{\rm esc},\widetilde{C}],\mathcal{H})}\leq \e.
\end{equation}
Now, by \eqref{approx_supp_vn}
$$
\Vert \widetilde \chi \vec v_n(T_{\rm esc})  - \vec v_n(T_{\rm esc}) \Vert_{\mathcal H} \leq \epsilon,
$$
hence by conservation of energy, for any $t$
\begin{align}
&\Vert \SO(t-T_{\rm esc})\widetilde \chi \vec v_n(T_{\rm esc}) - \vec v_n(t)\Vert_{ \mathcal H} \nonumber \\
& = \Vert \SO(t-T_{\rm esc})\widetilde \chi \vec v_n(T_{\rm esc}) - \SO(t-T_{\rm esc}) \vec v_n(T_{\rm esc})\Vert_{ \mathcal H} \nonumber \\
&= 
\Vert \widetilde \chi \vec v_n(T_{\rm esc}) -  \vec v_n(T_{\rm esc})\Vert_{ \mathcal H} \leq \epsilon. \label{eq:esc_o_supp_exact}
\end{align}
But, on the other hand, by Lemma \ref{lmeuHuygens} 
$$
\forall t \geq A, \hspace{0.3cm}
 \chi \SF(t)\widetilde \chi \vec v_n(T_{\rm esc}) = 0,
$$
from which it follows by Lemma \ref{solproche} that
$$
\forall t \geq A, \hspace{0.3cm} \SF(t)\widetilde \chi \vec v_n(T_{\rm esc})= \SO(t)\widetilde \chi \vec v_n(T_{\rm esc}),
$$
and after time change of variable
\begin{equation} \label{eq:esc_o_finite_speed}
\forall t \geq \widetilde{C}, \hspace{0.3cm} \SF(t-T_{\rm esc})\widetilde \chi \vec v_n(T_{\rm esc})= \SO(t-T_{\rm esc})\widetilde \chi \vec v_n(T_{\rm esc}).
\end{equation}
The above (\ref{eq:esc_o_finite_speed}) together with (\ref{eq:esc_o_supp_exact}) gives
$$
\Vert  \vec v_n - \SF(t-T_{\rm esc})\widetilde \chi \vec v_n(T_{\rm esc}) \Vert_{L^\infty([C, + \infty), \mathcal H)} \leq \epsilon.
$$
Finally, by conservation of energy 
for any $t$
\begin{align*}
&\Vert \SF(t-T_{\rm esc})\widetilde \chi \vec v_n(T_{\rm esc}) - \SF(t-T_{\rm esc}) \vec v_n(T_{\rm esc})\Vert_{ \mathcal H} \nonumber \\
&\leq 
\Vert \SF(t-T_{\rm esc})\widetilde \chi \vec v_n(T_{\rm esc}) - \SF(t-T_{\rm esc}) \vec v_n(T_{\rm esc})\Vert_{ \mathcal H(\mathbb R^3)} 
 \\
&= 
\Vert \widetilde \chi \vec v_n(T_{\rm esc}) -  \vec v_n(T_{\rm esc})\Vert_{ \mathcal H} \leq \epsilon. \label{eq:esc_o_supp_exact}
\end{align*}
The result follows {with $C=T_{\rm esc}$ by combining the previous estimates}.

\end{proof}

\subsection{Proof of the non-concentration Proposition \ref{propnoncon_loc}}

We will prove that for any sequence of times $t_n\geq C_n h_n$, we have $\nor{v_n(t_n)}{L^6}\rightarrow 0$. We distinguish three cases (up to extraction):
\begin{enumerate}
\item $t_n \rightarrow +\infty$ (escape at infinity)
\item $t_n\rightarrow c$ with $c\neq 0$ (semiclassical propagation)
\item  $t_n\rightarrow 0$ and $t_n\geq C_n h_n$ (locally Euclidian case)
\end{enumerate}
Each case will be studied in a separate subsubsection. Recall that, as shown in Lemma \ref{lm:enoughnonconc}, it is  enough to consider a non-trapped data (Definition \ref{ass:nontrapdata}).

\subsubsection{Escape at infinity: $t_n \rightarrow +\infty$}
Let $w_n$ be the Euclidian solution described in the Proposition \ref{propescapedom}. Using Sobolev embedding and (\ref{closeeucl}), it is enough to prove that $\nor{w_n(t_n)}{L^6}\rightarrow 0$. But, we also have by Property \ref{finitespeedwave} of Lemma \ref{lemmapropconcentrwave} that there exists $\chi\in C^{\infty}_0(\R^3)$ (related to the constant $C$ of Proposition \ref{propescapedom}) such that $w_n(C)=v_n(C)+\petito{1}_{\Hu}=\chi(x)v_n(C)+\petito{1}_{\Hu}=\chi(x)w_n(C)+\petito{1}_{\Hu}$.

The result is then a consequence of the following Euclidian Lemma applied to $z_n(\cdot):=w_n(\cdot-C)$. In the following, an $h_n$--oscillating sequence, $h_n \to 0$ is defined in \cite[Definition 2.2.1]{GG} and the equivalence with the usual Fourier oscillation is proven in \cite[Proposition 2.2.4]{GG}.
\begin{lem} 
Assume that $\vec f_n$ is bounded in $\mathcal{H}(\R^3)$ and $h_n$--oscillating and satisfies $\vec f_n=\chi(x)\vec f_n+\petito{1}_{\mathcal{H}(\R^3)}$ for some $\chi\in C^{\infty}_0(\R^3,[0,1])$. Denote $\vec u_n:=\SF \vec f_n$. Then, if ${t_n}\rightarrow +\infty$, we have $\nor{u_n(t_n)}{L^6}\rightarrow 0$.
\end{lem}
\begin{proof}
 We use the profile decomposition for the wave equation in $\mathbb R^3$, due to Bahouri-G\'erard \cite[Lemma 3.6]{BahouriGerard99}. Since $f_n$ is $h_n$--oscillating, we can easily impose (up to a fix dilation of the profile) that all the profiles have the same $h_n^j=h_n$. So, for any $l\in \N$, we have the following decomposition, up to a subsequence,
$$
u_n(t,x)=\sum_{j=1}^l \frac{1}{\sqrt{h_n}}V^j\left(\frac{t-t_n^j}{h_n},\frac{x-x_n^j}{h_n}\right)+w_n^l(t,x),
$$
with 
$$V^j=\SF \vec \psi^j, \hspace{0.3cm }\limsu{n}{\infty} \nor{w_n^l}{L^{\infty}(\R,L^6(\R^3))}\tend{l}{\infty}0.$$
In particular
\begin{equation} \label{eq:phiVj}
\forall j, \hspace{0.3cm} \chi(\cdot) \frac{1}{\sqrt{h_n}}\vec V^j\left(\frac{-t_n^j}{h_n},\frac{\cdot-x_n^j}{h_n}\right) = \frac{1}{\sqrt{h_n}}\vec V^j\left(\frac{-t_n^j}{h_n},\frac{\cdot-x_n^j}{h_n}\right) + o(1)_{\mathcal H},
\end{equation}
indeed, by orthogonality of the profiles and the support assumption, we have for any $l \geq 1$
\begin{equation*}
\Vert (1 - \chi) f_n \Vert^2_{\mathcal H(\R^3)}= \sum_{j=1}^l \Big\Vert \frac{1}{\sqrt{h_n}}(1-\chi(\cdot)) \vec V^j\left(\frac{-t_n^j}{h_n},\frac{\cdot-x_n^j}{h_n}\right) \Big\Vert^2_{\mathcal H(\R^3)} 
+ \Big\Vert  (1-\chi(\cdot)) \vec w_n^l(0,\cdot)\Big\Vert^2_{\mathcal H(\R^3)} + o(1)_{\mathcal H} 
 \to 0,
\end{equation*}
from which (\ref{eq:phiVj}) follows.

We now prove that we can select only the profiles such that 
$t_n^j$ is bounded.
Fix $j$ so that $t_n^j \rightarrow + \infty$.
We fix $\e>0$ and we select one radiation fields $F^j$ of Friedlander (see use \cite[Proposition 1.1]{CL} for the related statement) to be bounded and with compact support so that 
\begin{equation} \label{eq:profil_LaurentCote}
\Big\Vert\nabla V^j(t)-\frac{1}{|\cdot|}F^j\left(|\cdot|  + t, \frac{\cdot}{|\cdot|}\right)\frac{\cdot}{|\cdot|}\Big\Vert_{L^2(\R^3)}\leq \epsilon, \hspace{0.5cm}\forall t \leq -T_0.
\end{equation} 
Note that for a solution in the energy space, the Radiation field $F$ is usually in $L^2(\R\times \mathbb{S}^2)$ but an approximation allow to take it bounded and compactly supported up to a small loss in the energy space.
But
\begin{equation*}
\Big\Vert  {{h_n}}^{ - \frac 32} \chi(\cdot)\frac{1}{|\frac{\cdot-x_n^j}{h_n}|}F^j\left(\left|\frac{\cdot-x_n^j}{h_n}\right| - \frac{t_n^j}{h_n}, \frac{\frac{\cdot-x_n^j}{h_n}}{|\frac{\cdot-x_n^j}{h_n}|}\right)\Big\Vert_{L^2(\R^3)} 
= 
\Big\Vert  \chi(h_n \cdot + x_n^j)\frac{1}{|\cdot|}F^j\left(|\cdot| - \frac{t_n^j}{h_n}, \frac{\cdot}{|\cdot|}\right)\Big\Vert_{L^2(\R^3)}.
\end{equation*}
As $F^j$ and $\chi$ are compactly supported, there is $R_0, R_1 > 0$ so that, if the function integrated above is non zero, we have
$$
|y| \in \Big[\frac{t_n^j}{h_n} - R_0, \frac{t_n^j}{h_n} + R_0\Big] \text{ and } y \in B(\frac{-x_n^j}{h_n}, \frac{R_1}{h_n}).
$$
Thus, denoting
$$
K_n := B(\frac{-x_n^j}{h_n}, \frac{R_1}{h_n}) \cap \Big\{ |y| \in \big[\frac{t_n^j}{h_n} - R_0, \frac{t_n^j}{h_n} + R_0\big] \Big\},
$$
we have
\begin{equation} \label{eq:June11th}
\Big\Vert  {{h_n}}^{ - \frac 32} \chi(x)\frac{1}{|\frac{x-x_n^j}{h_n}|}F^j(|\frac{x-x_n^j}{h_n}| - \frac{t_n^j}{h_n}, \frac{\frac{x-x_n^j}{h_n}}{|\frac{x-x_n^j}{h_n}|})\Big\Vert_{L^2} \leq 
\Vert \frac{1}{|y|} \Vert_{L^2(K_n)}. 
\end{equation}
We then pass to a subsequence so that $\frac{|x_n^j|}{t_n^j} \to \ell  \in [0, +\infty]$.

\noindent \underline{$\triangleright$ First case: $\ell = 1$.} Since we have selected $j$ so that $t_n^j \rightarrow + \infty$, we have therefore $|x_n^j|\to +\infty$. We check that we can apply Lemma \ref{lm:volC} with $r=\frac{R_1}{h_n}$, $x_0=\frac{-x_n^j}{h_n}$, $R=R_0$, $t=\frac{t_n^j}{h_n}\to +\infty$ and $\epsilon=\max(R_1^{-1} |x_n^j|^{-1},R_0 h_n|x_n^j|^{-1})\lesssim |x_n^j|^{-1}\to 0$. Taking $n$ large enough so $\e\leq \e_0$ and $R_0\leq \epsilon_0 t$, so that the lemma applies, gives  
$$
|K_n| \lesssim  |x_n^j|^{-1}\left(\frac{t_n^j}{h_n}\right)^2\lesssim \frac{t_n^j}{h_n^2},
$$
where we have used the equivalence of $t_n^j$ and $|x_n^j|$. Now, (\ref{eq:June11th}) gives, using that $|y| \gtrsim \frac{t_n^j}{h_n}$ for $ y \in K_n$
$$
\nor{ {{h_n}}^{ - \frac 32} \chi(x)\frac{1}{|\frac{x-x_n^j}{h_n}|}F^j(|\frac{x-x_n^j}{h_n}| - \frac{t_n^j}{h_n}, \frac{\frac{x-x_n^j}{h_n}}{|\frac{x-x_n^j}{h_n}|})}{L^2(\R^3)} \lesssim 
 \frac{1}{\frac{t_n^j}{h_n}} \sqrt{\frac{t_n^j}{h_n^2}} = \frac{1}{(t_n^j)^{\frac 12}} \rightarrow 0,
$$
from which, going back to (\ref{eq:profil_LaurentCote}) and (\ref{eq:phiVj}),
\begin{equation}
\label{eq:June11th2}
\Big\Vert\frac{1}{\sqrt{h_n}}\vec V^j\left(\frac{-t_n^j}{h_n},\frac{\cdot-x_n^j}{h_n}\right) \Big\Vert_{\dot H^1(\mathbb R^3)} \leq 2\epsilon,
\end{equation}
for $n$ big enough.

\noindent \underline{$\triangleright$ Second case: $\ell \not= 1$.}
 If $\ell \not= 1$, then $K_n = \emptyset$ for $n$ big enough, thus (\ref{eq:June11th2}) holds as well. Indeed, if $y \in K_n$, we have from the definition of this set
$$
\max\big( \frac{t^j_n}{h_n}-R_0, \frac{|x_n|}{h_n}-\frac{R_1}{h_n}\big) \leq |y| \leq \min \Big( \frac{t^j_n}{h_n}+R_0, \frac{|x_n|}{h_n}+\frac{R_1}{h_n}\Big),
$$
from which if
$$
\frac{t^j_n}{h_n} - R_0 \geq \frac{|x_n|}{h_n}+\frac{R_1}{h_n} \hspace{0.3cm} \text{ or } \hspace{0.3cm} \frac{|x_n|}{h_n}- \frac{R_1}{h_n} \geq \frac{t^j_n}{h_n}+R_0,
$$
there can be no element in $K_n$. But this is equivalent to
$$
\frac{|x_n|}{t^j_n} \leq 1 - R_0\frac{h_n}{t^j_n} - R_1 \frac{1}{t^j_n} \hspace{0.3cm} \text{ or }\hspace{0.3cm}\frac{|x_n|}{t^j_n} \geq 1 + R_0\frac{h_n}{t^j_n} + R_1 \frac{1}{t^j_n},
$$
which is verified for $n$ big enough if $\ell \not= 1$ (recall that $t^j_n \to + \infty$ and $h_n \to 0$).
Arguing in the same way, the analog of (\ref{eq:June11th2}) holds for the component $\chi \partial_t V^j$. It follows that
(using Strichartz estimates to obtain the $o(1)_{L^\infty L^6}$ from $o(1)_\Hu$)
$$
\vec u_n=\sum_{\substack{1\leq j \leq l \\ t_{n}^j \text{ bounded }}} \frac{1}{\sqrt{h_n}} \vec V^j\left(\frac{t-t_n^j}{h_n},\frac{x-x_n^j}{h_n}\right)+ \vec w_n^l(t,x) + o(1)_{L^\infty_t L^6_x}.
$$

For the remaining profiles, so that $t_n^j$ is bounded, we compute
$$
\nor{\frac{1}{\sqrt{h_n}} \vec V^j\left(\frac{t_n-t_n^j}{h_n},\frac{x-x_n^j}{h_n}\right)}{L^6_x} = 
\nor{ \vec V^j\left(\frac{t_n-t_n^j}{h_n}, y \right)}{L^6_y} \rightarrow 0,
$$
using the fact that, as $t_n^j$ is bounded and $t_n \rightarrow + \infty$, 
$\frac{t_n-t_n^j}{h_n} \rightarrow + \infty$, and as $V^j$ is a fixed solution of the linear wave equation, we have $\nor{V^j(t)}{L^6}\tend{t}{+\infty} 0$. This ends the proof.
\end{proof}

\subsubsection{Semiclassical propagation: $t_n \rightarrow c >0$}
This is the Item \ref{enumnonconwave} of Lemma \ref{lemmapropconcentrwave}.

\subsubsection{Locally Euclidan case: $t_n \rightarrow 0$ with $t_n \geq C_n h_n$}
By finite speed of propagation, this is easy if the concentration point $x_{\infty}$ is far from the boundary. In the case $x_{\infty}\in\partial \Omega$, the proof of this part is exactly the same as the proof of Case (b) of Proposition 2.3.1 of Gallagher-G\'erard \cite{GG} which is a local argument that is valid in any geometry with regular enough boundary. Therefore, we only give an idea of the proof. The argument, detailed page 27, consists roughly in the following.
\begin{itemize}
\item They work in normal coordinates close to the boundary and extend the solution by symmetry. 
\item They rescale the concentrating wave so that the time $t_n$ corresponds to time $1$ (taking care of the boundary conditions).
\item In this scaling, the new solution satisfies a wave-type equation with piecewise smooth metric converging to the flat metric. The associated Wigner measure satisfies the same transport equation as the flat metric: that is, the energy travels along straight lines. In addition, the measure at time $0$ is concentrated only at the point $0$ (corresponding to $x_{\infty}$ in  normal coordinates).  The same argument of concentration-compactness as in Item \ref{enumnonconwave} of Lemma \ref{lemmapropconcentrwave} then allows to conclude.  
\end{itemize}

\subsection{Comparison between linear and non-linear concentrating profiles and the Euclidian flow}
We first describe the local behavior of the profiles that are asymptotically euclidian in a range of time of the form $[-Ch_n, Ch_n]$. The linear version is
\begin{lem}
\label{lmapproxconc}
Assume that $\mathcal{O}=\{(h_k,0,x_k)_k\}$ is a concentrating scale-core and $\vec \varphi\in \mathcal H(\mathbb R^3)$. Let $v_k=\SO\vec \varphi_{\mathcal{O},k} $ be a linear concentrating profile. Let $u=S_{X_{\mathcal{O}}} P_{X_{\mathcal{O}}}\vec\varphi $ extended to $\R^{3}$ by zero. 
Then
 $$
 v_k(t,x)= \frac{1}{\sqrt{h_k}}u\left(\frac{t}{h_k}, \frac{x-x_k}{h_k}\right)+o(1),
 $$
  where the $o(1)$ is a sequence converging to zero in $L^{\infty}([-Ch_k,Ch_k],\mathcal{H}(\R^3))$ for any $C>0$.
\end{lem}

\begin{proof}
This follows from Proposition 2.1.3 in \cite{GG}, with exactly the same proof since the argument is local close to the concentrating point. Indeed, one apply \cite[Proposition 2.1.3]{GG} together \cite[Lemma 2.1.2]{GG} with $(f_n^0, f_n^1) :=  P_{\Omega_n} \vec \varphi$. We get, denoting $X_{\mathcal O} := \Omega_\infty$
$$
S_{\Omega_n} P_{\Omega_n} \vec \varphi \rightarrow S_{\Omega_\infty} P_{\Omega_\infty}
\vec \varphi
$$
strongly in $L^\infty([-C, C], \mathcal H(\mathbb R^3))$, hence,
$$
T_{\mathcal O_n} S_{\Omega_n} P_{\Omega_n} \vec \varphi - T_{\mathcal O_n} S_{\Omega_\infty} P_{\Omega_\infty}
\vec \varphi \rightarrow 0
$$
strongly in $L^\infty([-Ch_n, Ch_n],  \mathcal H(\mathbb R^3))$
and, using the conjugation formulas
$$
T_{{\mathcal O}_n^{-1}} P_\Omega T_{\mathcal O_n} = P_{\Omega_n}, \hspace{0.3cm} 
T_{{\mathcal O}_n^{-1}} S_\Omega T_{\mathcal O_n} = S_{\Omega_n},
$$
this gives exactlty the result: indeed $T_{\mathcal O_n} S_{\Omega_n} P_{\Omega_n} \vec \varphi = T_{\mathcal O_n} S_{\Omega_n}T_{{\mathcal O}_n^{-1}} P_\Omega T_{\mathcal O_n} \vec \varphi = S_{\Omega}P_\Omega T_{\mathcal O_n} \vec \varphi = v_n$.
\end{proof}

Moreover, a similar result hold for the nonlinear equation as a consequence of the non-reconcentration Proposition \ref{propnoncon}.

\begin{prop}\label{prop:concprof}
Assume that $\mathcal{O}=({h_k},t_k,x_k)_k$ is a concentrating scale-core and $\vec \phi\in\mathcal H( \mathbb R^3)$. 

(i) For $k$ large enough (depending only on $\vec \phi$ and $\mathcal{O}$) 
there is a nonlinear solution $U_k\in L^5(\mathbb R, L^{10}(\Omega))$ of the equation \eqref{eq:NLWomega} with initial data $\vec \phi^\Omega _{\mathcal{O},k}$, and
\begin{equation}\label{ControlOnZNormForEP}
\Vert U_k\Vert_{L^5L^{10}}\lesssim_{\mathcal E_{\mathbb{R}^3}(\vec \phi)}1.
\end{equation}

(ii) There exists a Euclidean solution $U\in C(\mathbb{R}, \dot{H}^1_0(X_{\mathcal{O}}))$ of
\begin{equation}\label{EEq}
\left(\partial_t^{2}-\Delta_{X_{\mathcal{O}}}\right)U=U\vert U\vert^4
\end{equation}
with scattering data $\vec\phi^{\pm\infty} \in \mathcal H(X_{\mathcal O})$ defined as in  
\cite[(2.27) and Corollary 2.8]{BahouriGerard99} such that the following holds, up to a subsequence:
for any $\varepsilon>0$, there exists $T(\phi,\varepsilon)$ such that for all $T\ge T(\phi,\varepsilon)$ we have
\begin{equation}\label{ProxyEuclHyp}
\Vert U_k-\tilde{U}_k\Vert_{L^5 L^{10}(\{\vert t-t_k\vert\le T h_k\})}\le\varepsilon,
\end{equation}
for $k$ large enough, where
\begin{equation*}
\tilde{U}_k(t):=\lambda_k^{-\frac{1}{2}}U\left(\frac{t-t_k}{h_k},\frac{x-x_k}{h_k}\right).
\end{equation*}
In addition, up to a subsequence,
\begin{equation}\label{ScatEuclSol}
\Vert U_k(t)-\SO(t-t_{k}) T_{\mathcal{O},k}^{\Omega}\phi^{+\infty}\Vert_{L^\infty(t-t_k\geq Th_k, \Hu)}\le \varepsilon,
\end{equation}
and similarly with $\phi^{-\infty}$ in $L^\infty(t-t_k\leq - Th_k, \Hu)$, for $k$ large enough (depending on $\phi,\varepsilon,T$).
\end{prop}
\begin{proof}
The proof is exactly the same as in \cite{GG} once the non-reconcentration property of  Proposition \ref{propnoncon_loc} is at hand:  (\ref{ProxyEuclHyp}) is the estimate \cite[(1.10), p.38]{GG}, and (\ref{ScatEuclSol}) corresponds to \cite[(1.9)-(1.11) p. 39]{GG}. Therefore, we only only sketch the argument for the convenience of the reader.

In the case $t_{k}=0$:
\begin{itemize}
\item With a local argument analogous to Lemma \ref{lmapproxconc} for the nonlinear equation, we can prove that a nonlinear solution with a concentrating profile as initial data stays close to the solution of  \eqref{EEq} on each interval $[-Ch_{k},Ch_{k}]$ with $C$ as large as desired: this leads to \eqref{ProxyEuclHyp}.
\item The non-reconcentration property in $L^{\infty}L^6$ of Proposition \ref{propnoncon} that shows that after a time $Ch_{k}$ with large $C$, the solution with a concentrating initial data do not reconcentrate, together with a straightforward adaptation of the linearization theorem of \cite{G:96} to the case of an open domain, gives \eqref{ScatEuclSol}.
\end{itemize}

In the cases where $\frac{t_{k}}{h_{k}}\to\pm\infty$, the proof if slightly more involved. If for instance $\frac{t_{k}}{h_{k}}\to+\infty$, we construct the nonlinear solution $U$ on $X_{\mathcal{O}}$ that has for scattering data at $-\infty$ (see \cite[(2.27) and Corollary 2.8]{BahouriGerard99}) the profile $P_{X_{\mathcal{O}}}\vec \phi\in\mathcal H(X_{\mathcal{O}})$. More precisely, defining $v:=S_{X_{\mathcal{O}}}P_{X_{\mathcal{O}}}\vec \phi$, we have $\nor{\vec v(t)-\vec U(t)}{\mathcal{H}(X_{\mathcal{O})}}\underset{t\to -\infty}{\longrightarrow}0$.

We define the nonlinear solutions $U_{k}^{\rm app}:=\NLSO(t-t_{k}) \frac{1}{h_{k}}U\left(0,\frac{x-x_{k}}{h_{k}} \right)$. The ``flat solution''  $\frac{1}{h_{k}}U\left(\frac{t-t_{k}}{h_{k}},\frac{x-x_{k}}{h_{k}} \right)$  is a good approximation of $U_{k}^{\rm app}$ on some intervals $[t_{k}-Ch_{n},t_{k}+Ch_{k}]$ with $C$ as large as desired. If $C$ is large enough, $U(-C)$ is close to $v(-C)$. We consider the linear solution  $v_{k}^{\rm app}:=\SO(t-t_{k}) \frac{1}{h_{k}}\vec \phi \left(\frac{x-x_{k}}{h_{k}} \right)$. By Lemma \ref{lmapproxconc}, the ``flat solution''  $\frac{1}{h_{k}}v\left(\frac{t-t_{k}}{h_{k}},\frac{x-x_{k}}{h_{k}} \right)$  is a good approximation of $v_{k}^{\rm app}$ on some intervals $[t_{k}-Ch_{k},t_{k}+Ch_{k}]$ with $C$ as large as desired.

In particular, $v_{k}^{\rm app}(t_{k}-Ch_{k})$ is close to $U_{k}^{\rm app}(t_{k}-Ch_{k})$ for $C$ large enough.

Now, since we know that $\Vert v_{k}^{\rm app}\Vert_{L^{\infty}((-\infty,t_{k}-Ch_{k}],L^{6})}\leq \varepsilon$ for large $C$,  a straightforward adaptation of the linearization theorem of \cite{G:96} to the case of an open domain, gives $\Vert v_{k}^{\rm app}-U_{k}^{\rm app}\Vert_{L^{\infty}((-\infty,t_{k}-Ch_{k}],\mathcal{H})}\lesssim \varepsilon$.

For the asymptotic behavior of $U_{k}^{\rm app}$ after the concentration, we define $w :=S_{X_{\mathcal O}}\vec \phi^{+\infty}$, the linear solution so that $\nor{\vec w(t)-\vec U(t)}{\mathcal{H}(X_{\mathcal{O}})}\underset{t\to +\infty}{\longrightarrow}0$ given by the nonlinear scattering. We define similarly $w_{k}^{\rm app}:=\SO(t-t_{k}) \frac{1}{h_{k}}\vec \phi^{+\infty} \left(\frac{x-x_{k}}{h_{k}} \right)$ and get in the same way $\Vert w_{k}^{\rm app}-U_{k}^{\rm app}\Vert_{L^{\infty}([t_{k}Ch_{k}+\infty),\mathcal{H})}\lesssim \varepsilon$.

By construction, $U_{k}^{\rm app}$ satisfies \eqref{ScatEuclSol} (with $U_k$ replaced with $U_{k}^{\rm app}$) with $\vec \phi^{-\infty}=P_{X_{\mathcal{O}}}\vec \phi$ (on the interval $t-t_k\leq - T\lambda_k$) and $\vec\phi^{+\infty}$ defined above (on the interval $t-t_k\geq  T\lambda_k$); and \eqref{ProxyEuclHyp} (with $U_k$ replaced with $U_{k}^{\rm app}$). Moreover, we have $U_{k}^{\rm app}(0)=v_{k}^{\rm app}(0)+o(1)=\SO(-t_{k})\phi_{\mathcal{O},k}+o(1)$. In particular, $U_{k}$ can be globally defined and $U_{k}=U_{k}^{\rm app}+o(1)$ in $L^{\infty}(\R,\mathcal{H})$, and \eqref{ScatEuclSol}-\eqref{ProxyEuclHyp} follow.
\end{proof}

\section{Linear profile decomposition} \label{sec:lin_prof}

\begin{thm}[Linear profile decomposition] \label{th:lindec}
Assume that $\Omega$ has a smooth, compact boundary and verifies Assumptions \ref{ass:nonreco}, \ref{ass:weaktrap}, \ref{ass:strichartz}.
Let $\vec \varphi_n$ be a bounded sequence in $\mathcal{H}(\Omega)$. Then, there exist a family of profiles $\vec \psi^{(j)}$, and of orthogonal scale cores $\mathcal O ^{(j)}$ so that, up to a subsequence, for any $J\geq 0$
\begin{equation}\label{eq:lpd}
\vec \varphi_n = \sum_{j=1}^J \vec \psi^{(j)} _{\Omega, \mathcal{O}^{(j)},n} + w_n^{(J)},
\end{equation}
where the reminder satisfies the global-in-time decay
\begin{equation}\label{eq:lpd_rem}
\lim_{J \rightarrow \infty} \limsup_{n\rightarrow \infty} \; \Vert \SO(\cdot) \vec w_n^{\ell} \Vert_{L^\infty(\mathbb R, L^6(\Omega))} + \Vert \SO(\cdot) \vec w_n^{J}  \Vert_{L^5(\mathbb R, L^{10}(\Omega))} \; = 0.
\end{equation}
Moreover, the expansion verifies the following Pythagorean expansion: for any $J \geq 0$, as $n\rightarrow \infty$
\begin{equation}\label{eq:lpd_pyt}
\Vert \vec \varphi_n \Vert^2_{\mathcal H(\Omega)} = \sum_{j=1}^J \Vert \vec \psi^{(j)} _{\Omega, \mathcal{O}^{(j)},n} \Vert^2_{\mathcal H(\Omega)} + \Vert \vec w_n^{(J)}\Vert^2_{\mathcal H(\Omega)} + o_n(1),
\end{equation}
and its $L^6$ version
\begin{equation}\label{eq:lpd_L6}
\Vert \vec \varphi_n \Vert^6_{L^6(\Omega)} = \sum_{j=1}^J \Vert\vec \psi^{(j)} _{\Omega, \mathcal{O}^{(j)},n}\Vert^6_{L^6(\Omega)} + \Vert w_n^{(J)}\Vert^6_{ L^6(\Omega)} + o_n(1).
\end{equation}
\end{thm}

One of the main tool to prove the above is the following non concentration property. It is a direct consequence of the comparison results of \S\ref{sec:ass_prof} in the cases $\lambda_n \to +\infty$ or $x_n \to \infty$, and it is the main result of \S\ref{sec:conc_prof} in the case where $\lambda_n \to 0$, $x_n \to x_0$.  

\begin{prop}[Non concentration]
\label{propnoncon}
Let $(C_n)_{n\geq 1}$ be an arbitrary sequence converging to $+\infty$, and $v_n := \varphi_{\Omega,\mathcal{O},n}$ for an arbitrary scale-core $\mathcal{O} = \{ (\lambda_n)_{n\geq 1}, 0, (x_n)_{n\geq 1}  \}$.
Then, we have for $I_n=\R \setminus [-C_n \lambda_n,C_n \lambda_n]$, up to a subsequence
$$
\nor{v_n}{L^{\infty}(I_n,L^6(\Omega))}\rightarrow 0.
$$
\end{prop}
\begin{proof}
In the case where $\lambda_n \to 0$ and $x_n \to x_0 \in \overline \Omega$, this is Proposition \ref{propnoncon_loc}. In the case where
$\lambda_n \to \infty$ or $x_n \to \infty$, this is a consequence of Lemma \ref{lem:assfree} with $f_n=0$, together with Sobolev embedding and the similar result for the wave equation in $\mathbb R^3$. In the case where $\lambda_n \to \lambda_0 \neq 0$ and 
 $x_n \to x_0 \in \overline \Omega$, this comes from the fact that $\Vert u(t) \Vert_{L^6} \to 0$ as $t \to \pm \infty$ for any $u$ solution to the linear wave equation in $\Omega$ (this comes from example from linear scattering Lemma \ref{lem:lin_scat}, together with Sobolev embedding and the analog result for the linear wave equation in $\mathbb R^3$).
\end{proof}

We will also need the following Lemma that was already proved here in some cases (Lemma~\ref{lm:projectdilat}).

\begin{lem}\label{lm:projectgeneral}
Up to a subsequence, we always have $\Omega_n \to  X_{\mathcal O}$ with one limit described in Section~\ref{subsec:profnot}. Then, for any $\vec \varphi\in \mathcal{H}(\mathbb R^3)$, we have 
\begin{align*}
& \Big\Vert P_{\Omega_n}\vec \varphi - P_{X_{\mathcal O}}\vec \varphi \Big\Vert_{\mathcal{H}(\mathbb R^3)} 
\underset{n\to +\infty}{\longrightarrow}  0,\\
&T_{\mathcal{O}^{-1},n} P_{\Omega}T_{\mathcal{O},n}\vec \varphi= P_{X_{\mathcal O}}\vec \varphi+o_{\mathcal{H}(\mathbb R^3)} (1),\\
&\vec \varphi_{\Omega, \mathcal{\mathcal{O}},n}=(P_{X_{\mathcal O}}\vec \varphi)_{\Omega, \mathcal{\mathcal{O}},n}+o_{\mathcal{H}(\R^3)}(1).
\end{align*}
\end{lem}
\begin{proof}
The second statement is a direct consequence of \eqref{eq:conj_proj} and the first statement. For the third statement, we use formula~\eqref{EProf} and \eqref{eq:conj_proj} to write $
\vec \varphi _{\Omega, \mathcal{\mathcal{O}},n}:=\SO(-t_n)P_{\Omega} T_{\mathcal{O},n}\vec \varphi= \SO(-t_n)T_{\mathcal{O},n}P_{\Omega_n} \vec \varphi$. So the third statement comes again as a consequence of the first one. 

We now prove the first statement. The cases  $\lambda_n \to +\infty$ and $|x_n|\to +\infty$ are treated in Lemma~\ref{lm:projectdilat} where $P_{X_{\mathcal O}}=\operatorname{Id}$. In the case $\lambda_n\to 0$ and $x_n$ convergent, it was proved in \cite[Lemma 2.1.2]{GG} using that \cite[Proposition 2.1.1]{GG} implies $\Omega_n \to X_{\OO}$ with the definition of appropriate definition of $X_{\OO}$ in Section~\ref{subsec:profnot}. The only remaining case is $\lambda_n\to \lambda_0>0$ and $x_n\to x_{0}$. We prove that $\Omega_n \to X_{\OO}=\frac{\Omega-x_{0}}{\lambda_0}$ in the sense given in Section~\ref{subsec:profnot}. This will concludes thanks to \cite[Lemma 2.1.2]{GG}, noticing that the proof only involves the convergence of domain. We now come to the proof of $\Omega_n \to X_{\OO}$. Let $K$ be a compact subset of $\frac{\Omega-x_{0}}{\lambda_0}=\R^3 \setminus \Theta_0$ where $\Theta_0=\frac{\Theta-x_{0}}{\lambda_0}$. In particular, $\operatorname{dist}(K,\Theta_0)>0$. We easily check that $\operatorname{dist}(\Theta_n,\Theta_0)\to 0$, so we can take $n$ large enough so that $\operatorname{dist}(K,\Theta_n)\geq \operatorname{dist}(K,\Theta_0) - \operatorname{dist}(\Theta_n,\Theta_0) \geq \frac 12 \operatorname{dist}(K,\Theta_0) > 0$, thus $K \subset \mathbb R^3 \backslash \Theta_n = \Omega_n$. On the other hand, if $K$ is a compact subset included in $\overline{\frac{\Omega-x_{0}}{\lambda_0}}^c = \operatorname{Int}\Theta_0$, $K \subset \cup_{1=1}^N B(x_j, \epsilon)$, where $x_j \in K$, and for $\epsilon > 0$ small enough, $B(x_j, \epsilon) \subset \Theta_0$ for all $j$. The fact that $\lambda_n \to \lambda_0$ and $x_n \to x_0$ shows that $B(x_j, \frac\epsilon 2)\subset \Theta_n$ for $n$ large enough, and thus $K \subset \operatorname{Int}\Theta_n = \overline{\frac{\Omega-x_{n}}{\lambda_n}}^c$. This ends the proof.
 \end{proof}

With the above at hand, Theorem \ref{th:lindec} will be a consequence of the three following Lemma.

\begin{lem} \label{lem:ellpd}
Let $(f_n)_{n\in \N}$ be a bounded sequence in $\dot H_0^1(\Omega)$ such that for all scale core $\mathcal O$ (still denoting by $f_n$ its extension by zero to $\dot H^1(\mathbb R^3)$),
$$
T_{\mathcal O^{-1}, n} f_n \rightharpoonup 0 \text{ in } \dot H^1 (\mathbb R^3).
$$
Then, up to a subsequence,
$$
\Vert f_n \Vert_{L^6(\Omega)} \rightarrow 0.
$$
\end{lem}

\begin{proof}
We apply the elliptic profile decomposition in $\mathbb R^3$ of \cite{Gerard98} to $f_n$: there exists orthogonal scale cores $(\mathcal O^{(j)})_{j \geq 1}$ so that, up to a subsequence, for any $J \geq 1$
$$
f_n = \sum_{j=1}^J \frac {1}{\lambda_{j,n}^{\frac 12}} \psi^{(j)}\left( \frac{\cdot - x_{n, j}}{\lambda_{j,n}} \right) + w_{n, J}, \hspace{0.3cm} \lim_{J\rightarrow + \infty} \limsup_{n \rightarrow +\infty} \Vert w_{n, J} \Vert_{L^6} = 0,
$$
where, by construction, for any $j$
$$
T_{\mathcal (O^{(j)})^{-1}, n} f_n \rightharpoonup \psi^{(j)} \text{ in }\dot H^{1}(\mathbb R^3),
$$
hence $\psi^{(j)} = 0$ for all $j$ and the result follows.
\end{proof}

\begin{lem} 
\label{EquivFrames} 
If $\mathcal{O}^{1}$ and $\mathcal{O}^{2}$ are equivalent frames, then there exists an isometry $Q:\mathcal{H}(X_{\OO^{2}})\to\mathcal{H}(X_{\OO^{1}})$ such that for any profile $\vec \psi_{\Omega,\mathcal{O}^{2},n}$, with $\vec \psi \in \mathcal{H}(X_{\OO^{2}})$\footnote{If $\vec \psi \in \mathcal{H}(\R^3)$, Lemma~\ref{lm:projectgeneral} shows that the same result holds with $Q$ replaced by $Q\circ P_{X_{\OO^{2}}}$}, up to a subsequence there holds that
\begin{equation}\label{eq:fr_equiv}
\limsup_{n\to+\infty}
\nor{(\vec{Q\psi})_{\Omega,\mathcal{O}^{1},n}-\vec\psi_{\Omega,\mathcal{O}^{2},n}}{\mathcal{H}(\Omega)}=0.
\end{equation}
If moreover, $\lambda_n^1=\lambda_n^2$, $\lambda_n^1=\lambda_n^2$, $\lambda_n^i\to 0$, $x_n^i\to x_{\infty}$ with $t_n^1=t_n^2+\tau \lambda_n^1$, then $Q=S_{X_{\OO^{2}}}(\tau)$.
\end{lem}

\begin{proof}
$\;$
\subsection*{First case: for $j=1,2$,
$\lambda_{j,n} \rightarrow \infty$ or  $|x_{j,n}|\rightarrow \infty$}
Then, thanks to Lemma \ref{lem:assfree}, it suffices to prove the same result for solutions of the equation without obstacle, for which the result is obtained with $Q$ a composition of translation, dilation and the application of flow of the wave equation in $\R^3$ in finite time. 

\subsection*{Second case: for $j=1,2$, $C\leq \lambda_{j, n}\leq C^{-1}$ , $| x_{j, n}| \leq C$} We select a subsequence so that $\lambda_{j, n}$ and $x_{j, n}$ converge. This provides an isometry $T_0$ from $\mathcal{H}(X_{\OO^{1}})$ to $\mathcal{H}(X_{\OO^{2}})$ as a composition  of translation, dilation (with similar scaling as $T_{\mathcal{O},k}$), so that 
\begin{equation} \label{eq:iso1}
T_{(\mathcal{O}^2)^{-1},n}\circ T_{\mathcal{O}^1,n}\psi =T_0 \psi+o_{\mathcal{H}(\R^3)}(1).
\end{equation}
Up to a subsequence, we assume that $t_n^2-t_n^1\to t_0$. We choose 
$$
Q := T_0^{-1}\circ S_{X_{\OO^{2}}}(-t_0).$$  We compute, using that $\SO(t)$ is unitary on $\mathcal{H}(\Omega)$ and $T_{(\mathcal{O}^2)^{-1},n}$ from $\mathcal{H}(\Omega)$ to $\mathcal{H}(\Omega_n^2)$
\begin{align*}
 &   \nor{(\vec{Q\psi})_{\Omega,\mathcal{O}^{1},n}-\vec\psi_{\Omega,\mathcal{O}^{2},n}}{\mathcal{H}(\Omega)} \\
 & \quad = \nor{\SO (-t_n^1)P_{\Omega}T_{\mathcal{O}^1,n} T_0^{-1} S_{X_{\OO^{2}}}(-t_0)\vec  \psi-\SO (-t_n^2)P_{\Omega}T_{\mathcal{O}^2,n}\vec  \psi}{\mathcal{H}(\Omega)}\\
    & \quad = \nor{P_{\Omega}T_{\mathcal{O}^1,n} T_0^{-1}\circ S_{X_{\OO^{2}}}(-t_0)\vec  \psi-\SO (t_n^1-t_n^2)P_{\Omega}T_{\mathcal{O}^2,n}\vec  \psi}{\mathcal{H}(\Omega)}\\
    & \quad = \nor{P_{\Omega^2_n}T_{(\mathcal{O}^2)^{-1},n}T_{\mathcal{O}^1,n} T_0^{-1}S_{X_{\OO^{2}}}(-t_0)\vec  \psi-S_{\Omega_n^{2}} (t_n^1-t_n^2)P_{\Omega^2_n}\vec  \psi}{\mathcal{H}(\Omega_n^2)}.
\end{align*}
Combining  \cite[Proposition 2.1.3.]{GG}, which also applies here, and Lemma~\ref{lm:projectgeneral} gives 
\begin{align*}
S_{\Omega_n^{2}} (t_n^1-t_n^2)P_{\Omega^2_n}\vec  \psi&=S_{X_{\OO^{2}}}(-t_0)P_{X_{\OO^{2}}}\vec \psi+o_{\mathcal{H}(\R^3)}(1) \\ & =S_{X_{\OO^{2}}}(-t_0)\vec \psi+o_{\mathcal{H}(\R^3)}(1) \\
& = P_{\Omega^2_n} S_{X_{\OO^{2}}}(-t_0)\vec \psi+o_{\mathcal{H}(\R^3)}(1),
\end{align*}
where in the last line the projection comes for free as the left hand side is in $\mathcal H(\Omega_n^2)$.
In addition, by (\ref{eq:iso1}), 
\begin{align*}
T_{(\mathcal{O}^2)^{-1},n}T_{\mathcal{O}^1,n}T_0^{-1} S_{X_{\OO^{2}}}(-t_0)\vec  \psi & =   T_{0} T_0^{-1} S_{X_{\OO^{2}}}(-t_0)\vec  \psi+o_{\mathcal{H}(\R^3)}(1) \\
& =S_{X_{\OO^{2}}}(-t_0)\vec  \psi+o_{\mathcal{H}(\R^3)}(1).
\end{align*}
So, combining the three identities above, we finally obtain
 \begin{align*}
 &   \nor{(\vec{Q\psi})_{\Omega,\mathcal{O}^{1},n}-\vec\psi_{\Omega,\mathcal{O}^{2},n}}{\mathcal{H}(\Omega)}= \nor{P_{\Omega^2_n}S_{X_{\OO^{2}}}(-t_0)\vec  \psi-P_{\Omega^2_n}S_{X_{\OO^{(2)}}}(-t_0)\vec \psi}{\mathcal{H}(\Omega_n^2)} +o(1)=o(1).
\end{align*}

\subsection*{Third case: for $j=1,2$, $\lambda_{j,n} \rightarrow 0$ and $|x_{j,n}|\leq C$}

By assumption, up to a subsequence,
$$
t_{1,n} = t_{2,n} + \lambda_{1,n} \tau_n, \hspace{0.3cm} x_{1,n} =  x_{2,n} + \lambda_{1,n} \xi_n, \hspace{0.3cm} \lambda_{1,n} = \lambda_{2,n} \mu_n,
$$
with $\tau_n \rightarrow \tau \in \mathbb R$, $\xi_n \rightarrow \xi \in \mathbb R^3$, $\mu_n \rightarrow \mu >0$. Now, let
$$
\vec u_n(t) : = T_{(\mathcal O^{1})^{-1}, n} S_{\Omega}(\lambda_{1,n}t) P_\Omega T_{\mathcal O^{2}, n} \vec \psi,
$$
so that $\vec u_n(\tau_n)=T_{(\mathcal O^{1})^{-1}, n} S_{\Omega}(t_{1,n})\vec\psi_{\Omega,\mathcal{O}^{2},n}$ and \eqref{eq:fr_equiv} is equivalent to 
$$
\nor{P_{\Omega_n^1}\vec{Q\psi}-u_n(\tau_n)}{\mathcal{H}(\Omega_n^1)}\to 0.$$
If additionally, $\vec{Q\psi}\in \mathcal{H}(X_{\OO^{1}})$ as expected, this is equivalent to proving that $u_n(\tau_n)\to \vec{Q\psi}$ thanks to Lemma~\ref{lm:projectgeneral}. Observe that 
$$
u_n(t,x) = \lambda_{1,n}^{\frac 12} \tilde u_n(\lambda_{1,n} t,\lambda_{1,n}x+x_{1,n} ),$$
where $\tilde u_n(t,x) :=  S_{\Omega}(t) P_\Omega T_{\mathcal O^2, n} \vec \psi$, hence $u_n$ is the solution to
$$
\begin{cases}
\partial^2_t u_n - \Delta u_n = 0 &\text{ in }\Omega_n^1, \\
u_n =0  &\text{ on }\partial\Omega_n^1, \\
u_n(t=0) = T_{{(\mathcal O^{1})}^{-1}, n} P_\Omega T_{\mathcal O^{2}, n} \vec \psi. 
\end{cases}
$$
Let now $v_n(t,x) := (\frac{\lambda_{2,n}}{\lambda_{1,n}})^{\frac 12} u_n(\frac{\lambda_{2,n}}{\lambda_{1,n}} t,\frac{\lambda_{2,n}}{\lambda_{1,n}} x +\frac{x_{2, n}-x_{1,n}}{\lambda_{1,n}} )$. It is solution to
$$
\begin{cases}
\partial^2_t v_n - \Delta v_n = 0 &\text{ in }\Omega_n^2, \\
v_n =0  &\text{ on }\partial\Omega_n^2, \\
v_n(t=0) = T_{{(\mathcal O^{2})}^{-1}, n} P_\Omega T_{\mathcal O^{2}, n} \vec \psi. 
\end{cases}
$$
Observe that thanks to Lemma~\ref{lm:projectgeneral} 
$$
T_{{(\mathcal O^2)}^{-1}, n} P_\Omega T_{\mathcal O^2, n} \vec \psi = P_{\Omega_n^2} \vec \psi
\rightarrow P_{X_{\mathcal O ^{(2)}}} \vec \psi,
$$
in $\mathcal H(\mathbb R^3)$ and hence, by \cite[Proposition 2.1.3]{GG}, $\vec v_n$ converges in $L^\infty_{\rm loc} \mathcal H(\mathbb R^3)$ to $S_{X_{\OO^{2}}}(t) P_{X_{\OO^{2}}} \vec \psi$. As $\tau_n$ is bounded, it follows taking $t=\tau_n$ that
$$
\big((\frac{\lambda_{2,n}}{\lambda_{1,n}})^{\frac 12},(\frac{\lambda_{2,n}}{\lambda_{1,n}})^{\frac 32} \big) \vec u_n(\frac{\lambda_{2,n}}{\lambda_{1,n}} \tau_n,\frac{\lambda_{2,n}}{\lambda_{1,n}} x +\frac{x_{2, n}-x_{1,n}}{\lambda_{1,n}} ) \rightarrow S_{X_{\OO^{2}}}(\tau) P_{X_{\OO^{2}}} \vec \psi
$$
in $\mathcal H(\mathbb R^3)$, and hence
$$
\vec u_n(\tau_n) \rightarrow D_{\mu,\xi} S_{X_{\OO^{2}}}(\tau) P_{X_{\OO^{2}}} \vec \psi, \hspace{0.3cm} D_{\mu,\xi}\vec\varphi := (\mu^{-\frac 1 2}, \mu^{-\frac 3 2}) \varphi(\cdot \mu^{-1} - \xi).
$$
The result follows with $Q :=D_{\mu,\xi}\circ S_{X_{\OO^{2}}}(\tau)$. This gives also the last statement.
\end{proof}

\begin{lem}\label{lm:orth_weak}
Assume that some cores $\mathcal O^{1}$ and $\mathcal O^{2}$ are orthogonal and $\vec \psi \in \mathcal{H}(\R^3)$, then
\begin{equation*}
   T_{ {\mathcal (O^{2})}^{-1}, n}  \SO(t^{2}_n )   \vec \psi_{\Omega, \mathcal{O}^{1},n}\rightharpoonup 0.
\end{equation*}
\end{lem} 
\begin{proof}We denote $\vec V_n:= T_{ {\mathcal (O^{2})}^{-1}, n}  \SO(t^{2}_n )   \vec \psi_{\Omega, \mathcal{O}^{1},n}$ the function we consider. We will denote $\vec V_n = (V_n, V'_n)$ its two components.

\textbf{First case: $\lambda_n^1\to +\infty$ or $|x_n^{1}|\to +\infty$.}
In this case, Lemma~\ref{lem:assfree} and the unitarity of $T_{ {\mathcal (O^{2})}^{-1}, n}$ gives 
\begin{align}
\label{e:dilatweakinfty}
  \nonumber   \vec V_n= &  T_{ {(O^{2})}^{-1}, n}  \SF(t^{2}_n -t^{1}_n )T_{ {\OO^{1}}, n}\vec \psi+o_{\mathcal{H}(\R^3)}(1)\\
     =&T_{ {(\OO^{2})}^{-1}, n}T_{ {\OO^{1}}, n}  \SF\left(\frac{t^{2}_n -t^{1}_n }{\lambda_n^1}\right)\vec \psi+o_{\mathcal{H}(\R^3)}(1).
\end{align}
We are left to the case on $\R^3$, which is known.

\textbf{Second case: up to a subsequence, $\lambda_n^1\to 0$, $x_n^{1}\to x_{\infty}$, and $\left|\frac{t^{2}_n-t^{1}_n}{\lambda_n^1}\right| \to +\infty$.}
This is inspired by \cite[Lemma 3.7.]{GG}, see also \cite[Lemma 5.2.10]{L:11}. By unitarity of $T_{\OO, n}$ on $L^6(\R^3)$, we have 
$$
\nor{V_n}{L^6(\R^3)}=\nor{\SO \left(t^{2}_n-t^{1}_n \right)P_{\Omega}T_{ {\OO^{1}}, n}\vec \psi }{L^6(\R^3)} \to 0
$$
by Proposition~\ref{propnoncon_loc}. So, in particular, the first component of $T_{ { (\OO^{2})}^{-1}, n}  \SO(t^{2}_n )   \vec \psi_{\Omega, \OO^{1},n}$ converges weakly to zero, that is, there exists $f\in L^2(\R^3)$ so that for every $\vec \varphi\in \mathcal{H}(\R^3)$
\begin{equation}
\label{e:weakconv0f}
    \left<T_{ { (\OO^{2})}^{-1}, n}  \SO(t^{2}_n )   \vec \psi_{\Omega, \OO^{1},n},\vec \varphi\right>_{\mathcal{H}(\R^3)}\underset{n\to +\infty}{\longrightarrow}    \left< (0,f),\vec\varphi\right>_{\mathcal{H}(\R^3)}.
\end{equation}

To conclude, we will show that $f=0$. To this end, let $s\in \mathbb R$ a parameter and consider instead $\widetilde{t}_n^{2} := t^{2}_n+s \lambda_n^1$. Defining the associated cores $\widetilde{\OO}^{2}$, which satisfies the same assumptions as $\OO^2$, we obtain similarly, for another $\widetilde{f}_s\in L^2(\R^3)$,
\begin{equation}\label{e:weakconv0ftilde}
    \left<T_{ { (\OO^{2})}^{-1}, n}  \SO(\widetilde{t}_n^{2})   \vec \psi_{\Omega, \OO^{1},n},\vec \varphi\right>_{\mathcal{H}(\R^3)}\underset{n\to +\infty}{\longrightarrow}      \left<(0,\widetilde{f}_s), \vec \varphi\right>_{\mathcal{H}(\R^3)}.
\end{equation}
But, we have, using successively Lemma~\ref{lm:projectgeneral} and~\ref{EquivFrames} where $Q_s=S_{X_{\OO^2}}(s)$,
\begin{align*}
  &  \left<T_{ { (\OO^{2})}^{-1}, n}  \SO(\widetilde{t}_n^{2})   \vec \psi_{\Omega, \OO^{1},n},\vec \varphi\right>_{\mathcal{H}(\R^3)}=  \left< \SO(\widetilde{t}_n^{2})   \vec \psi_{\Omega, \OO^{1},n},    T_{ { \OO^{2}}, n} \vec \varphi\right>_{\mathcal{H}(\R^3)}\\
& =  \left< \SO(\widetilde{t}_n^{2})   \vec \psi_{\Omega, \OO^{1},n},  P_{\Omega}  T_{ { \OO^{2}}, n} \vec \varphi\right>_{\mathcal{H}(\Omega)}= \left<   \vec \psi_{\Omega, \OO^{1},n},  \SO(-\widetilde{t}_n^{2}) P_{\Omega}  T_{ { \OO^{2}}, n} \vec \varphi\right>_{\mathcal{H}(\Omega)} \\ 
& =\left<   \vec \psi_{\Omega, \OO^{1},n},   \vec \varphi_{\Omega,\widetilde{\OO}^{2},n}\right>_{\mathcal{H}(\Omega)}
 =\left<   \vec \psi_{\Omega, \OO^{1},n},   (P_{X_{\OO^2}}\vec \varphi)_{\Omega,\widetilde{\OO}^{2},n}\right>_{\mathcal{H}(\Omega)}+o(1)  \\
 &=\left<   \vec \psi_{\Omega, \OO^{1},n},   (Q_sP_{X_{\OO^2}}\vec \varphi)_{\Omega,\OO^{2},n}\right>_{\mathcal{H}(\Omega)}+o(1)
\\
& = \left<T_{ { (\OO^{2})}^{-1}, n}  \SO(t_n^{2})   \vec \psi_{\Omega, \OO^{1},n},Q_sP_{X_{\OO^2}}\vec \varphi\right>_{\mathcal{H}(\R^3)}+o(1) \\
&=  \left< (0,f),Q_s P_{X_{\OO^2}}\vec \varphi\right>_{\mathcal{H}(\R^3)}+o(1).
\end{align*}
where we have used \eqref{e:weakconv0f} with $Q_sP_{X_{\OO^2}}\vec \varphi$ as test function. Combining this with \eqref{e:weakconv0ftilde} gives 
$$
\forall s\in \mathbb R, \quad \forall \vec \varphi \in \mathcal{H}(X_{\OO^2}), \quad \left< (0,f),Q_s P_{X_{\OO^2}}\vec \varphi\right>_{\mathcal{H}(\R^3)}=\left<(0,\widetilde{f}_s), \vec \varphi\right>_{\mathcal{H}(\R^3)}.
$$ 
This implies $(0,f)\in \mathcal{H}(X_{\OO^2})$ and for any $s\in \R$,  $\vec \varphi \in \mathcal{H}(X_{\OO^2})$, 
\begin{equation*}
    \left<(0,\widetilde{f}_s), \vec \varphi\right>_{\mathcal{H}(X_{\OO^2})}=\left<  (0,f),Q_s P_{X_{\OO^2}}\vec \varphi\right>_{\mathcal{H}(X_{\OO^2})}=\left<  S_{X_{\OO^2}}(s)(0,f), \vec \varphi\right>_{\mathcal{H}(X_{\OO^2})},
\end{equation*}
from which we obtain 
$$
(0,\widetilde{f}_s) = S_{X_{\OO^2}}(s)(0,f). 
$$ 
This implies that $f=0$ since the only solution of the wave equation with the first component equal to $0$ is the zero solution. This ends the proof in this case. 

\textbf{Third case: up to a subsequence, $\lambda_n\to 0$, $x_n^1\to x_{\infty}$ and $\frac{t^{2}_n-t^{1}_n}{\lambda_n^1} \to s_{\infty}$}.
In this case, Lemma~\ref{lmapproxconc} gives  
$$
\SO \left(t^{2}_n-t^{1}_n \right)P_{\Omega}T_{ {\OO^{1}}, n}\vec \psi=T_{ {\OO^{1}}, n} S_{X_{\mathcal{O}}}(s_{\infty}) P_{X_{\mathcal{O}}} \psi+o_{\mathcal{H}(\R^3)}(1).
$$
In particular, 
$$
\vec V_n= T_{ {(\OO^{2})}^{-1}, n}  T_{ {\OO^{1}}, n} S_{X_{\mathcal{O}}}(s_{\infty}) P_{X_{\OO}} \psi+o_{\mathcal{H}(\R^3)}(1),
$$
which converges weakly to zero thanks to the case on $X_{\OO}$ which follows from the case of $\R^3$ since we have $\ln \left(\frac{\lambda_k^1}{\lambda^2_k} + \frac{|x_k^1-x^2_k|}{\lambda_k^1} \right)\to +\infty$.

\textbf{Fourth case: up to a subsequence,
$\lambda_n^1\to \lambda_0>0$ and $x_n^{1}\to x_{\infty}$.} We first notice that $ P_{\Omega}T_{ {\OO^{1}}, n}\vec \psi$ converges strongly to one fixed function $\vec\Psi\in \mathcal{H}(\R^3)$.

In the sub-case where $t^{2}_n-t^{1}_n\to s_{\infty}$, 
$$
\vec V_n= T_{ {(\OO^{2})}^{-1}, n}\SO \left(s_{\infty}\right) \vec \Psi+o_{\mathcal{H}(\R^3)}(1),
$$
which converges weakly to zero since by orthogonality, we have either $\lambda_n^2\to +\infty$ or $\lambda_n^2\to 0$ or $|x_n^{2}|\to +\infty$. 

In the last remaining sub-case where up to a subsequence, $t^{2}_n-t^{1}_n\to \pm\infty$, by Lemma~\ref{lem:lin_scat} there exists $\vec \Psi^{\pm} \in \mathcal H(\mathbb R^2)$ so that $\SO(t)\vec \Psi - \SF(t) \vec \Psi^{\pm} \to 0$ as $t \to \pm \infty$ in $\mathcal H(\mathbb R^3)$, and therefore
$$
\vec V_n 
= T_{ {(\OO^{2})}^{-1}, n}\SF \left(t^{2}_n-t^{1}_n \right) \Psi^{\pm}+o_{\mathcal{H}(\R^3)}(1).
$$
The result in $\R^3$ holds and gives the result.
\end{proof} 

\begin{remark}
Note that the conclusion of the previous Lemma is false if the non focusing Assumption~\ref{ass:nonreco} is not satisfied as it happens, for instance, on the sphere, see \cite{L:11}.    
\end{remark}

We can now prove the linear profile decomposition.

\begin{proof}[Proof of Theorem \ref{th:lindec}]
$\;$
\subsection*{Setup of the induction procedure}

For $\underline{\vec \varphi} = (\vec \varphi_n)_{n\geq 1}$ a bounded sequence in $\mathcal H$,
we denote $\Lambda (\underline{\vec \varphi})$ the set of all couples $(\vec \psi, \mathcal O)$ where $ \vec \psi \in \dot H^1 \times L^2 (\mathbb R^3)$ and $\mathcal O$ is a scale core so that, up to a subsequence
$$
T_{\mathcal O ^{-1}, n} \SO(t_n)\vec \varphi_n \rightharpoonup \vec \psi \hspace{0.3cm} \text{ in }\dot H^1 \times L^2 (\mathbb R^3),
$$
and we let
$$
\eta(\underline{\vec \varphi}) := \sup_{(\vec \psi, \mathcal O) \in \Lambda (\underline{\vec \varphi})} \Vert P_{X_{\mathcal O}} \vec \psi \Vert_{\mathcal H(X_{\mathcal O})}.
$$
Arguing inductively on $J \geq 0$, we will construct an extraction of $(\vec \varphi_n)_{n\geq 1}$, orthogonal cores
$\mathcal O^{(j)}$, profiles $\vec \psi^{(j)}$, and remainder $\vec w_{n}^{(j)}$ 
so that (\ref{eq:lpd}), (\ref{eq:lpd_pyt}), (\ref{eq:lpd_L6}) hold, and in addition 
\begin{equation} \label{eq:lpd:decstep}
\forall 1\leq j \leq J, \hspace{0.5cm} \eta( (\vec w_n^{(j-1)})_{n\geq 1} ) \leq 2 \Vert P_{X_{\mathcal O^{(j)}}} \vec \psi^{(j)} \Vert_{\mathcal H(X_{\mathcal O^{(j)}})},
\end{equation}
and
\begin{equation} \label{eq:lpd:decweak}
\forall 1\leq j \leq J, \hspace{0.5cm} \vec r_n^{(j)} := T_{ {\mathcal (O^{(j)})}^{-1}, n} \SO(t^{(j)}_n) \vec w_n^{(j)} \rightharpoonup \vec 0 \hspace{0.3cm} \text{ in }\mathcal{H}(\mathbb R^3).
\end{equation}
Observe that, taking
$$
\vec w_n^{(0)} := \vec \varphi_n,
$$
the decomposition (\ref{eq:lpd}), (\ref{eq:lpd_pyt}), (\ref{eq:lpd_L6}), with (\ref{eq:lpd:decstep}) and (\ref{eq:lpd:decweak}) holds at rank $J=0$. Therefore, let $J\geq 1$, assume that  we have a decomposition (\ref{eq:lpd}), (\ref{eq:lpd_pyt}), (\ref{eq:lpd_L6}), with (\ref{eq:lpd:decstep}) and (\ref{eq:lpd:decweak}) at rank $J-1$, and let us construct it at rank $J$.

 If $\eta((\vec w_n^{(J-1)})_{n\geq 1}) = 0$, then we take $\vec \psi^{(J)} := \vec 0$ and we are done. Otherwise, there exist $\vec \psi^{(J)} \in \mathcal{H} (\mathbb R^3)$ and a scale core $\mathcal O^{(J)} = \{ (t^{(J)}_n)_{n\geq 1}, (\lambda^{(J)}_n)_{n\geq 1}, (x^{(J)}_n)_{n\geq 1} \}$ such that, up to a subsequence,
\begin{equation} \label{eq:lpd:wl}
T_{ {\mathcal (\OO^{(J)})}^{-1}, n} \SO(t^{(J)}_n) \vec w_n^{(J-1)} \rightharpoonup \vec \psi^{(J)} \hspace{0.3cm} \text{ in }\mathcal{H}(\mathbb R^3), 
\end{equation}
and so that (\ref{eq:lpd:decstep}) holds.
We take
\begin{equation} \label{eq:lpd:fr}
w_n^{(J)} := \vec w_n^{(J-1)} - \vec \psi^{(J)} _{\Omega, \mathcal{O}^{(J)},n},
\end{equation}
so that (\ref{eq:lpd}) and (\ref{eq:lpd:decweak}) hold at step $J$. Let us detail why  (\ref{eq:lpd:decweak}) holds. We have to show that 
\begin{align}
    \label{verifweak}
    T_{ { (\OO^{(J)})}^{-1}, n} \SO(t^{(J)}_n) \vec \psi^{(J)} _{\Omega, \OO^{(J)},n}\rightharpoonup \psi^{(J)}  \hspace{0.3cm} \text{ in }\mathcal{H} (\mathbb R^3).
\end{align}
Observe that, using Lemma~\ref{lm:projectgeneral} to obtain the convergence in the last line,
\begin{align*}
T_{ {\mathcal (\OO^{(J)})}^{-1}, n} \SO(t^{(J)}_n) \vec \psi^{(J)} _{\Omega, \mathcal{O}^{(J)},n} &=T_{ { (\OO^{(J)})}^{-1}, n} \SO(t^{(J)}_n) \SO(-t^{(J)}_n) P_{\Omega}T_{\OO^{(J)},n}\vec \psi^{(J)} \\
& =T_{ {(\OO^{(J)})}^{-1}, n}P_{\Omega}T_{\OO^{(J)},n}\vec \psi^{(J)}=P_{\Omega^{(J)}_n}\vec \psi^{(J)} \\
& \to P_{X^{(J)}_{\OO}}\vec \psi^{(J)},
\end{align*}
Therefore, it only remains to show that $P_{X^{(J)}_{\OO}}\vec \psi^{(J)}=\vec \psi^{(J)}$. But since $\SO(t^{(J)}_n)\vec w_n^{(J-1)}\in \mathcal H(\Omega)$,  we have $T_{ {(\OO^{(J)})}^{-1}, n} \SO(t^{(J)}_n) \vec w_n^{(J-1)}\in \mathcal H(\Omega_n^{(J)})$.  In particular,  \eqref{eq:lpd:wl} implies that $\vec \psi^{(J)}$ is a weak limit of functions in $\mathcal H(\Omega_n^{(J)})$.  It implies that $\operatorname{Supp} (\vec \psi^{(J)})\subset X_{\OO^{(J)}}$ and therefore $P_{X^{(J)}_{\OO}}\vec \psi^{(J)}=\vec \psi^{(J)}$, see again Lemma~\ref{lm:projectgeneral}.

We now separate the cases to check (\ref{eq:lpd_pyt}) and (\ref{eq:lpd_L6}). 
In the sequel, we will denote
$$
\mathcal O = \{ (t_n)_{n\geq 1}, (\lambda_n)_{n\geq 1}, (x_n)_{n\geq 1} \} := \mathcal O^{(J)}
$$
for conciseness.

\subsection*{The Pythagorean expansion (\ref{eq:lpd_pyt})} Observe that, by the induction assumption, using (\ref{eq:lpd_pyt}) at rank $J-1$, it suffices to show that
\begin{equation} \label{eq:lpd_pytJ}
\Vert \vec w_n^{(J-1)} \Vert^2_{\mathcal H} = \Vert \vec w_n^{(J)} \Vert^2_{\mathcal H} + \Vert \vec \psi^{(J)} _{\Omega, \mathcal{O},n} \Vert^2_{\mathcal H} + o(1)
\end{equation}
in order to show (\ref{eq:lpd_pyt}) at rank $J$. We separate cases.

\subsubsection*{First case: $\lambda_n \rightarrow \bar \lambda \in (0, +\infty)$ and $x_n \rightarrow \bar x \in \mathbb R^3$ (up to a subsequence).}
Let $T$ be the ($n$-independent) rescaling associated with $\bar \lambda$ and $\bar x$, that is, with the notations of \S \ref{subsec:profnot},
$$
T  := T_{\{(\bar \lambda)_{n\geq 1}, (\bar x)_{n\geq 1} \}, n}.
$$
Observe that, by (\ref{eq:lpd:wl})  and using that $T_{ {\OO^{(J)}}, n} \vec \psi^{(J)} $ converges strongly to $T \vec \psi^{(J)}$, we get 
$$
\SO(t_n)\vec w_n^{(J-1)} \rightharpoonup  T \vec \psi^{(J)} \hspace{0.3cm} \text{ in }\mathcal{H}(\R^3),
$$
and hence
\begin{equation} \label{eq:lpd:wltw}
\SO(t_n)\vec w_n^{(J-1)} \rightharpoonup  P_{\Omega}T \vec \psi^{(J)} \hspace{0.3cm} \text{ in }\mathcal{H} (\Omega).
\end{equation} 
On the other hand, \eqref{eq:lpd:fr} and $T_{ {\OO}, n} \vec \psi^{(J)}=T \vec \psi^{(J)}+o_{\mathcal{H}(\Omega)}(1) $ give 
\begin{equation} \label{eq:lpd:eqT}
\langle \vec \psi^{(J)} _{\Omega, \mathcal{O},n}, \vec w_n^{(J)}  \rangle_{\mathcal{H} (\Omega)} 
= \langle \SO(-t_n) P_\Omega T  \vec \psi^{(J)},  \vec w_n^{(J-1)} - \SO(-t_n) P_\Omega T \vec \psi^{(J)} \rangle_{\dot H ^1 \times L^2 (\Omega)} + o(1).
\end{equation}
But, using (\ref{eq:lpd:wltw})
\begin{equation*}
\langle \SO(-t_n) P_\Omega T  \vec \psi^{(J)},  \vec w_n^{(J-1)} - \SO(-t_n) P_\Omega T \vec \psi^{(J)} \rangle_{\mathcal{H} (\Omega)}  
= \langle P_\Omega T  \vec \psi^{(J)},  \SO(t_n) \vec w_n^{(J-1)} -  P_\Omega T \vec \psi^{(J)} \rangle_{\mathcal{H} (\Omega)} \rightarrow 0.
\end{equation*}
The above combined with (\ref{eq:lpd:eqT}) gives (\ref{eq:lpd_pytJ}).

\subsubsection*{Second case: $\lambda_n \rightarrow \infty$ or $x_n \rightarrow \infty$ (up to a subsequence)}
Observe that
\begin{align*} 
&\langle \vec \psi^{(J)} _{\Omega, \mathcal{O},n}, \vec w_n^{(J)}  \rangle_{\mathcal{H}(\Omega)} \\
&\hspace{1cm}= \langle \SO(-t_n) P_\Omega T_{\mathcal O, n}  \vec \psi^{(J)},  \vec w_n^{(J-1)} - \SO(-t_n) P_\Omega T_{\mathcal O, n} \vec \psi^{(J)} \rangle_{\mathcal{H} (\Omega)} 
\\
&\hspace{1cm}= \langle  P_\Omega T_{\mathcal O, n}  \vec \psi^{(J)},  \SO(t_n) \vec w_n^{(J-1)} -   P_\Omega T_{\mathcal O, n} \vec \psi^{(J)} \rangle_{\mathcal{H} (\Omega)} \\
&\hspace{1cm}= \langle  T_{\mathcal O, n}  \vec \psi^{(J)},  \SO(t_n) \vec w_n^{(J-1)} -   T_{\mathcal O, n} \vec \psi^{(J)} \rangle_{\mathcal{H} (\R^3)} + o(1),
\end{align*}
where we used Lemma~\ref{lm:projectgeneral} to obtain the last line and get $P_\Omega T_{\mathcal O, n} \vec \psi^{(J)}= T_{\mathcal O, n} P_{\Omega_n}\vec \psi^{(J)}=T_{\mathcal O, n} \vec \psi^{(J)}+o_{\mathcal{H}}(1)$. Continuing this chain of equalities, we get 
\begin{align*} 
&\langle \vec \psi^{(J)} _{\Omega, \mathcal{O},n}, \vec w_n^{(J)}  \rangle_{\mathcal{H} (\Omega)}
= \langle   \vec \psi^{(J)},   T_{\mathcal O^{-1}, n} \SO(t_n) \vec w_n^{(J-1)} -   \vec \psi^{(J)} \rangle_{\mathcal{H} (\mathbb R^3)} + o(1) \rightarrow 0,
\end{align*}
thanks to \eqref{eq:lpd:wl}, and (\ref{eq:lpd_pytJ}) follows.

\subsubsection*{Third case: $\lambda_n \rightarrow 0$ (up to a subsequence)}

Using Lemma~\ref{lm:projectgeneral} and then \eqref{eq:lpd:wl}, we get 
\begin{align*}&
\langle \vec \psi^{(J)} _{\Omega, \mathcal{O},n}, \vec w_n^{(J-1)}  \rangle_{\mathcal{H}(\Omega)}= \langle \vec \SO(-t_n) P_{\Omega} T_{\mathcal O, n} \psi^{(J)}, \vec w_n^{(J-1)} \rangle_{\mathcal{H} (\Omega)} \\
&= \langle P_{\Omega_n}\vec \psi^{(J)}, T_{\mathcal O^{-1}, n}S_{\Omega}(t_n)\vec w_n^{(J-1)}  \rangle_{\mathcal{H} (\Omega_n)} =\langle P_{\Omega_\infty} \vec \psi^{(J)}, T_{\mathcal O^{-1}, n}S_{\Omega}(t_n)\vec w_n^{(J-1)}  \rangle_{\mathcal{H} (\R^3)} + o(1)\\
&= \Vert P_{\Omega_\infty} \vec \psi^{(J)} \Vert^2_{\mathcal{H}(\Omega_\infty)} + o(1) = \Vert  \vec \psi^{(J)}_{\Omega, \mathcal O, n} \Vert^2_{\mathcal{H}(\Omega)} + o(1),
\end{align*}
from which (\ref{eq:lpd_pytJ}) follows.

\subsection*{The $L^6$ expansion (\ref{eq:lpd_L6})} In the same way, by the induction assumption, using (\ref{eq:lpd_L6}) at rank $J-1$, it suffices to show that
\begin{equation} \label{eq:lpd_L6J}
\Vert  w_n^{(J-1)} \Vert^6_{L^6} = \Vert \vec w_n^{(J)} \Vert^6_{L^6} + \Vert  \psi^{(J)} _{\Omega, \mathcal{O},n} \Vert^6_{L^6} + o(1)
\end{equation}
in order to show (\ref{eq:lpd_L6}) at rank $J$. We separate cases again.

\subsubsection*{First case: $\lambda_n \rightarrow \bar \lambda \in (0, +\infty)$ and $x_n \rightarrow \bar x \in \mathbb R^3$ (up to a subsequence).}

If, up to a subsequence, $t_n \rightarrow \bar t \in \mathbb R$, we can assume that $(t_n, \lambda_n, x_n) = (0, 1, 0)$ and $(\bar t, \bar \lambda, \bar x) = (0, 1, 0)$ by replacing the limiting profile $\vec \psi^{(J)}$ by $\SO(-\bar t) P_{\Omega}T \vec \psi^{(J)}$ with the 
previous notations, and we have, 
\begin{equation} \label{eq:lpd:wltwbis}
\vec w_n^{(J)} \rightharpoonup 0 \hspace{0.3cm} \text{ in } \mathcal H.
\end{equation} 
Now, observe that
$$
||z+w|^6-|z|^6-|w|^6| \lesssim |z||w|(|z|^4+|w|^4),
$$
hence, denoting
$$
f_n :=\Big| \int | w_n^{(J-1)} |^6 - |\psi^{(J)} _{\Omega, \mathcal{O},n}|^6 - |w_n^{(J)}|^6 \Big|,
$$
we have
$$
f_n \lesssim \int |\psi^{(J)} _{\Omega, \mathcal{O},n}||w_n^{(J)}|g_n, \hspace{0.5cm} g_n := |\psi^{(J)} _{\Omega, \mathcal{O},n}|^4+|w_n^{(J)}|^4.
$$
But, using (\ref{eq:lpd:fr}), Sobolev embedding and conservation of energy,
\begin{multline*}
\Vert g_n \Vert^{\frac 3 2}_{L^{\frac 3 2}} \lesssim \Vert\psi^{(J)} _{\Omega, \mathcal{O},n} \Vert^6_{L^{6}} + \Vert w_n^{(J)} \Vert^6_{L^{6}} \lesssim \Vert\psi^{(J)} _{\Omega, \mathcal{O},n} \Vert^6_{L^{6}} + \Vert w_n^{(J-1)} \Vert^6_{L^{6}} 
 \lesssim \Vert\psi^{(J)} _{\Omega, \mathcal{O},n} \Vert^6_{L^{6}} + 1 \\
  \lesssim \Vert \mathcal P_\Omega T_{\mathcal O, n} \psi^{(J)} \Vert^6_{\mathcal H} + 1   \lesssim \Vert T_{\mathcal O, n} \psi^{(J)} \Vert^6_{\mathcal H(\mathbb R^3)} + 1 = 
  \Vert \psi^{(J)} \Vert^6_{\mathcal H(\mathbb R^3)} + 1,
\end{multline*}
hence, by H\"older inequality 
$$
f_n  \lesssim \int |\psi^{(J)} _{\Omega, \mathcal{O},n}|^3|w_n^{(J)}|^3
= 
\int |P_{\Omega} \vec \psi^{(J)}|^3|w_n^{(J)}|^3 \rightarrow 0,
$$
up to a subsequence, from which  (\ref{eq:lpd_L6J}) follows: indeed, $|w^{(J)}_n|^3$ is bounded in $L^2$ so have a weakly converging subsequence in $L^2$,
and from (\ref{eq:lpd:wltwbis}) together with Rellich theorem, after another extraction it converges strongly to zero in $L^{\frac 4 3}(K)$ for any compact $K$; hence, by uniqueness of the limit in the sense of distributions, $|w^{(J)}_n|^3 \rightharpoonup 0$ in $L^2$.
In the case where $t_n \rightarrow \infty$, we can still assume that $(\lambda_n, x_n) = (1, 0)$ and $(\bar \lambda, \bar x) = (1, 0)$ by replacing the limiting profile $\vec \psi^{(J)}$ by $T \vec \psi^{(J)}$. Then, observe that
$$
\Vert \psi^{(J)} _{\Omega, \mathcal{O},n} \Vert_{L^6} = \Vert \SO(t_n) P_\Omega \vec  \psi^{(J)} \Vert_{L^6} \rightarrow 0,
$$
by linear scattering (Lemma \ref{lem:lin_scat} together with an approximation argument) and because the analog is true for the linear flow in $\mathbb R^3$; and hence  (\ref{eq:lpd_L6}) follows as well.

\subsubsection*{Second case: $\lambda_n \rightarrow \infty$ or $x_n \rightarrow \infty$ (up to a subsequence)}

We first assume that, up to a subsequence, $\frac {t_n}{\lambda_n}$ has a finite limit $\tau$. Then, we can assume that $t_n = 0$ by replacing $\psi^{(J)}$ by $ \SF(- \tau) \vec \psi^{(J)}$: indeed, using Lemma \ref{lem:assfree}, we have in $\mathcal H(\Omega)$
\begin{align*}
\vec{\psi}_{\Omega, \mathcal O, n}^{(J)}&=\SO(-t_n) P_\Omega T_{\mathcal O, n}\vec \psi^{(J)} = \SF(-t_n) T_{\mathcal O, n}\vec \psi^{(J)} + o(1) \\
&= T_{\mathcal O, n} \SF(-\frac{t_n}{\lambda_n}) \vec \psi^{(J)} + o(1) = T_{\mathcal O, n} \SF(-\tau) \vec \psi^{(J)} + o(1) \\
&= \SO(0)P_\Omega T_{\mathcal O, n} \SF(-\tau) \vec \psi^{(J)} + o(1).
\end{align*}
Then, having replaced $t_n$ by $0$ and $\psi^{(J)}$ by $ \SF(- \tau) \vec \psi^{(J)}$
$$
\vec \psi^{(J)}_{\mathcal O, n} = P_\Omega T_{\mathcal O, n} \vec \psi^{(J)} = T_{\mathcal O, n} \vec \psi^{(J)} + o(1),
$$
using Lemma~\ref{lm:projectgeneral} and that  $\vec \psi^{(J)}\in \mathcal{H}(X_{\OO})$ so $\vec \psi^{(J)}=P_{X_{\OO}}\vec \psi^{(J)}$. It follows that, arguing as previously
$$
f_n \lesssim o(1) + \int |T_{\mathcal O, n} \psi^{(J)}|^3|w_n^{(J)}|^3 = 
\int | \psi^{(J)}|^3|T_{\mathcal O^{-1}, n} w_n^{(J)}|^3 \rightarrow 0,
$$
from which  (\ref{eq:lpd_L6J}) follows.

In the case where $\frac{t_n}{\lambda_n} \rightarrow \infty$, observe that
$$
\Vert \psi^{(J)} _{\Omega, \mathcal{O},n} \Vert_{L^6} = \Vert \SO(t_n) P_\Omega   T_{\mathcal O, n} \vec \psi^{(J)} \Vert_{L^6} = \Vert \SF(t_n) T_{\mathcal O, n} \vec \psi^{(J)} \Vert_{L^6} + o(1),
$$
where
$$
\Vert \SF(t_n) T_{\mathcal O, n} \vec \psi^{(J)} \Vert_{L^6}  =
\Vert T_{\mathcal O, n} \SF(\frac{t_n}{\lambda_n}) \vec \psi^{(J)} \Vert_{L^6} =
\Vert \SF(\frac{t_n}{\lambda_n}) \vec \psi^{(J)} \Vert_{L^6} 
\rightarrow 0,
$$
and  (\ref{eq:lpd_L6J}) follows as well.

\subsubsection*{Third case: $\lambda_n \rightarrow 0$ (up to a subsequence)}

Again, let us assume first that, up to a subsequence, $\frac {t_n}{\lambda_n}$ has a finite limit $\tau$. By Lemma \ref{lmapproxconc}, there exists an isometry $Q$ of $\mathbb R^3$ so that
$$
\vec \psi^{(J)}_{\mathcal O, n} = T_{O, n} \big( S_{X_O}(-\frac {t_n}{\lambda_n})\vec \psi^{(J)}\big) \circ Q + o(1) =
T_{O, n} \big( S_{X_O}(-\tau)\vec \psi^{(J)} \big) \circ Q + o_{\mathcal{H}}(1),
$$
hence, once again, we can assume that $t_n = 0$ by replacing $\psi^{(J)}$ by $\big( S_{X_O}(-\tau)\vec \psi^{(J)} \big) \circ Q$. Arguing in the same way as previously, we then get
$$
f_n \lesssim o(1) +
\int | \psi^{(J)}|^3|T_{\mathcal O^{-1}, n} w_n^{(J)}|^3 \rightarrow 0,
$$
and  (\ref{eq:lpd_L6}) follows.
In the case where $\frac {t_n}{\lambda_n} \rightarrow \infty$, 
thanks to Proposition \ref{propnoncon} we have
$$
\Vert \psi^{(J)} _{\Omega, \mathcal{O},n}\Vert_{L^6} \rightarrow 0,
$$
and  (\ref{eq:lpd_L6J}) follows once again.

\subsection*{Orthogonality of scale cores}
We now show the orthogonality of cores up to $j=J$. To this end, let $0\leq j \leq J-1$, and let us show that $\mathcal O^{(J)}$ 
and $\mathcal O^{(j)}$ are orthogonal. We work by decreasing iteration and assume that we have proved that  the cores $\mathcal O^{(J)}$ 
and $\mathcal O^{(l)}$ are orthogonal for $j+1\leq l\leq J-1$.
An iteration of \eqref{eq:lpd:fr} gives $$
w_n^{(j)} = \sum_{k=j+1}^{J-1} \vec \psi^{(l)} _{\Omega, \mathcal{O}^{(l)},n} + w_n^{(J-1)}.$$ 
In particular, combined with the definition of $\vec r_n^{(j)}$ in   \eqref{eq:lpd:decweak}, this gives the identity
\begin{align*}
  &  T_{ {\mathcal (O^{(J)})}^{-1}, n}  \SO(t^{(J)}_n - t^{(j)}_n) T_{ {\mathcal (O^{(j)})}, n} \vec r_n^{(j)}=T_{ {\mathcal (O^{(J)})}^{-1}, n}  \SO(t^{(J)}_n )  \vec w_n^{(j)} \\
    =&T_{ {\mathcal (O^{(J)})}^{-1}, n}  \SO(t^{(J)}_n )  \left(\sum_{k=j+1}^{J-1} \vec \psi^{(l)} _{\Omega, \mathcal{O}^{(l)},n} + w_n^{(J-1)}\right).
\end{align*}
Since by iteration, the cores $\mathcal O^{(J)}$ 
and $\mathcal O^{(l)}$ are orthogonal, Lemma~\ref{lm:orth_weak} provides 
$$ \forall j+1\leq l\leq J-1, \quad T_{ {\mathcal (O^{(J)})}^{-1}, n}  \SO(t^{(J)}_n )   \vec \psi^{(l)} _{\Omega, \mathcal{O}^{(l)},n}\rightharpoonup 0.$$ 
In particular, the application of \eqref{eq:lpd:wl} gives
$$
  T_{ {\mathcal (O^{(J)})}^{-1}, n}  \SO(t^{(J)}_n - t^{(j)}_n) T_{ {\mathcal (O^{(j)})}, n} \vec r_n^{(j)} \rightharpoonup \vec \psi^{(J)} \neq \vec 0.
$$
By Lemma \ref{EquivFrames}, this implies that $\mathcal O^{(J)}$ 
and $\mathcal O^{(j)}$ are orthogonal: indeed, if they are equivalent, using successively that $T_{ {\mathcal (O^{(j)})}, n} \vec r_n^{(j)}\in \mathcal{H}(\Omega)$ so that $T_{ {\mathcal (O^{(j)})}, n} \vec r_n^{(j)}\in \mathcal{H}(\Omega)=P_{\Omega}T_{ {\mathcal (O^{(j)})}, n} \vec r_n^{(j)}\in \mathcal{H}(\Omega)$, then Lemma \ref{EquivFrames} assuming equivalence of the frames,  and (\ref{eq:lpd:decweak}), we get for any test function $\vec \rho\in \mathcal{H}(\R^3)$
\begin{align*}
&\langle T_{ {\mathcal (O^{(J)})}^{-1}, n}  \SO(t^{(J)}_n - t^{(j)}_n)T_{ {\mathcal (O^{(j)})}, n} \vec r_n^{(j)}, \vec \rho \rangle \\
&\hspace{0.5cm}= 
\langle\SO(-t^{(j)}_n)T_{ {\mathcal (O^{(j)})}, n} \vec r_n^{(j)}, \SO(-t^{(J)}_n){P_{\Omega}}T_{ {\mathcal (O^{(J)})}, n}  \vec \rho \rangle\\ 
&\hspace{0.5cm} = 
\langle\SO(-t^{(j)}_n)T_{ {\mathcal (O^{(j)})}, n} \vec r_n^{(j)}, \SO(-t^{(j)}_n){P_{\Omega}}T_{ {\mathcal (O^{(j)})}, n}   {Q P_{X_{\OO^{J}}}} \vec \rho \rangle + o (1) \\
&\hspace{0.5cm}= \langle \vec r_n^{(j)}, P_{\Omega_n^j} {Q} \vec \rho \rangle + o (1) \rightarrow 0,=
\langle \vec r_n^{(j)},  {Q P_{X_{\OO^{J}}}} \vec \rho \rangle + o (1) \rightarrow 0,
\end{align*}
where we have used Lemma~\ref{lm:projectgeneral} in the last line. This is a contradiction with $\psi^{(J)} \neq \vec 0$.
\subsection*{Conclusion of the induction procedure}
Hence, we constructed 
orthogonal cores
$\mathcal O^{(j)}$, profiles $\vec \psi^{(j)}$, and for any $J\geq 0$ an extraction $\vec \varphi_{n_k^J}$ of $\vec \varphi_n$, with $(n_k^J)_k \subset (n_k^{J-1})_k$, so that (\ref{eq:lpd}), (\ref{eq:lpd_pyt}), and (\ref{eq:lpd_L6}) hold, with in addition 
the decay property (\ref{eq:lpd:decstep}). By a diagonal argument, we obtain an extraction of $\vec \varphi_n$ so that (\ref{eq:lpd}), (\ref{eq:lpd_pyt}), (\ref{eq:lpd_L6}), (\ref{eq:lpd:decstep}) hold for any $J$, and we are left with verifying the decay of the remainder (\ref{eq:lpd_rem}).

\subsection*{Decay of the remainder (\ref{eq:lpd_rem}).}
This will be a consequence of (\ref{eq:lpd:decstep}) together with Lemma \ref{lem:ellpd}.
To obtain (\ref{eq:lpd_rem}), it suffices to show that
\begin{equation} \label{eq:lpd:suffdecay}
\lim_{J \rightarrow \infty} \limsup_{n\rightarrow \infty} \; \Vert \SO(\cdot) \vec w_n^{J} \Vert_{L^\infty(\mathbb R, L^6)}  = 0.
\end{equation}
Indeed, by Hölder inequality, Strichartz estimates of Lemma \ref{ass:strichartz}, and finally using the fact that, by  (\ref{eq:lpd_pyt}), for any $J$ we have $\limsup \Vert \vec w_n^{J} \Vert_{\mathcal H} \lesssim \limsup \Vert \vec \varphi_n \Vert_{\mathcal H} \lesssim 1$
\begin{align*}
\Vert \SO(\cdot) \vec w_n^{J} \Vert_{L^5(\mathbb R, L^{10})} &\leq
\Vert \SO(\cdot) \vec w_n^{J} \Vert_{L^\infty(\mathbb R, L^{6})} ^{\theta} \Vert \SO(\cdot) \vec w_n^{J} \Vert_{L^{r}(\mathbb R, L^{s})} ^{1-\theta} \\
&\lesssim \Vert \SO(\cdot) \vec w_n^{J} \Vert_{L^\infty(\mathbb R, L^{6})} ^{\theta} \Vert \vec w_n^{J} \Vert_{\mathcal H} ^{1-\theta} \lesssim \Vert \SO(\cdot) \vec w_n^{J} \Vert_{L^\infty(\mathbb R, L^{6})} ^{\theta}.
\end{align*}
with $\theta= \frac{3(s-10)}{5(s-6)}>0$  since $s>10$. Let us therefore show (\ref{eq:lpd:suffdecay}). Observe that, by conservation of energy
and change of variable
$$
\Vert \vec  \psi^{(j)}_{\Omega, \mathcal{O}^{(j)},n} \Vert^2_{\mathcal H(\Omega)} 
= \Vert P_{\Omega}T_{\mathcal{O}^{(j)},n} \vec \psi^{(j)}\Vert^2_{\mathcal H(\Omega)} 
= \Vert T_{(\mathcal{O}^{(j)})^{-1},n} P_{\Omega}T_{\mathcal{O}^{(j)},n} \vec \psi^{(j)}\Vert^2_{\mathcal H(\Omega_n^j)}
$$
where $T_{(\mathcal{O}^{(j)})^{-1},n} P_{\Omega}T_{\mathcal{O}^{(j)},n} \vec \psi^{(j)} 
= P_{\Omega_n^j} \vec \psi^{(j)} \rightarrow P_{{X_{{\mathcal O}^{(j)}}}}\vec \psi^{(j)}$ in $\mathcal H(\mathbb R^3)$ by Lemma~\ref{lm:projectgeneral}, hence
$$
\Vert \vec  \psi^{(j)}_{\Omega, \mathcal{O}^{(j)},n} \Vert^2_{\mathcal H(\Omega)} \rightarrow \Vert P_{\mathcal{O}^{(j)}} \vec \psi^{(j)} \Vert_{\mathcal H(X_{\mathcal{O}^{(j)}})},
$$
and, by the Pythagorean expansion (\ref{eq:lpd_pyt}), the serie 
$
\sum_{j=1}^\infty \Vert P_{\mathcal{O}^{(j)}} \vec \psi^{(j)} \Vert_{\mathcal H(X_{\mathcal{O}^{(j)}})}
$
 converges. It follows that its general term goes to zero, and hence, as $J \rightarrow \infty $,
\begin{equation} \label{eq:lpd:decetaw}
\eta((w_n^J)_n) \rightarrow 0.
\end{equation}
Now, if (\ref{eq:lpd:suffdecay}) fails, by a diagonal argument, there exists $\epsilon>0$, a sequence $(t_k)$ and subsequences $(n_k)_k$
and $(J_k)$ so that
$$
\Vert \SO(t_k) w_{n_k}^{J_k} \Vert_{L^6(\Omega)} \geq \epsilon.
$$
By Lemma \ref{lem:ellpd}, it follows that there exists a scale core $\mathcal O = \{(t_k)_k, (x_k)_k, (\lambda_k)_k\}$ and $\vec \psi \in \mathcal H(\mathbb R^3) \neq \vec 0$, so that, up to a subsequence
$$
T_{\mathcal O^{-1}, k} \SO(t_k) w_{n_k}^{J_k} \rightharpoonup \vec \psi  \hspace{0.3cm} \text{ in }\mathcal{H} (\mathbb R^3).
$$ 
But 
$$
T_{\mathcal O^{-1}, k} \SO(t_k) w_{n_k}^{J_k} = \big(T_{\mathcal{O}^{-1},k} P_{\Omega}T_{\mathcal{O},k} \big) T_{\mathcal O^{-1}, k} \SO(t_k) w_{n_k}^{J_k},
$$
and as, in addition, $T_{\mathcal{O}^{-1},k} P_{\Omega}T_{\mathcal{O},k} \vec \psi \rightarrow P_{X_{\mathcal O}} \vec \psi$ by Lemma~\ref{lm:projectgeneral} again, 
we deduce that $\vec \psi = P_{X_{\mathcal O}} \vec \psi \in \mathcal H(X_{\OO})$. The fact that $\vec \psi \neq \vec 0$ is then a contradiction with the definition of $\eta$ and (\ref{eq:lpd:decetaw}).

\end{proof}

\section{Construction of a compact flow solution} \label{sec:cc}

Let $E_c(\Omega)$ be defined as
\begin{equation}
E_c(\Omega) := \sup \Big\{ E>0 \; \text{ s.t. } \;\exists C(E)>0, \; \EO(\vec \varphi)<E \implies  \Vert \NLSO(\cdot)\vec \varphi \Vert_{L^5(\mathbb R, L^{10}(\Omega))} \leq C(E) \Big\}.
\end{equation}
By the small-data theory, $E_c >0$. The goal of this section is to show the following Theorem. Observe in particular that it immediately implies Theorem \ref{th:const_intro}.

\begin{thm}[Construction of a compact-flow solution.] \label{th:const}
Assume that $\Omega$ has a smooth, compact boundary and verifies Assumptions \ref{ass:nonreco}, \ref{ass:weaktrap}, \ref{ass:strichartz}. Then, if $E_c(\Omega)<+\infty$, there exists $\vec \varphi_c \in \mathcal H(\Omega)$, $\vec \varphi_c \neq \vec 0$ with $\EO(\vec \varphi_c)=E_c$, so that the nonlinear flow
$ \big\{ \NLSO(t)\vec \varphi_c, \; t\in \mathbb R \big\} $
is relatively compact in $\mathcal H(\Omega)$ and $\NLSO(\cdot)\vec \varphi_c \notin L^5(\mathbb R, L^{10})$. 
\end{thm}

\begin{proof}
If $E_{c}<+\infty$, let $\vec{\varphi}_{0}^{n}$ be a minimizing
sequence for $E_{c}(\Omega)$, in the sense that
\begin{equation}
\EO (\vec{\varphi}_{0}^{n})\geq  E_{c}(\Omega),\; \lim_{n\rightarrow\infty}\EO(\vec{\varphi}_{0}^{n})=E_{c}(\Omega),\; \Vert \NLSO(\cdot)\vec{\varphi}_{0}^n\Vert_{L^{5}L^{10}}\to \infty,\label{eq:suite_u0n}
\end{equation}
where by convention $\|u_n\|_{L^{5}\left(\R,L^{10}\right)}=+\infty$ if $u_n\notin L^{5}\left(\R,L^{10}\right)$. 
Let us define 
$$
u_n : =\NLSO(\cdot)\vec{\varphi}_0^n.
$$
Translating in time, if necessary, we may assume
\begin{equation}
 \label{eq:symL5L10infty}
\lim_{n\to\infty} \|u_n\|_{L^{5}\left((0,+\infty),L^{10}\right)}=\lim_{n\to\infty} \|u_n\|_{L^{5}\left((-\infty,0),L^{10}\right)}=+\infty,
 \end{equation} 
with a similar convention as we did on $\R$, \textit{i.e.}\ by convention $\|u_n\|_{L^{5}\left((-\infty,0),L^{10}\right)}=+\infty$ if $u_n\notin L^{5}\left((-\infty,0),L^{10}\right)$, and similarly for $L^{5}\left((0,\infty),L^{10}\right)$.

As $\vec{\varphi}_{0}^{n}$ is bounded in $\mathcal{H}(\Omega)$, we
can, up to a subsequence, decompose it into profiles according to Theorem
\ref{th:lindec}: 
\begin{equation}
\vec \varphi_0^{n} = \sum_{j=1}^J \vec \psi^{(j)} _{\Omega, \mathcal{O}^{(j)},n} + \vec w_n^{(J)}.
\end{equation}
To each profile $(\vec \psi^{(j)}, \mathcal{O}^{(j)}) =:(\vec \psi^{(j)}, (t_{j,n}, \lambda_{j,n}, x_{j, n}))$, we will associate a Dirichlet nonlinear profile $U^{j}$, and possibly a free nonlinear profile $V^{j}$, in the following way.
\begin{itemize}
\item \textbf{(compact)} If $\lambda_{j,n} = 1$ and $x_{j, n} = 0$ for any $n$: we will write $j\in J_{\rm comp}$. 
If $t_{j,n} = 0$, let ${U^{(j)}}$ be the only
solution of (\ref{eq:NLWomega}) in $\Omega$ with Cauchy data $\vec{\psi}^{(j)}$ at time zero. If $t_{j,n} \to \pm \infty$, let ${U^{(j)}}$ be the only
solution of (\ref{eq:NLWomega}) in $\Omega$ such that
$$
\left\Vert{\vec{U}^{(j)}}(-t)-\vec{\SO}(-t)\vec{\psi}^{(j)}\right\Vert_{\mathcal{H}(\Omega)} \to 0 \quad \text{as} \; t\to\pm \infty.
$$
Recall that the existence of this solution, for example in the case $t \to +\infty$, is obtained as the solution of the fixed point
$$
\vec U(t) = \vec{\SO}(t)\vec{\psi}^{(j)} - \int_t^{+\infty} {\vec\SO}(t-s)\big(0, U^5 (s)\big) \, ds,
$$
working in $L^5([T_0, +\infty),L^{10}(\Omega))$ for $T_0 > 0$ big enough, using Strichartz estimates from Assumption \ref{ass:strichartz} (recall in particular that these homogeneous estimates imply inhomogeneous ones with source in $L^1 L^2$ by Minkowski).
In these three cases,
\begin{equation}
\lim_{n\rightarrow\infty}\left\Vert{\vec{U}^{(j)}}({-t_{j,n}})-\vec{\SO}(-t_{j,n})\vec{\psi}^{(j)}\right\Vert_{\mathcal{H}(\Omega)}=0.\label{eq:def_Uj_comp}
\end{equation} 
We set
\begin{equation}
U_{n}^{j}(t):=U^{j}({t-t_{j,n}}).\label{eq:def_Vnj_comp}
\end{equation}
Notice that, if $-t_{j,n}\to \pm\infty$, $U^{j}\in L^{5}(\mathbb{R}_{\pm},L^{10}(\Omega))$ by construction.

\item \textbf{(asymptotically free)} If $\lambda_{j,n} \rightarrow \infty $ or $x_{j, n} \rightarrow \infty$: we will write $j \in J_{\rm diff}$. We have, by Lemma \ref{lem:assfree}

$$
\lim_{n\to\infty} \left\|\vec \psi^{(j)} _{\Omega, \mathcal{O}^{(j)},n} - \vec \psi^{(j)} _{\mathbb R^3, \mathcal{O}^{(j)},n}\right\|_{\mathcal H{(\Omega)}}=0.
$$
Furthermore, denoting by $V_L^j(t):=\SF(t){\psi}^{(j)}$,
$$
\SF(t - t_{j,n}) \psi^{(j)} _{\mathbb R^3, \mathcal{O}^{(j)},n} =\frac{1}{\lambda_{j,n}^{1/2}}V^j_L\left(\frac{t-t_{j,n}}{\lambda_{j,n}},\frac{x-x_{j,n}}{\lambda_{j,n}}  \right).
$$
We define the \emph{free} nonlinear profile $V^j$ as the unique solution of the critical nonlinear wave equation in $\mathbb R^3$ such that 
\begin{equation} \label{eq:def_Uj_diff}
\lim_{n\to \pm\infty} \left\|\vec{V}^{j}(-t_{j,n}/\lambda_{j,n})-\vec{V}_L^{j}(-t_{j,n}/\lambda_{j,n})\right\|_{\mathcal{H} (\R^3})=0.
\end{equation}
This is possible distinguishing the cases $-t_{j,n}/\lambda_{j,n}$ convergent and using the Cauchy theory or $-t_{j,n}/\lambda_{j,n}\to -\infty$ using the scattering theory on $\R^3$ (see Proposition \ref{prop:perturb_scat} for instance).
{ Furthermore, we set
\begin{equation}
\label{eq:def_Vnj_diff}
{V}_{n}^{j}(t, x):=\frac{1}{\lambda_{j,n}^{1/2}}V^{j}\Big(\frac{t-t_{j,n}}{\lambda_{j,n}}, \frac{x-x_{j,n}}{\lambda_{j,n}}\Big),
\end{equation}
and we then define the associated family of Dirichlet nonlinear profiles as
\begin{equation}
U_{n}^{j}(t):=\NLSO(t-t_n^j)P_{\Omega}\left(\vec{V}_n^j(t_n^j)\right).
\label{eq:def_Vnj_diff}
\end{equation} 
Observe that, as a solution
of a defocusing nonlinear wave equation in $\mathbb{R}^{3}$, for
which the scattering is well known, we have $V^j\in L^{5}L^{10}(\mathbb{R}^{3})$.
Furthermore, by construction, from Lemma \ref{lem:assfree_nl}
$$
\forall j\in J_{\text {diff}},\hspace{1em}\sup_n \Vert U_n^j\Vert_{L^{5}(\mathbb{R},L^{10}(\Omega))} < \infty,
$$
and 
\begin{equation}
\forall j\in J_{\text{diff}},\hspace{1em}\sup_{t}\left\|\vec{V}_{n}^{j}(t)-\vec{U}_{n}^{j}(t)\right\|_{\mathcal H (\Omega)}+\Vert V_{n}^{j}-U_{n}^{j}\Vert_{L_{t}^{5}L_{x}^{10}}\underset{n\to\infty}{\longrightarrow}0.\label{eq:comp_U_V}
\end{equation} }

\item \textbf{(concentrating)}  If $\lambda_{j,n} \rightarrow 0$ and $x_{j,n}$ is bounded: we write $j \in J_{\rm conc}$. In this case,
 let $U^{j}_n$ be given by Proposition \ref{prop:concprof} (applied to the profile $ \psi^{(j)} _{\Omega, \mathcal{O}^{(j)},n}$). In addition, we define
 \begin{equation} \label{eq:conc_v}
 v_n^{j, \pm}(t) := \SO(t-t_{j,n}) T_{\mathcal{O}^{(j)},n}^{\Omega}\phi_j^{\pm\infty},
\end{equation}
where $\phi_j^{\pm\infty}$ are given by Proposition \ref{prop:concprof} as well.
\end{itemize}

Let us assume by contradiction that the decomposition
has strictly more than one non trivial profile, i.e
\begin{equation}
J>1.\label{eq:more_than_one_profile}
\end{equation}
Then, by the Pythagorean expansion and
 its $L^{6}$ version 
\[
\forall j\in J_{\text{comp}},\hspace{1em}\limsup_{n\rightarrow\infty}\;\EO\left(\SO(-t_{j,n})\vec{\psi}^{j}\right)<E_{c}.
\]
Hence, $\EO(U^j)<E_C$,
and ${U^{j}}\in L^{5}L^{10}(\Omega)$ by the definition of the critical
energy. Summing up, we have 
\begin{equation} \label{eq:tout_est_diffusif}
\begin{cases}
  \forall j, & \sup_n \Vert U^{j}_n \Vert_{{L^5}L^{10}} < +\infty, \\
\forall j \in J_{\rm comp}, & {U^{j}}\in L^{5}L^{10}(\Omega), \\
 \forall j \in J_{\rm diff}, & V^j \in L^{5}L^{10}(\mathbb R^3).
\end{cases}
\end{equation} 
The Theorem will follow from the following nonlinear
profile decomposition. 
\begin{prop}
\label{prop:non_lin_prof}We have
\begin{align}
\forall J,\ u_{n}(t)  &=\sum_{1\leq j \leq J} U_n^j(t) + R_{n}^{J}(t) \label{eq:dec_uj} \\
&= \sum_{\substack{j\in J_{\rm comp} \cup J_{\rm conc}\\1\leq j\leq J}} U_n^j(t)+ \sum_{\substack{j\in J_{\rm diff}\\1\leq j\leq J}} V_n^j(t) + \tilde R_{n}^{J}(t), \nonumber
\end{align}
where
\begin{equation*}
\lim_{J\rightarrow\infty}\limsup_{n\rightarrow\infty}\Vert R_{n}^{J}\Vert_{L^{5}L^{10}}= \lim_{J\rightarrow\infty}\limsup_{n\rightarrow\infty}\Vert \tilde R_{n}^{J}\Vert_{L^{5}L^{10}}=0.
\end{equation*}
\end{prop}
In order to prove the above, let
\begin{equation}
{\breve u}_{n}^{J}:=\sum_{j=1}^{J}U_{n}^{j}.
\end{equation}
Observe that $\breve u_{n}^{J}$ is solution in $\Omega$ of the following nonlinear wave
equation with Dirichlet boundary conditions:
\begin{equation}
(\partial_{t}^{2}-\Delta_{N}){\breve u}_{n}^{J}+({\breve u}_{n}^{J})^{5}=e_{n}^{J},\hspace{1em}\text{with }e_{n}^{J}:=({\breve u}_{n}^{J})^{5}-\sum_{j=1}^{J}(U_{n}^{j})^{5}.\label{eq:eq_vnj}
\end{equation}
Proposition \ref{prop:non_lin_prof} will come perturbatively from the following.
\begin{lem} \label{lem:nl_prof_pert}
We have
\begin{equation} \label{eq:orthNL}
    \forall j\neq k, \quad |U_{n}^{j}|^{4}|U_{n}^{k}| \to 0 \text{ in }L^1L^2,
\end{equation}
thus in particular
\begin{equation}
\lim_{J\to\infty}\limsup_{n\rightarrow\infty}\Vert e_{n}^{J}\Vert_{L^{1}L^{2}}=0,\label{eq:decay_en_J}.
\end{equation}
In addition,
\begin{equation}
\vec {\breve u}_{n\restriction t = 0}^J
=\vec{u}_{n\restriction t = 0}+\vec{\alpha}_{n}^{J};\hspace{1em}\ \lim_{J\to\infty}\limsup_{n\to\infty}\Vert \SO(\cdot) \vec \alpha_{n}^{J}\Vert_{L^5L^{10}}=0.\label{eq:v_n_data}
\end{equation}
\end{lem}

\begin{proof}
\textbf{(A) The $L^1 L^2$ decay (\ref{eq:decay_en_J}).}

We will first show (\ref{eq:decay_en_J}). Observe that

\begin{equation}
|e_{n}^{J}| \lesssim_{J}\sum_{1\leq j\neq k\leq J}|U_{n}^{j}|^{4}|U_{n}^{k}|. \label{eq:L1L2dec_reduc}
\end{equation}
We will treat the three above terms separately, beginning with the mixed terms $|U_{n}^{j}|^{4}|U_{n}^{k}|$. In order to do so, we have essentially six cases to consider. We highlight that the first three cases, corresponding to $(j,k)\in J_{\rm diff}^2 \cup J_{\rm diff} \times J_{\rm comp}\cup J_{\rm comp}^2$, were already treated in \cite{DL} (in the radial case, but this part of the proof holds in the same way).

\textbf{(A.1) The mixed term $|U_{n}^{j}|^{4}|U_{n}^{k}|$ for $j,k\in J_{\text{diff}}$.} 

We first show
\begin{equation} \label{eq:A1}
\left\Vert|U_{n}^{j}|^{4}|U_{n}^{k}|\right\Vert_{L^{1}L^{2}}\longrightarrow0,\ \text{for }j,k\in J_{\text{diff}}.
\end{equation}
The proof is the same as in \cite{DL}, we reproduce it for completeness.
Note that
$$
|U_{n}^{j}|^{4}|U_{n}^{k}|\leq|V_{n}^{j}|^{4}|V_{n}^{k}|+|U_{n}^{j}|^{4}|V_{n}^{k}-U_{n}^{k}|+|V_{n}^{j}||V_{n}^{k}-U_{n}^{k}|^{4},
$$
thus, by Hölder inequality
\begin{equation}
\left\Vert|U_{n}^{j}|^{4}|U_{n}^{k}|\right\Vert_{L^{1}L^{2}}\leq\left\Vert|V_{n}^{j}|^{4}|V_{n}^{k}|\right\Vert_{L^{1}L^{2}}+\left\Vert U_{n}^{j}\right\Vert_{L^{5}L^{10}}^{4}\left\Vert V_{n}^{k}-U_{n}^{k}\right\Vert_{L^{5}L^{10}}
+\left\Vert V_{n}^{j}\right\Vert_{L^{5}L^{10}}\left\Vert V_{n}^{k}-U_{n}^{k}\right\Vert_{L^{5}L^{10}}^{4}.\label{eq:ineq_Vnjk_jk_diff}
\end{equation}
On the one hand, as $V_{n}^{j}$ and $V_{n}^{k}$ are rescaled solutions
of the defocusing critical nonlinear wave equation in $\mathbb{R}^{3}$ associated with orthogonal parameters,
we have, as $n$ goes to infinity (see for example \cite{BahouriGerard99})
\begin{equation}
\left\Vert|V_{n}^{j}|^{4}|V_{n}^{k}|\right\Vert_{L^{1}L^{2}}\longrightarrow0.\label{eq:ineq_Vnjk_jk_diff_1}
\end{equation}
On the other hand, as
$$
\sup_{n}\ \left\Vert U_{n}^{j}\right\Vert_{L^{5}L^{10}}+\left\Vert V_{n}^{j}\right\Vert_{L^{5}L^{10}}<\infty,
$$
it follows from (\ref{eq:comp_U_V}) that
\begin{equation}
\left\Vert U_{n}^{j}\right\Vert_{L^{5}L^{10}}^{4}\left\Vert V_{n}^{k}-U_{n}^{k}\right\Vert_{L^{5}L^{10}}+\left\Vert V_{n}^{j}\right\Vert_{L^{5}L^{10}}\left\Vert V_{n}^{k}-U_{n}^{k}\right\Vert_{L^{5}L^{10}}^{4}\longrightarrow0\label{eq:ineq_Vnjk_jk_diff_2}
\end{equation}
as $n$ goes to infinity, and thus (\ref{eq:ineq_Vnjk_jk_diff}) combined
with (\ref{eq:ineq_Vnjk_jk_diff_1}) and (\ref{eq:ineq_Vnjk_jk_diff_2})
gives (\ref{eq:A1}).

\textbf{(A.2)  The mixed term $|U_{n}^{j}|^{4}|U_{n}^{k}|$ for $j\in J_{\rm comp}, k \in J_{\rm diff}$.} 

We now show
\begin{equation} \label{eq:A2}
\left\Vert|U_{n}^{j}|^{4}|U_{n}^{k}|\right\Vert_{L^{1}L^{2}}\longrightarrow0,\ \text{for }j\in J_{\rm comp}, k\in J_{\rm diff}.
\end{equation}
We follow \cite{DL} once more. Similarly as before,
\begin{equation}
\left\Vert|U_{n}^{j}|^{4}|U_{n}^{k}|\right\Vert_{L^{1}L^{2}}\leq
\left\Vert|U_{n}^{j}|^{4}|V_{n}^{k}|\right\Vert_{L^{1}L^{2}}+\left\Vert U_{n}^{j}\right\Vert_{L^{5}L^{10}}^4\left\Vert V_{n}^{k}-U_{n}^{k}\right\Vert_{L^{5}L^{10}}.\label{eq:Vnj_Vnk_compdiff_1}
\end{equation}
On the one hand, we already saw that for $k\in J_{\text{diff}}$
\begin{equation}
\left\Vert U_{n}^{j}\right\Vert_{L^{5}L^{10}}^4\left\Vert V_{n}^{k}-U_{n}^{k}\right\Vert_{L^{5}L^{10}}\underset{n\to\infty}{\longrightarrow}0.\label{eq:Vnj_Vnk_compdiff_2}
\end{equation}
On the other hand, by Hölder inequality and change of variables
\begin{multline*}
\left\Vert|U_{n}^{j}|^{4}|V_{n}^{k}|\right\Vert_{L^{1}L^{2}}\leq\left\Vert U_{n}^{j}\right\Vert_{L^{5}L^{10}}^{3}\left\Vert V_{n}^{k}U_{n}^{j}\right\Vert_{L^{5/2}L^{5}}\\
=\left\Vert U^{j}\right\Vert_{L^{5}L^{10}}^{3}\frac{1}{\sqrt{\lambda_{k,n}}}\Big(\int\Big(\int_{\Omega}U^{j}(s,x)^{5}V^{k}\Big(\frac{s+t_{j,n}-t_{k,n}}{\lambda_{k,n}},\frac{x-x_{k,n}}{\lambda_{k,n}}\Big)^{5}\ dx\Big)^{1/2}ds\Big)^{2/5}.
\end{multline*}
As the above expression is uniformly continuous in $(U^j, V^{k})\in L^{5}L^{10}$,
we can assume that both are continuous and compactly supported. Then, if $\lambda_{n,k}$ is bounded and $|x_{n,k}|\rightarrow+\infty$, the above vanishes for $n$ big enough. On the other hand, if $\lambda_{k,n}\rightarrow +\infty$, we get
\begin{equation}
\left\Vert|U_{n}^{j}|^{4}|V_{n}^{k}|\right\Vert_{L^{1}L^{2}}\lesssim\frac{1}{\sqrt{\lambda_{k,n}}}\longrightarrow0,\label{eq:Vnj_Vnk_compdiff3}
\end{equation}
and thus by (\ref{eq:Vnj_Vnk_compdiff_1}), (\ref{eq:Vnj_Vnk_compdiff_2})
and (\ref{eq:Vnj_Vnk_compdiff3}), (\ref{eq:A2}) follows.

\textbf{(A.3) The mixed term $|U_{n}^{j}|^{4}|U_{n}^{k}|$ for $j, k \in J_{\rm comp}$.} 

Let us show
\begin{equation} \label{eq:A3}
\left\Vert|U_{n}^{j}|^{4}|U_{n}^{k}|\right\Vert_{L^{1}L^{2}}\longrightarrow0,\ \text{for }j,k\in J_{\rm comp}.
\end{equation}
Observe that
\begin{equation*}
\left\Vert|U_{n}^{j}|^{4}|U_{n}^{k}|\right\Vert_{L^{1}L^{2}}
=\int\Big(\int_{\Omega}U^{j}\big({t-t_{j,n}},{x}\big)^{8}U^{k}\big({t-t_{k,n}},{x}\big)^{2}\ dx\Big)^{1/2}dt.\label{eq:Vnj_Vnk_compcomp_1}
\end{equation*}
Hence, by change of variable $s=t-t_{j,n}$ 
\[
\left\Vert|U_{n}^{j}|^{4}|U_{n}^{k}|\right\Vert_{L^{1}L^{2}}=\int\Big(\int_{\Omega}U^{j}\big({s},{x}\big)^{8}U^{k}\big({s+t_{j,n}-t_{k,n}},{x}\big)^{2}\ dx\Big)^{1/2}ds.
\]
As this expression is uniformly continuous in $(U^{j},U^{k})\in L^{5}L^{10}$,
we may assume that both are continuous and compactly supported. But
for such functions, the above expression vanishes for $n$ large enough
because of the orthogonality of the parameters
$$
|t_{j,n}-t_{k,n}|\longrightarrow+\infty,
$$
and therefore (\ref{eq:A3}) holds. 

\textbf{(A.4)  The mixed term $|U_{n}^{j}|^{4}|U_{n}^{k}|$ for $j\in J_{\rm diff}, k \in J_{\rm conc}$.} 

We now show
\begin{equation} \label{eq:A4}
\left\Vert|U_{n}^{j}|^{4}|U_{n}^{k}|\right\Vert_{L^{1}L^{2}}\longrightarrow0,\ \text{for }j\in J_{\rm diff}, k\in J_{\rm conc}.
\end{equation}
In order to do so, we use again
$$
\left\Vert|U_{n}^{j}|^{4}|U_{n}^{k}|\right\Vert_{L^{1}L^{2}}\leq\left\Vert|V_{n}^{j}|^{4}|U_{n}^{k}|\right\Vert_{L^{1}L^{2}}+\left\Vert U_{n}^{k}\right\Vert_{L^{5}L^{10}}\left\Vert V_{n}^{j}-U_{n}^{j}\right\Vert^4_{L^{5}L^{10}},
$$
where we already saw that for $j\in J_{\text{diff}}$
$$
\left\Vert U_{n}^{k}\right\Vert_{L^{5}L^{10}}^4\left\Vert V_{n}^{j}-U_{n}^{j}\right\Vert^4_{L^{5}L^{10}}\underset{n\to\infty}{\longrightarrow}0,
$$
and hence, in order to show (\ref{eq:A4}), it remains to show that
\begin{equation}
\left\Vert|V_{n}^{j}|^4U_{n}^{k}\right\Vert_{L^{1}L^{2}} \longrightarrow 0. \label{eq:A4_0}
\end{equation}
In order to do so, let $U^k$, $\widetilde U_n^k$, $\phi^{\pm \infty}_k$ , and for $\epsilon >0$ arbitrary, $T=T_k>0$ be given by Proposition \ref{prop:concprof} (applied to the profile $\vec \psi^{(k)} _{\Omega, \mathcal{O}^{(k)},n}$, recall that $k$ is here fixed).
On the one hand, recalling the definition of $v_n^{k, \pm}$ (\ref{eq:conc_v}), by Hölder inequality, the triangle inequality, and Proposition \ref{prop:concprof}, we have for $n$ big enough
\begin{align}
\left\Vert|V_{n}^{j}|^{4}U_{n}^{k}\right\Vert_{L^{1}(| t - t_{k,n}| \geq T \lambda_{k,n}, L^{2})}
&\leq
 \left\Vert V_{n}^{j}\right\Vert_{L^{5}L^{10}}^4\left\Vert  U_n^k\right\Vert_{L^{5}(| t - t_{k,n}| \geq T \lambda_{k,n}, L^{10})} \nonumber \\
&\leq
\left\Vert V_{n}^{j}\right\Vert_{L^{5}L^{10}(\mathbb R^3)}^4\left\Vert  U_n^k - v_n^{k, +}\right\Vert_{L^{5}( t - t_{k,n} \geq T \lambda_{k,n}, L^{10})} \nonumber \\
&\hspace{0.5cm}+ \Vert V^j \Vert_{L^{5}L^{10}(\mathbb R^3)}^4 \left\Vert  v_n^{k,+}\right\Vert_{L^{5}(| t - t_{k,n}| \geq T \lambda_{k,n}, L^{10})} \nonumber \\
&\hspace{0.5cm}+
\left\Vert V_{n}^{j}\right\Vert_{L^{5}L^{10}(\mathbb R^3)}^4\left\Vert  U_n^k - v_n^{k, -}\right\Vert_{L^{5}( t - t_{k,n} \leq - T \lambda_{k,n}, L^{10})} \nonumber \\
&\hspace{0.5cm}+ \Vert V^j \Vert_{L^{5}L^{10}(\mathbb R^3)}^4 \left\Vert  v_n^{k,-}\right\Vert_{L^{5}(| t - t_{k,n}| \geq T \lambda_{k,n}, L^{10})} \nonumber \\
&\leq 4 \Vert V^j \Vert_{L^{5}L^{10}(\mathbb R^3)}^4 \epsilon \label{eq:A4_1},
\end{align}
where we used Proposition \ref{propnoncon} together with Hölder inequality and Strichartz estimates to control the Strichartz norm of $v_n^{k, \pm}$ outside the concentration times.
On the other hand, still by Proposition \ref{prop:concprof},
\begin{align*}
 &\left\Vert|V_{n}^{j}|^{4}U_{n}^{k}\right\Vert_{L^{1}(| t - t_{k,n}|\leq T \lambda_{k,n}, L^{2})}  \\
&\hspace{1cm}\leq \left\Vert|V_{n}^{j}|^{4}\tilde U_n^k\right\Vert_{L^{1}L^{2}}+\left\Vert V_{n}^{j}\right\Vert_{L^{5}L^{10}}^4\left\Vert  U_n^k-\tilde U_n^k\right\Vert_{L^{5}(| t - t_{k,n}| \leq T \lambda_{k,n}, L^{10})} \nonumber \\
&\hspace{1cm} \leq \left\Vert|V_{n}^{j}|^{4}\tilde U_n^k\right\Vert_{L^{1}L^{2}}+\Vert V^j \Vert_{L^{5}L^{10}(\mathbb R^3)}^4 \epsilon.
\end{align*}
Combining the above with (\ref{eq:A4_1}), in order to show (\ref{eq:A4_0}) and hence to conclude, it only remains to show that, as $n\rightarrow \infty$
\begin{equation} \label{eq:A4_3}
\left\Vert|V_{n}^{j}|^{4} \tilde U_n^k\right\Vert_{L^{1} L^{2}} \longrightarrow 0.
\end{equation}
This follows by the orthogonality of the parameters. Indeed, by Hölder inequality,
\begin{multline*}
\left\Vert|V_{n}^{j}|^4 \tilde U_n^k \right\Vert_{L^{1}L^{2}}\leq\left\Vert  V_{n}^{j}\right\Vert_{L^{5}L^{10}}^{3}\left\Vert V_{n}^{j} \tilde U_n^k \right\Vert_{L^{5/2}L^{5}} 
= \Vert V^j \Vert^3_{L^5 L^{10}} \frac{1}{\sqrt{\lambda_{k,n} \lambda_{j,n}}} \\ \times \Bigg( \int \Big( \int_\Omega \Big| U^k \big(\frac{t-t_{n,k}}{\lambda_{k,n}},\frac{x-x_{n,k}}{\lambda_{k,n}}) V^j \big(\frac{t-t_{n,j}}{\lambda_{j,n}},\frac{x-x_{n,j}}{\lambda_{j,n}})\Big|^5 dx \Big)^{\frac 1 2} dt \Bigg)^{\frac 2 5},
\end{multline*}
and, by change of variable, we get the two inequalities
\begin{equation*}
\left\Vert V_{n}^{j} \tilde U_n^k \right\Vert_{L^{5/2}L^{5}} 
\leq 
\begin{cases}
\sqrt{\frac{\lambda_{k,n}}{ \lambda_{j,n}}}  \Bigg( \int \Big( \int_{} \Big| U^k \big(\tau, y) V^j \big(\frac{\tau \lambda_{k,n} + t_{k,n} -t_{n,j}}{\lambda_{j,n}},\frac{y \lambda_{k,n} + t_{k,n}-x_{n,j}}{\lambda_{j,n}})\Big|^5 dy \Big)^{\frac 1 2} d\tau \Bigg)^{\frac 2 5}, \\
\sqrt{\frac{\lambda_{j,n}}{ \lambda_{k,n}}}  \Bigg( \int \Big( \int_{} \Big| V^j \big(\tau, y) U^k \big(\frac{\tau \lambda_{j,n} + t_{j,n} -t_{n,k}}{\lambda_{k,n}},\frac{y \lambda_{j,n} + t_{j,n}-x_{n,k}}{\lambda_{k,n}})\Big|^5 dy \Big)^{\frac 1 2} d\tau \Bigg)^{\frac 2 5}.
\end{cases}
\end{equation*}
Observing that the above expressions are uniformly continuous in $(U^{j},V^{k})\in L^{5}L^{10}$,
and assuming that both are continuous and compactly supported, each regime of parameter orthogonality gives the desired result, and (\ref{eq:A4_3}) follows.
This finishes the proof of (\ref{eq:A4}).

\textbf{(A.5)  The mixed term $|U_{n}^{j}|^{4}|U_{n}^{k}|$ for $j\in J_{\rm comp}, k \in J_{\rm conc}$.} 

Let us show
\begin{equation} \label{eq:A5}
\left\Vert|U_{n}^{j}|^{4}|U_{n}^{k}|\right\Vert_{L^{1}L^{2}}\longrightarrow0,\ \text{for }j\in J_{\rm comp}, k\in J_{\rm conc}.
\end{equation}
This case is essentially a simplified version of the previous one (A.4).
Again, let $U^k$, $\tilde U_n^k$, and for $\epsilon >0$ arbitrary, $T=T_k>0$ be given by Proposition \ref{prop:concprof}.
On the one hand, by H\"older inequality, triangle inequality, and Proposition \ref{prop:concprof}, we have for $n$ big enough
\begin{align*}
&\left\Vert|U_{n}^{j}|^{4}U_{n}^{k}\right\Vert_{L^{1}(| t - t_{k,n}| \geq T \lambda_{j,n}, L^{2})} \\
&\hspace{1cm}\leq
\left\Vert U_{n}^{j}\right\Vert_{L^{5}L^{10}}^4\left\Vert  U_n^k\right\Vert_{L^{5}(| t - t_{k,n}| \geq T \lambda_{k,n}, L^{10})} \\
&\hspace{1cm}\leq \Vert U^j \Vert_{L^{5}L^{10}(\mathbb R^3)}^4 \Big( \left\Vert  U_n^k - v_n^{k,+}\right\Vert_{L^{5}( t - t_{k,n} \geq T \lambda_{k,n}, L^{10})} + \left\Vert v_n^{k,+}\right\Vert_{L^{5}(| t - t_{k,n}| \geq T \lambda_{k,n}, L^{10})} \Big) \\
&\hspace{1.3cm}+ \Vert U^j \Vert_{L^{5}L^{10}(\mathbb R^3)}^4 \Big( \left\Vert  U_n^k - v_n^{k,-}\right\Vert_{L^{5}( t - t_{k,n} \leq -T \lambda_{k,n}, L^{10})} + \left\Vert v_n^{k,-}\right\Vert_{L^{5}(| t - t_{k,n}| \geq T \lambda_{k,n}, L^{10})} \Big) \\
&\hspace{1cm}\leq 4 \Vert U^j \Vert_{L^{5}L^{10}(\mathbb R^3)}^4 \epsilon,
\end{align*}
where we again used Proposition \ref{propnoncon} to control the $v_n^{k, \pm}$ terms outside of concentration times. On the other hand, in the same way,
\begin{align*}
&\left\Vert|U_{n}^{j}|^{4}U_{n}^{k}\right\Vert_{L^{1}(| t - t_{k,n}| \leq T \lambda_{k,n}, L^{2})}
\\
&\hspace{1cm}\leq
\left\Vert|U_{n}^{j}|^{4}\tilde U_n^k\right\Vert_{L^{1}L^{2}}+\left\Vert U_{n}^{j}\right\Vert_{L^{5}L^{10}}^4\left\Vert  U_n^k-\tilde U_n^k\right\Vert_{L^{5}(| t - t_{k,n}| \leq T \lambda_{k,n}, L^{10})} \\
&\hspace{1cm}\leq \left\Vert|U_{n}^{j}|^{4}\tilde U_n^k\right\Vert_{L^{1}L^{2}}+\Vert U^j \Vert_{L^{5}L^{10}(\mathbb R^3)}^4 \epsilon.
\end{align*}
It therefore only remains to show that, as $n\rightarrow \infty$
$$
\left\Vert|U_{n}^{j}|^{4}\tilde U_n^k\right\Vert_{L^{1}L^{2}} \longrightarrow 0.
$$
This follows exactly as in the end of (A.4) and (\ref{eq:A5}) follows.

\textbf{(A.6) The mixed term $|U_{n}^{j}|^{4}|U_{n}^{k}|$ for $j, k\in J_{\text{conc}}$.}
 
Finally, we show
\begin{equation} \label{eq:A6}
\left\Vert|U_{n}^{j}|^{4}|U_{n}^{k}|\right\Vert_{L^{1}L^{2}}\longrightarrow0,\ \text{for }j,k\in J_{\rm conc}.
\end{equation}
In the same spirit as for (A.4) and (A.5), let $U^k$, $U^j$, $\tilde U_n^k$, $\tilde U_n^j$,  and for $\epsilon >0$ arbitrary, $T_k, T_j>0$ be given by Proposition \ref{prop:concprof}. We let $T := \max(T_k, T_j).$
On the one hand, by H\"older inequality, for $n$ big enough
\begin{align} \label{eq:A6_1}
&\left\Vert|U_{n}^{j}|^{4}U_{n}^{k}\right\Vert_{L^{1}(| t - t_{k,n}| \geq T \lambda_{k,n}, L^{2})} \nonumber \\
&\hspace{1cm} \leq 
\left\Vert U_{n}^{j}\right\Vert_{L^{5}L^{10}}^4\left\Vert  U_n^k\right\Vert_{L^{5}(| t - t_{k,n}| \geq T \lambda_{k,n}, L^{10})} \nonumber \\
&\hspace{1cm}\leq \sup_n \Vert U_n^j \Vert_{L^{5}L^{10}(\mathbb R^3)}^4 \Big ( \left\Vert   U_n^k - v_n^{k,+}\right\Vert_{L^{5}( t - t_{k,n} \geq T \lambda_{k,n}, L^{10})} \nonumber + \left\Vert   v_n^{k,+}\right\Vert_{L^{5}(| t - t_{k,n}| \geq T \lambda_{k,n}, L^{10})} \nonumber \\
& \hspace{2cm} + \left\Vert   U_n^k - v_n^{k,-}\right\Vert_{L^{5}( t - t_{k,n} \leq -T \lambda_{k,n}, L^{10})} + \left\Vert   v_n^{k,-}\right\Vert_{L^{5}(| t - t_{k,n}| \geq T \lambda_{k,n}, L^{10})} \Big)\nonumber \\
&\hspace{1cm} \lesssim_j \epsilon,
\end{align}
where we again used Proposition \ref{propnoncon} to control the $v_n^{k,\pm}$ terms outside of concentration times. In the same way
\begin{equation} \label{eq:A6_2}
\left\Vert|U_{n}^{j}|^{4}U_{n}^{k}\right\Vert_{L^{1}(| t - t_{j,n}| \geq T \lambda_{j,n}, L^{2})} \lesssim_k \epsilon.
\end{equation}
On the other hand, denoting
$$
I_{j,k}^n := \big\{ | t - t_{j,n}| \leq T \lambda_{j,n} \big\} \cap \{ | t - t_{k,n}| \leq T \lambda_{k,n} \big\},
$$
we have by the triangle inequality together with the inequality $(a+b)^4 \lesssim a^4 + b^4$ for $a,b>0$ 
\begin{align*}
\left\Vert|U_{n}^{j}|^{4}|U_{n}^{k}|\right\Vert_{L^{1}(I_{j,k}^n, L^{2})}
\lesssim& \left\Vert|\tilde U_{n}^{j}|^{4}\tilde U_{n}^{k}\right\Vert_{L^{1}L^{2}}
+ \left\Vert|\tilde U_{n}^{j} -  U_n^j \big|^{4}|\tilde U_{n}^{k} -  U_n^k|\right\Vert_{L^{1}(I_{j,k}^n, L^{2})} \\
&+ \left\Vert|\tilde U_{n}^{j} -  U_n^j \big|^{4}  U_n^k \right\Vert_{L^{1}(I_{j,k}^n, L^{2})} 
+ \left\Vert | U_n^j|^4 |\tilde U_{n}^{k} -  U_n^k \big| \right\Vert_{L^{1}(I_{j,k}^n, L^{2})} 
\end{align*}
and hence by Hölder inequality and Proposition \ref{prop:concprof},
$$
\left\Vert|U_{n}^{j}|^{4}|U_{n}^{k}|\right\Vert_{L^{1}(I_{j,k}^n, L^{2})}
\lesssim  \left\Vert|\tilde U_{n}^{j}|^{4}\tilde U_{n}^{k}\right\Vert_{L^{1}L^{2}}
+ \epsilon.
$$
Therefore, combining the above with (\ref{eq:A6_1}) and (\ref{eq:A6_2}) it only remains to show that
$$
\left\Vert|\tilde U_{n}^{j}|^{4}\tilde U_{n}^{k}\right\Vert_{L^{1}L^{2}}
\longrightarrow 0.
$$
This follows exactly as in the end of the case (A.4). We conclude that (\ref{eq:A6}) holds.

\textbf{(A.7) Conclusion: the $L^1 L^2$ decay (\ref{eq:decay_en_J}).}

We have shown that
$$
\begin{cases}
\left\Vert|U_{n}^{j}|^{4}U_{n}^{k}\right\Vert_{L^{1}L^{2}}\longrightarrow0, \\
\forall (j,k)\in J_{\rm diff}^2 \cup J_{\rm diff} \times J_{\rm comp}\cup J_{\rm comp}^2 \cup J_{\rm diff} \times J_{\rm conc} \cup J_{\rm comp} \times J_{\rm conc}  \cup J_{\rm conc}^2.
\end{cases}
$$
The remaining, almost symmetrical cases $(j,k) \in J_{\rm comp} \times J_{\rm diff} \cup
J_{\rm conc} \times J_{\rm diff} \cup  J_{\rm conc} \times J_{\rm comp}$ are obtained in the exact same way. Hence 
$$
\left\Vert|U_{n}^{j}|^{4}U_{n}^{k}\right\Vert_{L^{1}L^{2}}\longrightarrow0, \hspace{0.5cm}
\forall (j,k),
$$
and it follows that
\begin{equation}
\forall J,\ \left\Vert\sum_{1\leq j\neq k\leq J}|U_{n}^{j}|^{4}|U_{n}^{k}|\right\Vert_{L^{1}L^{2}}\longrightarrow0.\label{eq:enj_dec_mix}
\end{equation}
Combining (\ref{eq:enj_dec_mix}) 
with (\ref{eq:L1L2dec_reduc}), we obtain the $L^{1}L^{2}$ decay of the error term $e_{n}^{J}$, that is (\ref{eq:decay_en_J}).

\textbf{(B) The data approximation (\ref{eq:v_n_data})}

Let us now show (\ref{eq:v_n_data}). Observe that
\begin{align*}
&{\breve u}_{n}(0) - u_n(0) = \sum_{1\leq j \leq J} \vec U_n^j(0) - \vec \varphi_0^{n}  \\
&\hspace{0.5cm}= \sum_{\substack{1\leq j \leq J \\ j \in J_{\rm diff}}} \Big(\vec{U}_n^j(0) - \vec \psi^{(j)} _{\Omega, \mathcal{O}^{(j)},n} \Big)
+\sum_{\substack{1\leq j \leq J \\ j \in J_{\rm comp}}} \Big( \vec U^{j}({-t_{j,n}}) - \vec \psi^{(j)} _{\Omega, \mathcal{O}^{(j)},n} \Big)
- \vec w_n^{(J)},
\end{align*}
where we used the fact that, by definition, for $j \in J_{\rm conc}$, we have $\vec{U}_n^j(0)= \vec \psi^{(j)}_{\Omega, \mathcal{O}^{(j)},n}$.
The decay of $\SO(\cdot) \vec w_n^{(J)}$ in $L^5(\mathbb R, L^{10}(\Omega))$ comes directly from the decay of the reminder in the linear profile decomposition. On the other hand, for any $j\in J_{\rm diff}$, by 
(\ref{eq:comp_U_V}) and
(\ref{eq:def_Uj_diff}), in $\dot H^1(\Omega)$, as $n$ goes to infinity
\begin{align*}
\vec{U}_n^j(0) &= \vec V_n^j(0) + o(1) = \frac{1}{\lambda_{j,n}^{1/2}}V^{j}\Big(\frac{-t_{j,n}}{\lambda_{j,n}}, \frac{\cdot-x_{j,n}}{\lambda_{j,n}}\Big) + o(1) \\
 &=  \frac{1}{\lambda_{j,n}^{1/2}}V_L^{j}\Big(\frac{-t_{j,n}}{\lambda_{j,n}}, \frac{\cdot-x_{j,n}}{\lambda_{j,n}}\Big) + o(1) \\
 &=  \psi^{(j)} _{\mathbb R^3, \mathcal{O}^{(j)},n} + o(1)
 =  \psi^{(j)} _{\Omega, \mathcal{O}^{(j)},n} + o(1),
\end{align*}
were we used Lemma \ref{lem:assfree} on the last line. The time derivative component is handled in the same way, and we obtain, in $\mathcal H(\Omega)$,
$$
\vec{V}_n^j(0) - \vec \psi^{(j)} _{\Omega, \mathcal{O}^{(j)},n} \rightarrow 0.
$$
Furthermore, from (\ref{eq:def_Uj_comp}), we have for $j\in J_{\rm comp}$, in $\mathcal H(\Omega)$
$$
\vec U^{j}({-t_{j,n}}) - \vec \psi^{(j)} _{\Omega, \mathcal{O}^{(j)},n}\rightarrow 0.
$$
The estimate (\ref{eq:v_n_data}) follows from global Strichartz estimates. This ends the proof of the Lemma.
\end{proof}
We also need some uniform bound of the approximations ${\breve u}_{n}^{J}$
\begin{lem} \label{lem:bddUnif} 
    There exists a uniform $M>0$ so that for every $J\in \N$,
\begin{equation*}
\limsup_{n\rightarrow\infty}\nor{{\breve u}_{n}^{J}}{L^{5}L^{10}} \leq M.
\end{equation*}
\end{lem}
\begin{proof}
By (\ref{eq:orthNL}),
$$
\lim_{J \to \infty} \lim_{n\to \infty}\Vert {\breve u}_{n}^{J} \Vert^5_{L^5L^{10}} = \lim_{J \to \infty} \lim_{n\to \infty}\sum_{j=1}^J \Vert U_n^j \Vert^5_{L^5L^{10}},
$$
hence in particular, 
$$
\lim_{n\to \infty}\Vert {\breve u}_{n}^{J} \Vert^5_{L^5L^{10}} \leq C + \lim_{n\to \infty}\sum_{j=1}^J \Vert U_n^j \Vert^5_{L^5L^{10}}.
$$
By the small data theory \eqref{e:strigchglob} in Proposition \ref{prop:perturb_scat}, there exists $\epsilon_0 > 0$ so that
$$
E(\vec \varphi) \leq \epsilon_0 \implies 
\Vert  \NLSO(\cdot)\vec \varphi \Vert_{L^5L^{10}} \leq C_{\epsilon_0} \Vert \vec \varphi \Vert_{\mathcal H}.
$$
Let $J_0 \geq 0$ big enough so that
$$
j \geq J_0 \implies \Vert \vec \psi^{(j)}_{\Omega, \mathcal O^{(j)}, n}\Vert_{\mathcal H} \leq \epsilon_0.
$$
Using the fact that, as shown in the proof of Lemma \ref{lem:nl_prof_pert}, part (B),
$$
\forall j, \quad \Vert U_n^j(\tau_{j,n}) \Vert_{\mathcal H}  = \Vert \vec \psi^{(j)}_{\Omega, \mathcal O^{(j)}, n}\Vert_{\mathcal H} + o_n(1), \quad \tau_{j,n} := 
\begin{cases}
0 &\text{ if } j \in J_{\rm diff} \cup J_{\rm conc}, \\
-t_{j,n} &\text{ if } j \in J_{\rm comp},
\end{cases}
$$
we can now write, using \eqref{eq:tout_est_diffusif} for $j< J_0$ and the previous estimate otherwise,
\begin{align*}
\lim_{n\to \infty}\Vert {\breve u}_{n}^{J} \Vert^5_{L^5L^{10}} 
&\leq C_{J_0} + C_{\epsilon_0}\lim_{n\to \infty}\sum_{j=J_0}^J \Vert \vec \psi^{(j)}_{\Omega, \mathcal O^{(j)}, n} \Vert^5_{\mathcal H}  \\
&  \lesssim 1 + \lim_{n\to \infty}\sum_{j=J_0}^J \Vert \vec \psi^{(j)}_{\Omega, \mathcal O^{(j)}, n} \Vert^2_{\mathcal H},
\end{align*}
where we used the injection $\ell^{5} \hookrightarrow \ell^2$, and the result follows thanks to the orthogonality of the  linear profiles (\ref{eq:lpd_pyt}).
\end{proof}

The proof of the nonlinear profile decomposition follows:
\begin{proof}[Proof of Proposition \ref{prop:non_lin_prof}]
By Lemma \ref{lem:nl_prof_pert} and Lemma \ref{lem:bddUnif}, 
the perturbative result of Proposition \ref{lem:perturb} gives
$$
u_{n}=\breve u_{n}^{J}+{R}_{n}^{J},
$$
with
$$
\lim_{J\rightarrow\infty}\limsup_{n\rightarrow\infty}\Vert{R}_{n}^{J}\Vert_{L^{5}L^{10}}=0.
$$
Finally, (\ref{eq:comp_U_V}) enables us to replace all the $U_{n}^{j}$ by
$V_{n}^{j}$ for $j \in J_{\text{diff}}$ in the definition of $\breve u_{n}^{J}$.
\end{proof}

We are now in position to end the proof of the Theorem. Indeed, by
the nonlinear profiles decomposition Proposition \ref{prop:non_lin_prof}, together with (\ref{eq:tout_est_diffusif}), $u_{n}$ is in $L^{5}L^{10}$ 
with a uniform bound in $n$ for $n$ large enough, and the definition of the minimizing sequence (\ref{eq:suite_u0n}) is
contradicted. Therefore the assumption $J>1$ (\ref{eq:more_than_one_profile})
cannot hold, that is, $J=1$: there is only one non-trivial profile
in the decomposition (\ref{eq:suite_u0n}):
\begin{equation}
\label{dev_u0n}
\vec{\varphi}_{0}^{n}=\vec \psi _{\Omega, \mathcal{O},n}+\vec{w}_{n},\hspace{1em}\Vert \SO(\cdot)\vec{w}_{n}\Vert_{L^{5}L^{10}}\longrightarrow0  
\end{equation} 
Let us show that it is the fully-compact one: $t_{1,n}=0$,
$\lambda_{1,n}=1$, $x_{1,n} = 0$;
that is $1\in J_{\rm comp}$ with $t_{1,n}=0$.

As noticed before,
as the scattering in the free space $\mathbb{R}^{3}$ is well known,
we have $V^j\in L^{5}L^{10}$ for any $j\in J_{\text{diff}}$. Therefore, if $1\in J_{\text{diff}}$,
the same proof as before yields
the decomposition: 
\begin{equation}
u_{n}(t) =\frac{1}{\lambda_{1,n}^{1/2}}V^{1}\Big(\frac{t-t_{1,n}}{\lambda_{1,n}},\frac{\cdot}{\lambda_{1,n}}\Big)+R_{n}(t)\label{eq:dec_uj-1}
\end{equation}
with 
\begin{equation}
\label{decRn}
\limsup_{n\rightarrow\infty}\Vert R_{n}\Vert_{L^{5}L^{10}}=0,
\end{equation}
proving that $u_n\in L^5L^{10}$, a contradiction. 
Similarly, if $1 \in J_{\rm conc}$, the same proof as before gives thanks to Proposition \ref{prop:concprof}
$$
u_n (t) = U_n(t) +R_{n}(t)
$$
with
$$
\sup_n \Vert U_n \Vert_{L^5 L^{10}}<\infty, \hspace{0.3cm} \limsup_{n\rightarrow\infty}\Vert R_{n}\Vert_{L^{5}L^{10}}=0,
$$
proving again that $u_n\in L^5L^{10}$, a contradiction. Thus $1\in J_{\rm comp}$.

It remains to eliminate the case $t_{1,n}\rightarrow\pm\infty$.
Let us for example assume, by contradiction, that $t_{1,n}\rightarrow+\infty$.
This implies
$$\lim_{n\to\infty}\left\|\SO(\cdot-t_{1,n})\vec{\psi}^1\right\|_{L^5\left((-\infty,0)L^{10}\right)}=0,$$
and we obtain, by the small data well-posedness theory, that for large $n$, $u_n\in L^{5}((-\infty,0), L^{10})$ with 
$$\lim_{n\to\infty} \|u_n\|_{L^5((-\infty,0),L^{10})}=0,$$
contradicting (\ref{eq:symL5L10infty}). 
The case $t_{1,n}\rightarrow-\infty$
is eliminated in the same way.

Therefore, $\vec{\varphi}_{0}^{n}$ writes:
\begin{equation} \label{eq:dec_crit_sol}
\vec{\varphi}_{0}^{n}=\vec{\psi}^{1}+\vec{w}_{n},\hspace{1em}\Vert \SO (\cdot)\vec{w}_{n}\Vert_{L^{5}L^{10}}\rightarrow0.
\end{equation}
with $\vec{\psi}^{1} \in \mathcal H(\Omega)$. By the Pythagorean expansion and its $L^{6}$ version, $\mathscr{E}(\vec{\psi}^{1})\leq E_{c}$,
and therefore
$$
\mathscr{E}(\vec{\psi}^{1})=E_{c}
$$
otherwise, by the definition of $E_{c}$,
$u_{n}$ scatters. This implies, by the Pythagorean expansion again,
$$
\Vert\vec{w}_{n}\Vert_{\mathcal{H}(\Omega)}\rightarrow0,
$$
that is $\vec{u}_{0}^{n} \rightarrow \vec{\psi}^{1}$ strongly in $\mathcal H (\Omega)$.
We take $\vec{\varphi}_{c}$ to be this profile:
$$
\vec{\varphi}_{c}:=\vec{\psi}^{1}.
$$
By the conservation of energy, we have $\mathscr{E}(\NLSO(t)\vec{\varphi}_{c})=E_{c}$
for any $t$, and the same argument applied to
$\NLSO(t_{n})\vec{\varphi}_{c}$
for any sequence $(t_{n})_{n\geq1}\in\mathbb{R}^{\mathbb{N}}$ shows
that the nonlinear flow $\big\{ t\in\mathbb{R},\;\NLSO(t)\vec{\varphi}_{c}\big\} $
has a compact closure in $\mathcal{H}(\Omega)$. Indeed this sequence satisfies the same assumptions 
as $\vec{\varphi}^0_n$, namely \eqref{eq:suite_u0n} at the beginning of the proof, and will therefore have a convergent
subsequence in $\mathcal{H}(\Omega)$ as well. Finally, observe that $\mathscr{E}(\vec{\varphi}_{c})=E_{c}>0$
insures in particular that $\vec{\varphi}_{c} \neq \vec{0}$; and
we have $\NLSO(\cdot)\vec \varphi_c \notin L^5(\mathbb R, L^{10}(\Omega))$ otherwise  the non-scattering property of the minimizing sequence (\ref{eq:suite_u0n}) is
contradicted using (\ref{eq:dec_crit_sol}) together with the perturbative result of Proposition \ref{prop:non_lin_prof} as before.

\end{proof}

\section{Rigidity outside two strictly convex obstacles} \label{sec:rigi}
The purpose of this section is to show
the following rigidity property in the exterior of two strictly convex obstacles 
verifying Assumption \ref{ass:normal}. 

\begin{thm}[Rigidity]
\label{thm:rig}
Let $\Theta_{1},\Theta_{2}$ be two smooth, strictly convex subsets of $\mathbb{R}^{3}$ with compact boundary verifying Assumption \ref{ass:normal},
and $\Omega:=\mathbb{R}^{3}\backslash(\Theta_{1}\cup\Theta_{2})$.
There is no non-trivial solution $u$ of (\ref{eq:NLWomega})  such that $u$ does not scatter and the flow $\left\{ (u(t),\partial_{t}u(t)),\ t\geq0\right\} $
is relatively compact in $\dot{H}^{1}\times L^{2}$.
\end{thm}

Observe that, put together with Theorem \ref{th:const_intro} shown in the previous section (recalling that Assumptions \ref{ass:nonreco} and \ref{ass:weaktrap} are verified in \S \ref{sec:verifA} and Assumption \ref{ass:strichartz} is the main result of  \cite{DStrW}), this shows Theorem \ref{th:main}.

Our main tool will be the following momentum identity, which was first
introduced by Morawetz in a similar form to show some decay properties
of the linear wave equation. The normal at the boundary, denoted $\vec n$ (or $n$ where there is no ambiguity) is outgoing for $\Omega$, that is pointing inside the obstacle.
\begin{lem}
\label{lem:moment}Let $u$ be a solution of (\ref{eq:NLWomega}) in a
domain $\Omega$ of $\mathbb{R}^{3}$ and
$\chi\in C^{\text{\ensuremath{\infty}}}(\Omega,\mathbb{R})$ . Then
we have
\begin{multline}
\Bigg[\int_{\Omega}-\partial_{t}u\nabla u\nabla\chi-\frac{1}{2}\Delta\chi u\partial_{t}u \Bigg]_0^t=\int_0^t\int_{\Omega}(D^{2}\chi\nabla u,\nabla u)-\frac{1}{4}\int_0^t\int_{\Omega}u^{2}\Delta^{2}\chi\\
+\frac{1}{3}\int_0^t\int_{\Omega}|u|^{6}\Delta\chi-\frac{1}{2}\int_0^t\int_{\partial\Omega}|\partial_{n}u|^{2}\partial_{n}\chi.\label{eq:mor}
\end{multline}
\end{lem}
\begin{proof}
The identity can be shown by standard integrations by parts justified
by an approximation argument. 
\end{proof}

\subsection{A scattering criterion}

The scattering in $\mathbb{R}^{3}$ was shown by \cite{BahouriShatah98}.
Their proof still holds in the case of a domain with boundaries if
we are able to control the boundary term arising in their computations, as shown in the following lemma.
\begin{lem}
\label{lem:bordBS}Let $u$ be a solution of (\ref{eq:NLWomega}) in a
domain $\Omega$ of $\mathbb{R}^{3}$ with compact boundary and verifying Assumption \ref{ass:strichartz}. If
\begin{equation} \label{eq:dec_bdd_proof}
\frac{1}{T}\int_{0}^{T}\int_{\partial\Omega}|\partial_{n}u|^{2}d\sigma dt\ \to 0,
\end{equation}
as $T$ goes to infinity, then $u$ scatters in $\mathcal{H}$.
\end{lem}
\begin{proof} \label{eq:decayL6__}
The scattering classically follows from the decay estimate
\begin{equation} \label{eq:decayL6__}
    \int |u(x,t)|^6 \,dx \to 0 \text{ as } t \to \infty,
\end{equation}
together with Assumption \ref{ass:strichartz}. Indeed, if (\ref{eq:decayL6__}) holds, $\epsilon > 0$ been given, there exists $T>0$ large enough so that $\Vert u(t) \Vert_{L^6} \leq \epsilon$ for any $t\geq T$. Then by Assumption \ref{ass:strichartz} (and Lemma \ref{lem:L1L2_}), for any $S> T$
$$
\Vert u \Vert_{L^5((T,S), L^{10})} + \Vert u \Vert_{L^r((T,S), L^{s})} \leq C(E + \Vert u^5 \Vert_{L^1((T,S), L^2)}) = C(E + \Vert u \Vert^5_{L^5((T,S), L^{10})}),
$$
where
$$
\Vert u \Vert_{L^5((T,S), L^{10})} \leq
\Vert u \Vert_{L^\infty((T,S), L^{6})} ^{\theta} \Vert u  \Vert_{L^{r}((T,S), L^{s})} ^{1-\theta} \leq \epsilon^\theta \Vert u  \Vert_{L^{r}((T,S), L^{s})} ^{1-\theta},
$$
with $\theta= \frac{3(s-10)}{5(s-6)}>0$  since $s>10$, thus 
$$
\Vert u \Vert_{L^5((T,S), L^{10})} + \Vert u \Vert_{L^r((T,S), L^{s})} \leq C(E + \epsilon^{5\theta} \Vert u  \Vert_{L^{r}((T,S), L^{s})} ^{5-5\theta}),
$$
from which $u \in L^5(\mathbb R_+, L^{10}(\Omega))$ thanks to a continuity argument, and the scattering follows by Proposition \ref{prop:perturb_scat}.

The decay estimate (\ref{eq:decayL6__}) when $\Omega = \mathbb R^3$ is due to Bahouri and Shatah \cite{BahouriShatah98}. When $\Omega^c$ is star-shaped, this is due to \cite{MR2566711} by remarking that the boundary term arising in Bahouri and Shatah computations has the right sign. Without geometrical assumption on $\Omega^c$ , this boundary term decays to zero as soon as (\ref{eq:dec_bdd_proof}) holds.
More precisely, in the exact same way as \cite{BahouriShatah98} and \cite[Proof of Lemma 4.2]{MR2566711}, using  a flux identity and the time-translation invariance of the equation it suffices to show that
$$
\int_{|x| \leq T} |u(x,T)|^6 \, dx \to 0
$$
as $T \to \infty$.
Following \cite[Proof of Lemma 4.2]{MR2566711}, integrating Bahouri-Shatah space-time divergence identity over the truncated light-cone $\{ |x|\leq t, \; T_1 \leq t \leq T_2\}$ gives
$$
0 \geq \rm{I} + \rm{II} + \rm{III} + \mathcal B,
$$
where $\rm{I}, \rm{II}, \rm{III}$ are the exact same terms as in \cite[Proof of Lemma 4.2]{MR2566711} and $\mathcal B$ is the boundary term, which is not signed anymore and is given by
$$
\mathcal B = - \frac 12 \int_{T_1}^{T_2} \int_{\partial \Omega} \vec n \cdot x \ (\partial_n u)^2 \, d\sigma.
$$
The rest of the proof consists in taking $T_1 := \epsilon T$, $T_2 := T$ with $\epsilon>0$ small enough to control the integrals on the bottom of the truncated cone $\{(x,\epsilon T), \; |x| \leq \epsilon T\}$ as well as the fluxes through the mantel, and getting the control of the $L^6$ norm on $\{(x, T), \; |x| \leq  T\}$ from the top of the truncated cone. Following \cite[Proof of Lemma 4.2]{MR2566711} verbatim but keeping the boundary term $\mathcal B$ this gives, after dividing by $T$, for $T \gg 1$ big enough 
$$
\int_{|x| \leq T} |u(x,T)|^6 \, dx \leq \epsilon C(E) + \frac 1T |\mathcal B| \leq \epsilon C(E) + \frac {1}{2T} \sup_{x\in \partial \Omega} |x|\int_0 ^T \int_{\partial \Omega} |\partial_n u|^2 \, d\sigma,
$$
and the result follows from (\ref{eq:dec_bdd_proof}).
\end{proof}
Note that the trace of the normal derivative is not an easy object
to deal with, because this trace is a priori not defined in $L^{2}(\partial\Omega)$
for elements of $\dot{H}^{1}(\Omega)$. Moreover, even if we can define
it for almost every $u(t)$ when $u$ is a solution of (NLW) (see for instance in \cite{LLT:86} the classical hidden regularity for the linear equation) because
of the particular structure of the equation, the application 
\[
u\in\dot{H}^{1}\cap\left\{ \text{value in time \ensuremath{t} of solutions of NLW}\right\} \mapsto\partial_{n}u\in L^{2}(\partial\Omega)
\]
 is not continuous. 

For this reason, we prefer to deal with the following criterion, which
involves only the local energy of the equation, and which we deduce
from the previous one using the momentum identity (\ref{eq:mor}):
\begin{lem}
\label{lem:contrenerg}Let $u$ be a solution to (\ref{eq:NLWomega}) in a domain $\Omega$ of $\mathbb{R}^{3}$ with compact boundary. There exists $A>0$, $B(0,A)\supset\partial\Omega$,
such that, if
\begin{equation}
\frac{1}{T}\int_{0}^{T}\int_{\Omega\cap B(0,A)}|\nabla u(x,t)|^{2}+|u(x,t)|^{6}\ dxdt\longrightarrow0,\label{eq:contrenerg}
\end{equation}
as $T$ goes to infinity, then $u$ scatter in  $\mathcal{H}$.
\end{lem}
\begin{proof}
Let $\chi\in C_{0}^{\infty}(\mathbb{R}^{3},\mathbb{R})$ be such that
$\nabla\chi=-n$ on $\partial\Omega$, supported in $B(0,A)$. Suppose
that
\[
\frac{1}{T}\int_{0}^{T}\int_{\Omega\cap B(0,A)}|\nabla u(x,t)|^{2}+|u(x,t)|^{6}\ dxdt\longrightarrow0
\]
as $T$ goes to infinity. We use  Lemma \lemref{moment} with the weight
$\chi$ to get:
\begin{equation*}
\partial_{t}\left(\int_{\Omega}-\partial_{t}u\nabla u\nabla\chi-\frac{1}{2}\Delta\chi u\partial_{t}u\right)=\int_{\Omega}(D^{2}\chi\nabla u,\nabla u)-\frac{1}{4}\int_{\Omega}u^{2}\Delta^{2}\chi
+\int_{\Omega}|u|^{6}\Delta\chi+\frac{1}{2}\int_{\partial\Omega}|\partial_{n}u|^{2}d\sigma.
\end{equation*}
Integrating in time we get
\begin{equation*}
\int_{0}^{T}\int_{\partial\Omega}|\partial_{n}u|^{2}d\sigma dt\lesssim\int_{\Omega\cap B(0,A)}|\partial_{t}u\nabla u|+|u\partial_{t}u|+\int_{0}^{T}\int_{\Omega\cap B(0,A)}|u|^{6}+|u|^{2}+|\nabla u|^{2},
\end{equation*}
and using Minkowsky inequality,
\begin{align*}
\int_{0}^{T}\int_{\partial\Omega}|\partial_{n}u|^{2}d\sigma dt&\lesssim\left(\int_{\Omega}|\partial_{t}u|^{2}\right)^{\frac{1}{2}}\left(\int_{\Omega}|\nabla u|^{2}\right)^{\frac{1}{2}}+A^{\frac{1}{3}}\left(\int_{\Omega}|\partial_{t}u|^{2}\right)^{\frac{1}{2}}\left(\int_{\Omega}|u|^{6}\right)^{\frac{1}{6}}\\
&\quad+\int_{0}^{T}\int_{\Omega\cap B(0,A)}\left(|u|^{6}+|\nabla u|^{2}\right)+A^{\frac{2}{3}}\int_{0}^{T}\left(\int_{\Omega\cap B(0,A)}|u|^{6}\right)^{\frac{1}{3}}\\
&\lesssim_{A}C(E)+\int_{0}^{T}\int_{\Omega\cap B(0,A)}(|u|^{6}+|\nabla u|^{2})+T^{\frac{2}{3}}\left(\int_{0}^{T}\int_{\Omega\cap B(0,A)}|u|^{6}\right)^{\frac{1}{3}}.
\end{align*}
 Thus
\begin{equation*}
\frac{1}{T}\int_{0}^{T}\int_{\partial\Omega}|\partial_{n}u|^{2}d\sigma dt\lesssim_{A}\frac{C(E)}{T}+\frac{1}{T}\int_{0}^{T}\int_{\Omega\cap B(0,A)}(|u|^{6}+|\nabla u|^{2})
+\left(\frac{1}{T}\int_{0}^{T}\int_{\Omega\cap B(0,A)}|u|^{6}\right)^{\frac{1}{3}}\longrightarrow0
\end{equation*}
as $T\rightarrow\infty$ and by  Lemma \lemref{bordBS} we conclude that
$u$ scatters in $\dot{H}^{1}$.
\end{proof}

\subsection{Proof of Theorem \thmref{rig}}

In order to prove Theorem \thmref{rig}, we will show that the previous scattering
criterion is verified using a carefully chosen weight.

In the following Lemma, we recall that $n$
is the normal oriented toward the interior of $\Theta_1\cup \Theta_2$, and we choose coordinates such that the trapped ray
$\mathcal{R}$ is a segment of the line $\left\{ x_{2}=x_{3}=0\right\}$.
 Remark that a version of this Lemma, adapted to potentials instead of obstacles, originates from the first author's work \cite{D2pot}.

\begin{lem}
\label{lem:poids}
Let $\Theta_{1},\Theta_{2}$ be two smooth, strictly convex subsets of $\mathbb{R}^{3}$ with compact boundary verifying Assumption \ref{ass:normal}.
Let $c_1>0$ and $c:=(c_{1},0,0)$. Denote
$$
\chi(x) := |x-c| + |x+c|.
$$
Then, for any $c_1>0$ fixed big enough,
\[
\nabla \chi(x)\cdot(-n)(x) \geq 0, \quad \forall x\in\partial(\Theta_{1}\cup\Theta_{2}).
\]
\end{lem}
\begin{proof}
We first do the proof in dimension $2$, as it makes the main idea clearer, and we then give the full three-dimensional argument. 

Thus, assume first that $\Theta_1, \Theta_2 \subset \mathbb R^2$ with 
$\mathcal{R}$ a segment of the line $\left\{ x_{2}=0\right\} = \mathbb Re_1$. 
By strict convexity and Assumption \ref{ass:normal}, for any $x\in\big(\partial \Theta_1 \cup \partial \Theta_2\big) \backslash \mathbb R e_1$, $n(x)$ is not colinear to $\mathbb R e_1$. Indeed, let $w$ be so that (for example) $-n(w) = e_1$. Then, the tangent to $\Theta_1$ in $w$ is carried by $e_2$, so by convexity, $\Theta_1 \subset \{ x_1 \leq w \cdot e_1\}$. It follows that 
the functions $x \in \partial \Theta_1 \mapsto x\cdot e_1 \in \mathbb R$ has a maximum at $w$.
By strict convexity, such maximal points are unique, so $w \in \mathbb Re_1$. It follows that, for any $x\in\big(\partial \Theta_1 \cup \partial \Theta_2\big) \backslash \mathbb R e_1$ there is a (unique) $t_0(x){>} 0$, depending continuously on $x$, so that
$$
x+t_0(x)n(x) \in \mathbb Re_1.
$$
We extend $t_0(x)$ continuously to the whole $\partial \Theta_1 \cup \partial \Theta_2$, still verifying the above property, by setting, for $x\in\big(\partial \Theta_1 \cup \partial \Theta_2\big) \cap \mathbb R e_1$, $t_0(x) := \kappa(x) > 0$, where $\kappa(x)$ is the radius of curvature of $\partial \Theta_1 \cup \partial \Theta_2$ at $x$.
\begin{figure}
\includegraphics[scale=1]{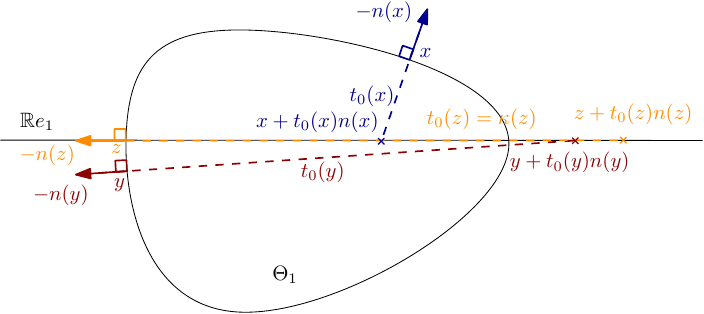}\caption{Illustration of the proof of Lemma \ref{lem:poids} in the two dimensional case: $t_0$ for a few points of $\partial\Theta_1$.} \
\end{figure}
Denote $(-c,c)$ the open segment
$$
(-c,c) := \big\{ (x_1, 0), \quad x_1 \in (-c_1, c_1) \big\}.
$$
As $x \mapsto t_0(x)$ is continuous, by compactness of $\partial \Theta_1 \cup \partial \Theta_2$ we can chose $c_1>0$ big enough so that, for any $x\in \partial \Theta_1 \cup \partial \Theta_2$,
$$
x+t_0(x)n(x) \in (-c, c).
$$
{We can also impose $x\notin \{-c,c\}$. }Now, observe that {$\nabla \chi (x)=\frac{x-c}{|x-c|}+\frac{x-c}{|x-c|}$, so that}
$$
\forall x \in (-c,c), \quad \nabla \chi (x) = 0.
$$
In addition, as $\chi$ is convex and $C^1$ outside $\{ -c, c \}$, we have for any $p, q$ so that $[p,q]\cap \{-c,c\}=\emptyset$
$$
(\nabla \chi(p) - \nabla \chi(q)) \cdot (p - q) \geq 0.
$$
If necessary, take $c_1$ larger so that $\{ -c, c \} \cap \partial (\Theta_1\cup\Theta_2) = \emptyset$.
Now, take $x\in\partial (\Theta_1 \cup \partial \Theta_2) = \emptyset$ and apply the above with $p=x$, $q=x+t_0(x)n(x)$. We get, as $\nabla \chi(q) = 0$,
$$
\nabla \chi(x) \cdot t_0(x)(-n)(x) \geq 0.
$$
This ends the proof in dimension two since $t_0>0$ .

We now go back to the three dimensional setting we are interested in. 
We will reduce to the two dimensional setting above.
Let $x\in \partial \Theta_1 \cup \partial \Theta_2$.
Denote by $\Pi_x$ be the plane generated by $e_1$ and $x$:
$$
\Pi_x := \operatorname{span}(x, e_1),
$$
and, for $j=1,2$
$$
\mathcal C_j := \Theta_j \cap \Pi_x.
$$
Observe that, as $x, -c, c \in \Pi_x$,
$$
\nabla \chi(x) \in \Pi_x.
$$
It follows that, denoting $\pi_x$ the orthogonal projection onto $\Pi_x$,
$$
(-n(x)) \cdot \nabla \chi(x) = (-\pi_x n(x)) \cdot \nabla \chi(x).
$$
On the other hand, by the two dimensional argument detailed above,  there is $C(x)>0$, depending continuously on $x$, so that for any $c_1 \geq C(x)$
$$
\nabla \chi(x) \cdot (-\nu(x)) \geq 0,
$$
where $\nu(x)$ is the inward-pointing normal to $\mathcal C_1 \cup \mathcal C_2$ in $\Pi_x$. As $\nu(x)$ is positively colinear to $\pi_x n(x)$, we conclude that
$$
\nabla \chi(x) \cdot (-n(x)) \geq 0,
$$
for any $c_1 \geq C >0$, with a uniform $C>0$ by compactness.
The proof is completed.
\end{proof}
\begin{remark}
 It is insightful to consider the case of the exterior of two balls: in this case, an explicit computation gives the result with the points $\pm c$ the center of each ball. This was remarked in \cite[Section 5.2]{DStrW} (with no scattering result at the time). Observe also in the above proof that $c_1$ needs to be fixed larger as the radii of curvature of the obstacles are bigger, hence when the trapped trajectories becomes less instable.
\end{remark}
We are now in position to prove the rigidity Theorem:
\begin{proof}[Proof of  Theorem \thmref{rig}.]
Let $u$ be a solution of (\ref{eq:NLWomega}) in $\Omega=\mathbb{R}^{3}\backslash\left(\Theta_{1}\cup\Theta_{2}\right)$
with a relatively compact flow $\left\{ u(t),\ t\geq0\right\} $ in
$\dot{H}^{1}$. We will show that $u$ scatters in $\dot{H}^{1}\times L^2$. Let $A$ be  given by Lemma~\ref{lem:contrenerg}. We set $c=(c_{1},0,0)$ with $c_{1}>0$, fixed big enough according to Lemma \ref{lem:poids}  and so that $\pm c\notin B(0,A)$, and choose the weight
\[
\chi(x):=|x+c|+|x-c|.
\]
so that 
\begin{equation}
    \label{e:forDeltachi}
    \Delta\chi=\frac{2}{|x-c|}+\frac{2}{|x+c|}.
\end{equation}
Observe that 
$$
\Delta^{2}\chi=-8\pi(\delta_{-c}+\delta_{c}) \leq 0,
$$
hence, Lemma \lemref{moment} together with an approximation argument  gives

\begin{equation}
\partial_{t}\left(\int_{\Omega}-\partial_{t}u\nabla u\nabla\chi-\frac{1}{2}\Delta\chi u\partial_{t}u\right)\geq \int_{\Omega}(D^{2}\chi\nabla u,\nabla u) 
+\frac{1}{3}\int_{\Omega}|u|^{6}\Delta\chi-\frac{1}{2}\int_{\partial\Omega}|\partial_{n}u|^{2}\nabla\chi\cdot n\ d\sigma.\label{eq:momsomme-1}
\end{equation}

On the other hand, according to Lemma \lemref{poids}, we have
\begin{equation}
\nabla\chi\cdot(-n) \geq 0 
\end{equation}
 Thus, combining (\ref{eq:momsomme-1}) and the above, we obtain the inequality:
\begin{equation}
\partial_{t}\left(-\int_{0}^{T}\partial_{t}u\nabla u\cdot\nabla\chi+\frac{1}{2}\Delta\chi u\partial_{t}u\right)\geq\frac{1}{3}\int_{\Omega}|u|^{6}\Delta\chi+\int_{\Omega}(D^{2}\chi\nabla u,\nabla u)
\label{eq:morcontr}
\end{equation}
Integrating this estimate and controling the left-hand side using
the Hardy inequality 
\[
\int_{\Omega}|f\Delta \chi|^{2}\lesssim \sum_{\pm}\int_{\Omega}\frac{|f|^{2}}{|x\pm c|^{2}}\lesssim\int_{\Omega}|\nabla f|^{2}\ \text{for }f\in\dot{H}_{0}^{1}(\Omega)
\]
and that $\nabla \chi$ is bounded, we get
\begin{equation}
\int_{0}^{T}\int_{\Omega}|u|^{6}\Delta\chi+(D^{2}\chi\nabla u,\nabla u)\ dxdt\lesssim E. 
\label{eq:eq1}
\end{equation}

From the one hand, from \eqref{e:forDeltachi}, 
\[
\Delta\chi(x)\gtrsim 1, \textnormal{ for }x\in B(0,A),
\]
 thus 
\[
\int_{\Omega\cap B(0,A)}|u|^{6}\lesssim \int_{\Omega\cap B(0,A)}|u|^{6}\Delta\chi\lesssim \int_{\Omega}|u|^{6}\Delta\chi,
\]
and therefore, by (\ref{eq:eq1}) and the non-negativity of $D^2\chi$ (since $\chi$ is convex)
\begin{equation}
\frac{1}{T}\int_{0}^{T}\int_{\Omega\cap B(0,A)}|u|^{6}\ dxdt\lesssim \frac{E}{T}. 
\label{eq:decayl6}
\end{equation}

Now, we would like to estimate the localised cinetic energy using
(\ref{eq:eq1}) again. We have
\[
D^{2}\chi(x)=\frac{1}{|x+c|}\left(\text{Id}-\frac{(x+c)(x+c)^{t}}{|x+c|^{2}}\right)+\frac{1}{|x-c|}\left(\text{Id}-\frac{(x-c)(x-c)^{t}}{|x-c|^{2}}\right).
\]
The operators corresponding to the matrices 
\[
\text{Id}-\frac{(x+c)(x+c)^{t}}{|x+c|^{2}},\text{ resp. }\text{Id}-\frac{(x-c)(x-c)^{t}}{|x-c|^{2}},
\]
are the orthogonal projections on the plane normal to $\frac{x+c}{|x+c|}$,
resp. to $\frac{x-c}{|x-c|}$. Thus, 
\begin{equation}
(D^{2}\chi\cdot\xi,\xi)=\left(\frac{1}{|x+c|}+\frac{1}{|x-c|}\right)|\xi|^{2}-\frac{1}{|x+c|}\left(\xi\cdot\frac{x+c}{|x+c|}\right)^{2}-\frac{1}{|x-c|}\left(\xi\cdot\frac{x-c}{|x-c|}\right)^{2}.\label{eq:d2xi1}
\end{equation}
We choose orthonormal coordinates (depending of $x$ and $c$) such that 
\[
\frac{x+c}{|x+c|}=(1,0,0),\ \frac{x-c}{|x-c|}=(a,b,0),
\]
for $a=a(x)=\frac{x+c}{|x+c|}\cdot \frac{x-c}{|x-c|}$ and $b=b(x)=\sqrt{1-a^2}$. Notice that $(a,b)=(\cos \theta,\sin\theta)$ where $\theta$ is the angle between $x-c$ and $x+c$. Then we have, if $\xi=\begin{pmatrix}\hat{\xi}_{1} & \hat{\xi}_{2} & \hat{\xi}_{3}\end{pmatrix}$
in this set of coordinates
\[
\frac{1}{|x+c|}\left(\xi\cdot\frac{x+c}{|x+c|}\right)^{2}+\frac{1}{|x-c|}\left(\xi\cdot\frac{x-c}{|x-c|}\right)^{2}=\begin{pmatrix}\hat{\xi}_{1} & \hat{\xi}_{2}\end{pmatrix}\begin{pmatrix}\frac{1}{|x+c|}+\frac{a^{2}}{|x-c|} & \frac{ab}{|x-c|}\\
\frac{ab}{|x-c|} & \frac{b^2}{|x-c|}
\end{pmatrix}\begin{pmatrix}\hat{\xi}_{1}\\
\hat{\xi}_{2}
\end{pmatrix}.
\]
The largest eigenvalue of this positive quadratic form in $\begin{pmatrix}\hat{\xi}_{1} & \hat{\xi}_{2}\end{pmatrix}$
writes
\[
\lambda_{2}=\frac{1}{2}\left(\frac{1}{|x-c|}+\frac{1}{|x+c|}+\sqrt{\left(\frac{1}{|x-c|}+\frac{1}{|x+c|}\right)^{2}-4\frac{b^{2}}{|x+c||x-c|}}\right),
\]
Therefore, since $\pm c\notin B(0,A)$, there exists $C=C(A,b)$ so that $\frac{1}{|x+c||x-c|}\geq C \left(\frac{1}{|x-c|}+\frac{1}{|x+c|}\right)^{2}$ for $x\in\Omega\cap B(0,A)$ and we have
\[
\lambda_{2}\leq \frac{1}{2} \left(\frac{1}{|x-c|}+\frac{1}{|x+c|}\right) \left(1+ \sqrt{1-4 C b^{2}}\right) ,
\]
and there exists another $C>0$ such that, we have, for $x\in\Omega\cap B(0,A)$
and $\alpha>0$ small enough
\begin{equation}
b^2\geq\alpha\implies\lambda_{2}\leq\left(\frac{1}{|x-c|}+\frac{1}{|x+c|}\right)(1-C\alpha).\label{eq:d2xi2}
\end{equation}

On the other hand
\[
\frac{1}{|x+c|}\left(\xi\cdot\frac{x}{|x+c|}\right)^{2}+\frac{1}{|x-c|}\left(\xi\cdot\frac{x-c}{|x-c|}\right)^{2}\leq\lambda_{2}|(\hat{\xi}_{1},\hat{\xi}_{2})|^{2}\leq\lambda_{2}|\xi|^{2},
\]
thus we get, combining this last inequality with (\ref{eq:d2xi1})
and (\ref{eq:d2xi2}), for $x\in\Omega\cap B(0,A)$
\begin{equation}
b^2\geq\alpha\implies(D^{2}\chi\cdot\xi,\xi)\gtrsim \alpha |\xi|^{2},\label{eq:contrang}
\end{equation}
and let us denote, for $\alpha\leq\alpha_{0}$ 
\[
S(\alpha)=\Omega\cap B(0,A)\cap\left\{ b^{2}(x)\geq\alpha\right\}.
\]
We have, on $S(\alpha)$, because of (\ref{eq:contrang})
\[
(D^{2}\chi\cdot\xi,\xi)\gtrsim{\alpha}|\xi|^{2}.
\]
Thus we get
\[
\int_{\Omega}(D^{2}\chi\nabla u,\nabla u)\geq\int_{S(\alpha)}(D^{2}\chi\nabla u,\nabla u)\gtrsim{\alpha}\int_{S(\alpha)}|\nabla u|^{2},
\]
and by (\ref{eq:eq1}) we obtain
\begin{equation}
\frac{1}{T}\int_{0}^{T}\int_{S(\alpha)}|\nabla u|^{2}\ dxdt\lesssim\frac{E}{\alpha T}.
\label{eq:decaygrad}
\end{equation}

\begin{figure}
\includegraphics[scale=0.45]{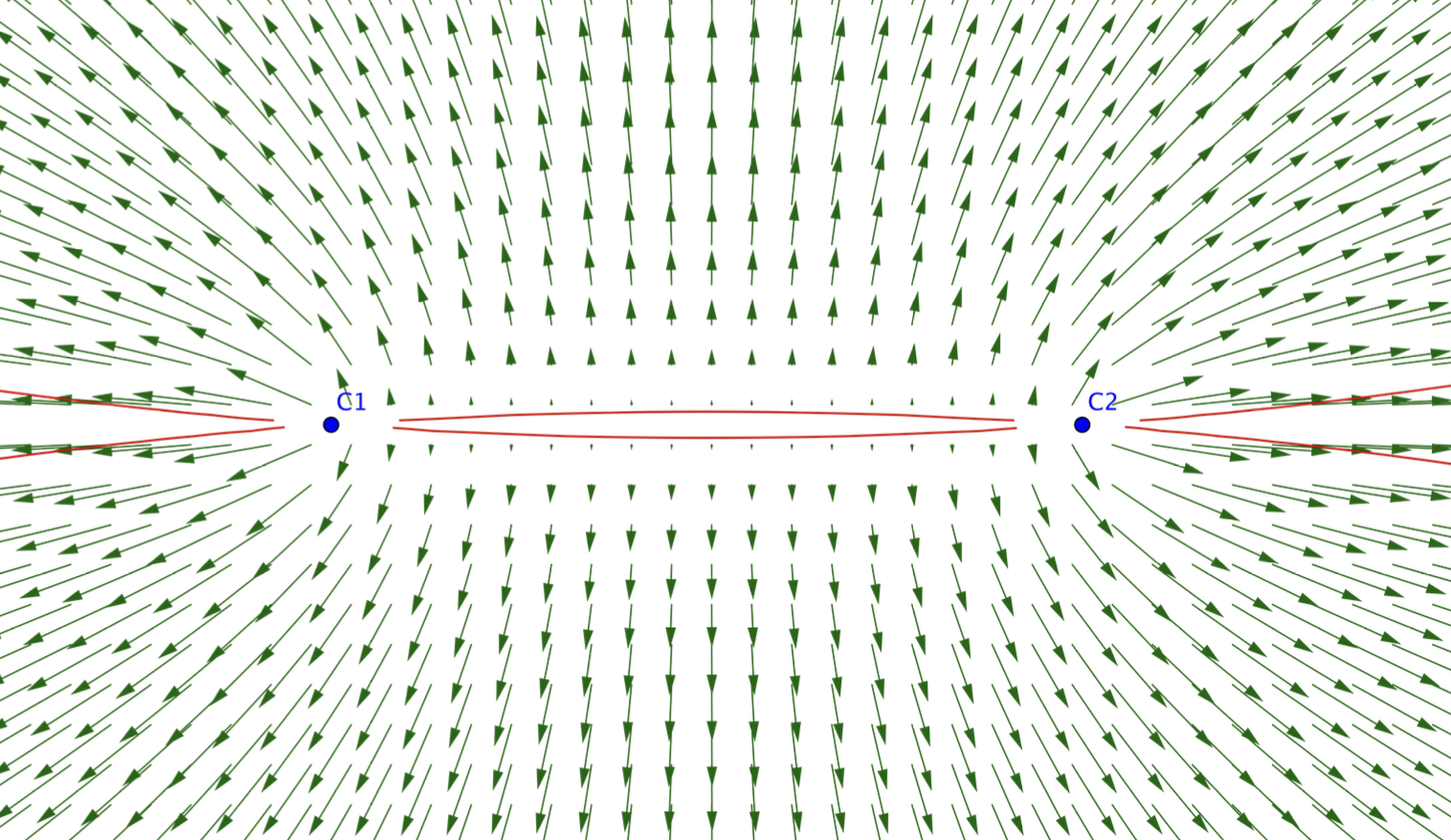}\caption{$\nabla\chi$ and $S(\alpha)$}
\end{figure}

From the other hand, because the flow $\left\{ u(t),\ t\geq0\right\} $
is relatively compact in $\dot{H}^{1}$, we have in a classical way
(see Lemma \lemref{relcompfam})
\[
\sup_{t\geq0}\int_{(\Omega\cap B(0,A))\backslash S(\alpha)}|\nabla u|^{2}(t,x)dx=\epsilon(|(\Omega\cap B(0,A))\backslash S(\alpha)|),
\]
where $\epsilon(h)\rightarrow0$ as $h\rightarrow0$ and is decreasing. And in particular
\begin{equation}
\frac{1}{T}\int_{0}^{T}\int_{(\Omega\cap B(0,A))\backslash S(\alpha(T))}|\nabla u|^{2}(t,x)dx=\epsilon(|(\Omega\cap B(0,A))\backslash S(\alpha)|).\label{eq:relcomp}
\end{equation}
Denote $m(\alpha):=|(\Omega\cap B(0,A))\backslash S(\alpha,c)|$. By the dominated convergence theorem, we have $m(\alpha)\underset{\alpha\to 0}{\longrightarrow}|(\Omega\cap B(0,A))\cap\left\{ b=0\right\}|$. We check that $b(x)=0$ implies $a(x)=\frac{x+c}{|x+c|}\cdot \frac{x-c}{|x-c|}=1$, which implies that $x-c$ is colinear to $x+c$, that is $x\in \R e_1$. In particular, $|(\Omega\cap B(0,A))\cap\left\{ b=0\right\}|=0$ and $m(\alpha)\underset{\alpha\to 0}{\longrightarrow}0$.

Collecting (\ref{eq:decayl6}), (\ref{eq:decaygrad}) and (\ref{eq:relcomp}) we get
\[
\frac{1}{T}\int_{0}^{T}\int_{\Omega\cap B(0,A)}|\nabla u(x,t)|^{2}+|u(x,t)|^{6}\ dxdt\lesssim \frac{E}{T}+\frac{E}{\alpha T}+\epsilon(m(\alpha)).
\]
We take (say) $\alpha:=T^{-1/2}$. Then all the
right hand terms go to zero as $T$ goes to infinity, and $u$ scatters
in $\dot{H}^{1} \times L^2$ by Lemma \lemref{contrenerg}.
\end{proof}

\section{Geometric facts} \label{sec:verifA}

The purpose of this section is to verify that the exterior of two strictly convex obstacles satisfies the geometrical Assumptions \ref{ass:nonreco} and \ref{ass:weaktrap}.

\subsection{Non reconcentration}

\begin{lem} \label{lem:rays_noreco}
Let $\Theta_{1}, \Theta_2$ be two smooth, strictly convex subsets of $\text{\ensuremath{\mathbb{R}}}^{n}$
and $\Omega:=\mathbb{R}^{n}\backslash(\Theta_{1}\cup\Theta_{2})$. Then, $\Omega$ satisfies Assumption \ref{ass:nonreco}.
\end{lem}
\begin{proof}
We call a finite sequence of elements of $\left\{ 1,2\right\} $,
$J=(i_{1},\dots,i_{n})$ with $i_{k}\in\left\{ 1,2\right\} $ and
$i_{k+1}\neq i_{k}$, a story of reflections. Let $\Phi_{J,t}(x,\xi)$
be, \textit{if it exists}, the point of $\bar{\Omega}$ obtained following
the story of reflections $J$ for a time $t$, starting from the point
$x$ and direction $\xi$ : the first reflection occurs on $\Theta_{i_{1}}$,
the second one on $\Theta_{i_{2}}$, and so on. 

We will show that, for any story of reflections $J$, any $x\in\bar{\Omega}$
and any $\xi_{1},\xi_{2}\in\mathbb{R}^{n}$, as soon as $\Phi_{J,t}(x,\xi_{1})$
and $\Phi_{J,t}(x,\xi_{2})$ exist
\begin{equation}
\Phi_{J,t}(x,\xi_{1})=\Phi_{J,t}(x,\xi_{2})\implies\xi_{1}=\xi_{2},\label{eq:nonrecfin}
\end{equation}
and the lemma will follow: indeed, if (\ref{eq:nonrecfin}) holds,
there is at most one direction permitting to go from $x$ to $x_{0}$
in time $t$ following a given story of reflections, and therefore,
because the size of stories connecting $x$ and $x_{0}$ in time $t$
is bounded by $c|t|$, $\left\{ \xi\in\mathbb{R}^{n}\ \text{s.t.}\ \Phi_{t}(x,\frac{\xi}{|\xi|})=x_{0}\right\} $
is a finite collection of half lines, thus a set of zero measure.

To this purpose, let $J=(i_{1},\cdots,i_{n})$ be a story of reflections,
$x\in\bar{\Omega}$, $\xi_{1},\xi_{2}\in\mathbb{S}^{n-1}$, suppose
that $\Phi_{J,t}(x,\xi_{1,2})$ exists and that
\begin{equation}
\Phi_{J,t}(x,\xi_{1})=\Phi_{J,t}(x,\xi_{2}).\label{eq:nonrecsupp}
\end{equation}
We denote, for $j=1,2$
\begin{align*}
\xi_{j}^{(0)} & =\xi_{j},\\
x_{j}^{(0)} & =x,
\end{align*}
for $0\leq k\leq n-1$,
\begin{align}
x_{j}^{(k+1)} & =x_{j}^{(k)}+t_{j}^{(k)}\xi_{j}^{(k)}\in\partial\Theta_{i_{k}},\label{eq:nonrecx}\\
\xi_{j}^{(k+1)} & =\xi_{j}^{(k)}-2(\xi_{j}^{(k)}\cdot \vec n(x_{j}^{(k)}))\vec n(x_{j}^{(k)}),\label{eq:nonrecxi}
\end{align}
the points and directions of the $k$-th reflection so that $\xi_{j}^{(k+1)}\cdot \vec n =\xi_{j}^{(k)}\cdot \vec n$. Note that , $\xi_{j}^{(k+1)} =\xi_{j}^{(k)}$ for tangential points. We define $t_{j}^{(n)}$
by
\[
t_{j}^{(n)}=t-(t_{j}^{(1)}+\cdots+t_{j}^{(n-1)}),
\]
in such a way that
\[
\Phi_{J,t}(x,\xi_{j})=x_{j}^{(n)}+t_{j}^{(n)}\xi_{j}^{(n)}.
\]
For convenience, we will denote $\vec n_{j}^{(k)}=\vec n(x_{j}^{(k)})$. On
the one hand, because of (\ref{eq:nonrecsupp}), we have
\begin{equation}
(x_{1}^{(n)}-x_{2}^{(n)})\cdot(\xi_{1}^{(n)}-\xi_{2}^{(n)})=(-t_{1}^{(n)}\xi_{1}^{(n)}+t_{2}^{(n)}\xi_{2}^{(n)})\cdot(\xi_{1}^{(n)}-\xi_{2}^{(n)})
=-(t_{1}^{(n)}+t_{2}^{(n)})(1-\xi_{1}^{(n)}\cdot\xi_{2}^{(n)})\leq0.\label{eq:nonrec0}
\end{equation}
by unitarity and Cauchy-Schwarz inequality.
On the other hand, note that, using (\ref{eq:nonrecxi})
\begin{multline}
(x_{1}^{(n)}-x_{2}^{(n)})\cdot(\xi_{1}^{(n)}-\xi_{2}^{(n)})=(x_{1}^{(n)}-x_{2}^{(n)})\cdot(\xi_{1}^{(n-1)}-\xi_{2}^{(n-1)})\\
-2(x_{1}^{(n)}-x_{2}^{(n)})\cdot\left((\xi_{1}^{(n-1)}\cdot \vec n_{1}^{(n-1)})\vec n_{1}^{(n-1)}-(\xi_{2}^{(n-1)}\cdot \vec n_{2}^{(n-1)})\vec n_{2}^{(n-1)}\right).\label{eq:nonrec1}
\end{multline}
But, for $y,z$ belonging to the boundary of a convex body $\mathcal{C}$,
we always have
\[
(y-z)\cdot \vec n(y)\leq0
\]
where $\vec n$ is oriented toward the interior of $\mathcal{C}$. In addition, we also have $\xi_{j}^{(n-1)}\cdot \vec n_{j}^{(n-1)} \geq 0$ because, by definition, $\xi_j^{(n-1)}$ points toward $\Theta_{j_{n}}$. Thus,
because $x_{1}^{(n)}$ and $x_{2}^{(n)}$ belong to the boundary of
the same obstacle, the second term in (\ref{eq:nonrec1}) is non-negative
and we get
\begin{equation}
(x_{1}^{(n)}-x_{2}^{(n)})\cdot(\xi_{1}^{(n)}-\xi_{2}^{(n)})\geq(x_{1}^{(n)}-x_{2}^{(n)})\cdot(\xi_{1}^{(n-1)}-\xi_{2}^{(n-1)}).\label{eq:nonrec2}
\end{equation}
Moreover, by (\ref{eq:nonrecx})
\begin{multline*}
(x_{1}^{(n)}-x_{2}^{(n)})\cdot(\xi_{1}^{(n-1)}-\xi_{2}^{(n-1)})=(x_{1}^{(n-1)}-x_{2}^{(n-1)})\cdot(\xi_{1}^{(n-1)}-\xi_{2}^{(n-1)})\\
+(t_{1}^{(n-1)}\xi_{1}^{(n-1)}-t_{2}^{(n-1)}\xi_{2}^{(n-1)})\cdot(\xi_{1}^{(n-1)}-\xi_{2}^{(n-1)})\\
=(x_{1}^{(n-1)}-x_{2}^{(n-1)})\cdot(\xi_{1}^{(n-1)}-\xi_{2}^{(n-1)})+(t_{1}^{(n-1)}+t_{2}^{(n-1)})(1-\xi_{1}^{(n-1)}\cdot\xi_{2}^{(n-1)}).
\end{multline*}
Therefore, combining this identity with (\ref{eq:nonrec2})
\begin{equation*}
(x_{1}^{(n)}-x_{2}^{(n)})\cdot(\xi_{1}^{(n)}-\xi_{2}^{(n)})\geq(x_{1}^{(n-1)}-x_{2}^{(n-1)})\cdot(\xi_{1}^{(n-1)}-\xi_{2}^{(n-1)})
+(t_{1}^{(n-1)}+t_{2}^{(n-1)})(1-\xi_{1}^{(n-1)}\cdot\xi_{2}^{(n-1)}).
\end{equation*}
And by induction we get
\[
(x_{1}^{(n)}-x_{2}^{(n)})\cdot(\xi_{1}^{(n)}-\xi_{2}^{(n)})\geq\sum_{k=0}^{n-1}(t_{1}^{(k)}+t_{2}^{(k)})(1-\xi_{1}^{(k)}\cdot\xi_{2}^{(k)}).
\]
Therefore, by (\ref{eq:nonrec0}), if $x\notin \partial \Omega$, as $t_{1}^{(0)}+t_{2}^{(0)} >0$, we conclude that  $\xi_1 = \xi_2$. In the case were $x \in \partial \Omega$, we obtain similarly $\xi^{(1)}_{1} = \xi^{(1)}_{2}$, and hence $\xi_1 = \xi_2$.

\end{proof}

\subsection{Weak trapping}

\begin{lem} \label{lem:finite_trap} 
Let $\Theta_1, \Theta_2$ be two smooth, strictly convex subsets of $\text{\ensuremath{\mathbb{R}}}^{n}$
and $\Omega:=\mathbb{R}^{n}\backslash(\Theta_{1}\cup\Theta_{2})$.
Then, every point on $\overline \Omega$ is only on the way of a finite number of
trajectories that are trapped either in the future or in the past. In particular, $\Omega$ satisfies Assumption \ref{ass:weaktrap}.
\end{lem}

\begin{proof}
Let $x \in \bar \Omega$ and $\xi_1, \xi_2 \in \mathbb{S}^{n-1}$. We will see that 
if both $(x, \xi_1)$ and $(x, \xi_2)$ are trapped in the future, then either $\xi_1 = \xi_2$ or the rays starting from $(x, \xi_1)$ and $(x, \xi_2)$ first reflect on different obstacles: this will show that there is at most two different trajectories from $x$ that are trapped in the future. 

Assume that $(x, \xi_1)$ and $(x, \xi_2)$ are trapped in the future and both first reflect on the same obstacle, and let us show that $\xi_1 = \xi_2$. By assumption, 
$$
\bigcap_{T\geq 0} \overline{\{\varphi_t(x,\xi_{1}), \; t\geq T \} } \; \text{ and } \;
\bigcap_{T\geq 0} \overline{\{\varphi_t(x,\xi_{2}), \; t\geq T \} } 
$$
are compact sets of $S^b \overline{\Omega}$. On the other hand, both of these sets are invariant by the flow of geometrical optics (in the future and in the past). The only such set is the trapped set $\mathcal R \times \{-e, e\}$, where $e\in \mathbb{S}^{n-1}$ is parallel to the trapped ray $\mathcal R$. Indeed, if $\xi$ is not colinear to $e$, then $|x(t) \cdot \xi(t)| \to +\infty $ 
for 
either $t \to + \infty$ or $t\to - \infty$,
where $\varphi_t(x, \xi) =: (x(t), \xi(t))$.  Hence 
$$
\bigcap_{T\geq 0} \overline{\{\varphi_t(x,\xi_{1}), \; t\geq T \} }
= \bigcap_{T\geq 0} \overline{\{\varphi_t(x,\xi_{2}), \; t\geq T \} } = \mathcal R \times \{-e, e\}.
$$
We now adopt the notations of the proof of Lemma \ref{lem:rays_noreco}. From any subsequence of $\xi_{j}^{(n)}$, we can extract a converging subsequence. Because of the above, the limit can be either $-e$ or $+e$. In addition, because the rays starting from $(x,\xi_{1})$ and $(x,\xi_{2})$ follow the same story of reflections, $\xi^{(n)}_1 \cdot e$ and  $\xi^{(n)}_2 \cdot e$ have the same sign for any $n$. Therefore, the difference $\xi^{(n)}_1 - \xi^{(n)}_2$ goes to zero. Hence, for $\epsilon > 0$ arbitrary, we can fix $n \geq 1$ large enough so that
$$
\big| (x_{1}^{(n)}-x_{2}^{(n)})\cdot(\xi_{1}^{(n)}-\xi_{2}^{(n)}) \big| \leq \epsilon.
$$
On the other hand, we saw in the proof of Lemma \ref{lem:rays_noreco} that
$$
(x_{1}^{(n)}-x_{2}^{(n)})\cdot(\xi_{1}^{(n)}-\xi_{2}^{(n)})\geq\sum_{k=0}^{n-1}(t_{1}^{(k)}+t_{2}^{(k)})(1-\xi_{1}^{(k)}\cdot\xi_{2}^{(k)}).
$$
In particular (recall that $\xi^{(0)}_{1,2} = \xi_{1,2}$),
$$
0 \leq  (t_{1}^{(0)}+t_{2}^{(0)})(1-\xi_{1} \cdot\xi_{2}) \leq (x_{1}^{(n)}-x_{2}^{(n)})\cdot(\xi_{1}^{(n)}-\xi_{2}^{(n)}) \leq \epsilon.
$$
Therefore, as $\epsilon>0$ is arbitrary, if $x\notin \partial \Omega$ and hence $t_{1}^{(0)}+t_{2}^{(0)} >0$, we conclude that  $\xi_1 = \xi_2$. In the case were $x \in \partial \Omega$, we obtain similarly $\xi^{(1)}_{1} = \xi^{(1)}_{2}$, and hence $\xi_1 = \xi_2$.

The same holds for trajectories trapped in the past, and the Lemma follows.
\end{proof}

\appendix

\section{Functional analysis}
We used the following property, which is certainly classical, see \cite[Lemma 17.2]{TartarBookinterp} for a closely related statement.
\begin{lem} \label{lem:H1moins_point}
$\dot H_0^1(\mathbb R^3 \backslash \{ 0 \}) = \dot H^1(\mathbb R^3)$.
\end{lem}
\begin{proof}
It suffices to show that any element of  $C^\infty_c(\mathbb R^3)$ can be approximated in $\dot H^1(\mathbb R^3)$ seminorm by elements of $ C^\infty_c(\mathbb R^3\backslash \{ 0 \})$.
In order to do so, let $f\in C^\infty_c(\mathbb R^3)$, and $\chi \in C^\infty_c(\mathbb R^3)$ be so that $\operatorname{supp}\chi\subset B(0,2)$ and $\chi = 1$ in $B(0,1)$. For any $\epsilon>0$, we set
$$
f_\epsilon(x) := f(x)(1-\chi(\frac{x}{\epsilon})).
$$
Observe that $f_\epsilon \in C^\infty_c(\mathbb R^3\backslash \{ 0 \})$, and
\begin{align*}
\Vert f - f_\epsilon \Vert_{\dot H^1} &\leq \Vert \nabla f \chi(\frac{x}{\epsilon}) \Vert_{L^2} + \epsilon^{-1}\Vert f \nabla \chi(\frac{x}{\epsilon}) \Vert_{L^2} \\
&\lesssim \Vert \nabla f \Vert_{L^2(B(0,2\epsilon))} + \epsilon^{-1} \Vert f \Vert_{L^2(B(0,2\epsilon))} \lesssim (1 + \epsilon^{-1})|B(0,2\epsilon)|^{1/2} \lesssim \epsilon^{1/2}.
\end{align*}
This finishes the proof.
\end{proof}

\begin{lem}
\label{lem:relcompfam}Let $\{u(t),\ t\geq0\}$ be a relatively compact
family of $\dot{H}^{1}$, and $S_{k}$ be a family of subsets of $\Omega$
such that the Lebesgue mesure of $S_{k}$ goes to zero as $k$ goes
to infinity. Then we have
\[
\sup_{t\geq0}\int_{S_{k}}|\nabla u(t,x)|^{2}dx\underset{k\to +\infty}{\longrightarrow}0.
\]

\end{lem}
\begin{proof}
If it is not the case, there exists a subsequence $S_{n_{k}}$, a
sequence of times $t_{k}$ and $\epsilon>0$ such that
\[
\forall k,\ \int_{S_{n_{k}}}|\nabla u(t_{k})|^{2}dx\geq\epsilon.
\]
Because $\{u(t)\ t\geq0\}$ is relatively compact in $\dot{H}^{1}$,
we can suppose, up to extract a subsequence, that,
\[
u(t_{k})\underset{k\to +\infty}{\longrightarrow} u^{\star}\in\dot{H}^{1}\text{ in }\dot{H}^{1}.
\]
We have
\[
\int_{S_{n_{k}}}|\nabla u(t_{k})|^{2}dx\leq 2\int_{S_{n_{k}}}|\nabla u^{\star}|^{2}dx+ 2\int_{\Omega}|\nabla(u^{\star}-u(t_{k}))|^{2}dx.
\]
Both terms are going to zero as $k$ goes to infinity and we obtain
a contradiction.
\end{proof}

\section{Reparametrization of the flow}
\begin{lem}\label{lmrepar}
Let $a<b$ and $X$ a Hausdorff topological space. Assume that for any $s\in \R$, there exists $\phi_{s}:[a,b] \times X\longrightarrow [a,b] \times X$  so that $(s,\tau,\rho)\in \R\times [a,b]\times X\mapsto \phi_{s}(\tau,\rho)$ is continuous and for any $c<d$, its restriction to $[c,d]\times [a,b]\times X$ is proper\footnote{that is, $f^{-1}(K)$ is compact for any compact $K$}. We assume moreover that $\phi$ leaves the first variable $\tau$ invariant. Let $\mu$ a Radon measure on $[a,b] \times X$ so that $\phi_{s}^{*}\mu=\mu$ for any $s\in \R$. Let $g:[a,b]\mapsto \R$ be a continuous function. Denote $\widetilde{\phi}_{s}(\tau,\rho)=\phi_{sg(\tau)}(\tau,\rho)$. Then, $\widetilde{\phi}_{s}^{*}\mu=\mu$ for any $s\in \R$.
\end{lem}
\begin{proof}
Let $m\in C^{0}_{c}([a,b] \times X)$, $s_{0}\in \R$ and $\varepsilon>0$. The application $ [a,b]^{2} \times X\ni(\tau',\tau,\rho)\mapsto m(\phi_{s_{0}g(\tau')}(\tau,\rho))\in [a,b] \times X$ is continuous. We claim that it is also compactly supported. Indeed, denote $K=\operatorname{supp}(m)$ which is compact. $g([a,b])$ is compact, so there exists $c<d$ so that $s_{0}g([a,b])\subset [c,d]$. Denote $K'=f^{-1}(K)$ where $f$ is defined from $[c,d]\times [a,b]\times X$ to $[a,b]\times X$ by $f(s,\tau,\rho)= \phi_{s}(\tau,\rho)$. So by assumption, $K'$ is compact in $[c,d]\times [a,b]\times X$. If $\Pi_{X}$ is the projection on $X$, continuous from $[c,d]\times [a,b]\times X$ to $X$, then $K_{X}':=\Pi_{X}(K')$ is compact. In particular, $\phi_{s_{0}g(\tau')}(\tau,\rho)\in K$ implies $(s_{0}g(\tau'),\tau,\rho)\in K'\subset [c,d]\times [a,b]\times K_{X}'$. The support of the application $[a,b]^{2} \times X\ni(\tau',\tau,\rho)\mapsto m(\phi_{s_{0}g(\tau')}(\tau,\rho))\in [a,b] \times X$ is therefore compact,proving the claim. In particular, there exists $\delta>0$ so that $|\tau_{1}-\tau_{2}|\leq \delta$ implies $\left|m(\phi_{s_{0}g(\tau_{1})}(\tau,\rho))-m(\phi_{s_{0}g(\tau_{2})}(\tau,\rho)) \right|\leq \varepsilon$ for all $(\tau,\rho)\in [a,b]\times X$. Let now $\chi_{i}\in C([a,b],[0,1])$ for $i\in I$, $I$ finite, so that $\operatorname{supp}(\chi_{i})\subset [\tau_{i}-\delta/2, \tau_{i}+\delta/2]$ for some $\tau_{i}\in [a,b]$ and $\sum_{i}\chi_{i}(\tau)=1$ for $\tau\in [a,b]$. We write
\begin{align*}
\left<\widetilde{\phi}_{s_{0}}^{*}\mu,m\right>=\left<\mu,m\circ \widetilde{\phi}_{s_{0}}\right>=\sum_{i\in I}\left<\mu,(\chi_{i}m)\circ \widetilde{\phi}_{s_{0}}\right>.
\end{align*}
We define 
$$
r_{i}(\tau,\rho):=m(\phi_{s_{0}g(\tau)}(\tau,\rho))-m(\phi_{s_{0}g(\tau_{i})}(\tau,\rho)),$$ 
so that 
$$
\left[(\chi_{i}m)\circ \widetilde{\phi}_{s_{0}}\right](\tau,\rho)=(\chi_{i}m)\circ \phi_{s_{0}g(\tau_{i})}(\tau,\rho)+\chi_{i}(\tau)r_{i}(\tau,\rho),
$$
where we have used that $\phi_{s}$ leaves $\tau$ invariant so that $(\chi_{i}\circ \widetilde{\phi}_{s_{0}})(\tau)=\chi_{i}(\tau)$. 
Using the invariance of $\mu$, we have $\left<\mu,(\chi_{i}m)\circ \phi_{s_{0}g(\tau_{i})}(\tau,\rho)\right>=\left<\mu,\chi_{i}m\right>$, so that after summing up $$
\left<\widetilde{\phi}_{s_{0}}^{*}\mu,m\right>=\left<\mu,m\right>+\sum_{i\in I}\left<\mu,\chi_{i}r_{i}\right>.$$
Moreover, since $|\tau_{i}-\tau|\leq \delta$ on the support of $\chi_{i}$, we have the estimate 
$$
\sum_{i\in I}|\chi_{i}r_{i}(\tau,\rho)|\leq \sum_{i\in I}\chi_{i}(\tau)\left|m(\phi_{s_{0}g(\tau)}(\tau,\rho))-m(\phi_{s_{0}g(\tau_{i})}(\tau,\rho)) \right|\leq \varepsilon \sum_{i\in I}\chi_{i}=\varepsilon.$$
In addition, with the previous notation, $\operatorname{supp}(r_{i})\subset [a,b]\times K_{X}'$ which is compact and therefore, $\mu([a,b]\times K_{X}')$ is finite and independent on $\varepsilon$. In particular, we have $\left|\sum_{i\in I}\left<\mu,\chi_{i}r_{i}\right>\right|\leq \varepsilon \mu([a,b]\times K_{X}')$. It gives the result since $\varepsilon$ is arbitrary and the other part is independent on $\varepsilon$.
\end{proof}
\section{A geometric lemma}
\begin{lem}
\label{lm:volC}
There exists $\epsilon_0>0$ and $D>0$ so that for every $0<\epsilon<\epsilon_0$ and $(x_0,r,R,t)\in ( \R^3,\R_+,\R_+,\R_+)$ so that $r\leq \epsilon |x_0|$, $R\leq \epsilon |x_0|$, $R\leq \epsilon_0 t$, the set 
   $$    C(x_0,r,R,t) := B(x_0,r) \cap \big\{x\in \R^3\textnormal{ s.t. } |x| \in [t-R, t+R] \big\}.$$
satisfies
   \begin{equation} \label{eq:taille_C}
|C(x_0,r,R,t)|\leq D t^2 R\epsilon.
\end{equation}
\end{lem}
 \begin{proof}
      Let $x \in C(x_0,r,t)$. We have
$$
\frac{x}{|x|} \cdot \frac{x_0}{|x_0|} =  \frac{1}{|x|}x_0\cdot  \frac{x_0}{|x_0|} + \frac{1}{|x|}(x-x_0) \cdot  \frac{x_0}{|x_0|} ,
$$
from which
$$
\left| \frac{x}{|x|} \cdot \frac{x_0}{|x_0|} \right| \geq  
\frac{1}{t+R} |x_0| - \frac{1}{t-R}r. 
$$
Therefore, for $t \in [|x_0|-2r,|x_0|+2r]$ 
$$
\left| \frac{x}{|x|} \cdot \frac{x_0}{|x_0|} \right| \geq 
\frac{1}{|x_0|+2r+R} |x_0| - \frac{r}{|x_0|-2r-R}=1- \frac{2r+R}{|x_0|+r+R}  - \frac{r}{|x_0|-2r-R},
$$
and using $r\leq \epsilon |x_0|$, $R\leq \epsilon |x_0|$
$$
\left| \frac{x}{|x|} \cdot \frac{x_0}{|x_0|} \right| \geq 1-\frac{3\epsilon|x_0|}{|x_0|}- \frac{\epsilon|x_0|}{|x_0|(1-3\epsilon)}\geq 1-5\epsilon,
$$
where we have used $\epsilon <1/6$. Assuming $\frac{x_0}{|x_0|}=(1,0,0)$ by rotation, we  have obtained that $x \in C$ implies 
$$
\left| \frac{x_1}{|x|}\right|^2 \geq (1-5\epsilon)^2\geq 1-10 \epsilon (1-\epsilon /2)\geq 1-11 \epsilon ,
$$
for $\epsilon$ small enough. Therefore, 
\begin{align*}
|C|&\leq  4\pi \int_{t-R}^{t+R}\rho^2  \int_{s_1=\sqrt{1-11\epsilon}}^1 \sqrt{1-s_1^2}ds d\rho =4\pi \left[(t+R)^3-(t-R)^3\right]\int_{0}^{\arccos(\sqrt{1-11\epsilon}) } \sin (y)^2dy\\
&\leq Dt^3  \left[(1+\frac{R}{t})^3-(1-\frac{R}{t})^3\right]\epsilon\leq DRt^2\epsilon
\end{align*}
where we have taken $0<\epsilon \leq \epsilon_0$ with $\epsilon_0>0$ small enough and used that $\frac{R}{t}\leq \epsilon_0 \ll 1$.
 \end{proof}

\bibliographystyle{amsalpha}
\bibliography{refs}

\end{document}